\documentclass[11pt]{amsart}

\usepackage{graphicx}
\usepackage{psfrag}
\usepackage{perpage}
\usepackage{url}
\usepackage{scalerel}
\usepackage{color}
\usepackage{mathrsfs}
\usepackage{dsfont} 
\usepackage{mathdots}
\usepackage{tikz}
\usepackage{young}
\usepackage{amsfonts,color}
\usepackage{amsmath, amssymb, amsthm, amscd}
\usepackage{bbm}

\usepackage{bbm}
\usepackage{enumitem}
\usepackage{booktabs}
\usepackage{varioref}

\usepackage{epsfig}
\usepackage{caption}

\usepackage[utf8]{inputenc}
\usepackage[T1]{fontenc}
\usepackage{lmodern}
\usepackage{microtype}
\usepackage[a4paper,scale={0.72,0.74},marginratio={1:1},footskip=10mm,headsep=10mm]{geometry}
\usepackage{hyperref}

\makeatletter
\def\@secnumfont{\bfseries\scshape}

\def\section{\@startsection{section}{1}%
  \z@{.7\linespacing\@plus\linespacing}{.5\linespacing}%
  {\normalfont\large\bfseries\scshape\centering}}

\def\subsection{\@startsection{subsection}{2}%
  \z@{.5\linespacing\@plus.7\linespacing}{-.5em}%
  {\normalfont\bfseries\scshape}}

\def\subsubsection{\@startsection{subsubsection}{3}%
  \z@{.5\linespacing\@plus.7\linespacing}{-.5em}%
  {\normalfont\scshape}}

\def\specialsection{\@startsection{section}{1}%
  \z@{\linespacing\@plus\linespacing}{.5\linespacing}%
  {\normalfont\centering\large\bfseries\scshape}}
\makeatother

\makeatletter

\renewenvironment{proof}[1][\proofname]{\par
\pushQED{\qed}%
\normalfont \topsep4\p@\@plus4\p@\relax
\trivlist
\item[\hskip\labelsep
\bfseries
#1\@addpunct{.}]\ignorespaces
}
{
\popQED\endtrivlist\@endpefalse
}
\makeatother

\makeatletter
\newcommand \Dotfill {\leavevmode \leaders \hb@xt@ 6pt{\hss .\hss }\hfill \kern \z@}
\makeatother

\makeatletter
\def\@tocline#1#2#3#4#5#6#7{\relax
  \ifnum #1>\c@tocdepth 
  \else
    \par \addpenalty\@secpenalty\addvspace{#2}%
    \begingroup \hyphenpenalty\@M
    \@ifempty{#4}{%
      \@tempdima\csname r@tocindent\number#1\endcsname\relax
    }{%
      \@tempdima#4\relax
    }%
    \parindent\z@ \leftskip#3\relax \advance\leftskip\@tempdima\relax
    \rightskip\@pnumwidth plus4em \parfillskip-\@pnumwidth
    #5\leavevmode\hskip-\@tempdima
      \ifcase #1
       \or\or \hskip 1.65em \or \hskip 3.3em \else \hskip 4.95em \fi%
      #6\nobreak\relax
    \Dotfill
    \hbox to\@pnumwidth{\@tocpagenum{#7}}\par
    \nobreak
    \endgroup
  \fi}
\makeatother

\makeatletter
\def\l@section{\@tocline{1}{0pt}{1pc}{}{\scshape}}
\renewcommand{\tocsection}[3]{%
\indentlabel{\@ifnotempty{#2}{\ignorespaces#1 #2.\hskip 0.7em}}#3}
\def\l@subsection{\@tocline{2}{0pt}{1pc}{5pc}{}}

\def\l@subsubsection{\@tocline{3}{0pt}{1pc}{7pc}{}}

\makeatother

\numberwithin{equation}{section}

\newtheoremstyle{mytheorem}{.7\linespacing\@plus.3\linespacing}{.7\linespacing\@plus.3\linespacing}%
     {\itshape}
     {}
     {\bfseries}
     {. }
     {0.3ex}
     {\thmname{{\bfseries #1}}\thmnumber{ {\bfseries #2}}\thmnote{ (#3)}}  

\theoremstyle{mytheorem}

\newtheorem{theorem}{Theorem}[section]
\newtheorem{lemma}[theorem]{Lemma}
\newtheorem{proposition}[theorem]{Proposition}
\newtheorem{corollary}[theorem]{Corollary}
\newtheorem{remark}[theorem]{Remark}
\newtheorem{definition}[theorem]{Definition}

\newcommand{\cB}{{\ensuremath{\mathcal B}} }

\newcommand{\cL}{{\ensuremath{\mathcal L}} }

\renewcommand{\tilde}{\widetilde}          
\DeclareMathSymbol{\leqslant}{\mathalpha}{AMSa}{"36} 
\DeclareMathSymbol{\geqslant}{\mathalpha}{AMSa}{"3E} 
\DeclareMathSymbol{\eset}{\mathalpha}{AMSb}{"3F}     

\newcommand{\PP}{\ensuremath{\mathbb{P}}}
\newcommand{\EE}{\ensuremath{\mathbb{E}}}

\renewcommand{\epsilon}{\varepsilon}
\renewcommand{\theta}{\vartheta}
\renewcommand{\rho}{\varrho}

\newenvironment{myenumerate}{%
\renewcommand{\theenumi}{\arabic{enumi}}%
\renewcommand{\labelenumi}{{\rm(\theenumi)}}%
\begin{list}{\labelenumi}
	{%
	\setlength{\itemsep}{0.4em}%
	\setlength{\topsep}{0.5em}%
	\setlength\leftmargin{2.45em}%
	\setlength\labelwidth{2.05em}%
	\setlength{\labelsep}{0.4em}%
	\usecounter{enumi}%
	}%
	}%
{\end{list}
}

{\end{list}
}

{\end{list}
}

\renewenvironment{enumerate}{
\begin{myenumerate}}%
{\end{myenumerate}}

\newenvironment{myitemize}{%
\begin{list}{$\bullet$}%
 	{%
	\setlength{\itemsep}{0.4em}%
	\setlength{\topsep}{0.5em}%
	\setlength\leftmargin{2.45em}%
	\setlength\labelwidth{2.05em}%
	\setlength{\labelsep}{0.4em}%
	}%
	}%
{\end{list}}

\renewenvironment{itemize}{
\begin{myitemize}}%
{\end{myitemize}}

\MakePerPage[2]{footnote} 

\newcommand{\cpl}{\mathrm{cpl}}

\usepackage{bbold} 

\newcommand{\free}{\textup{free}}

\newcommand{\overbar}[1]{\mkern 1.5mu\overline{\mkern-1.5mu#1\mkern-1.5mu}\mkern 1.5mu}

\newcommand{\rectangle}{{%
  \ooalign{$\sqsubset\mkern3mu$\cr$\mkern5mu\vphantom{}^\bullet$\cr$\mkern3mu\sqsupset$\cr}%
}}

\newcommand{\R}{\mathbb{R}}

\newcommand{\Z}{\mathbb{Z}}
\newcommand{\N}{\mathbb{N}}

\newcommand{\Ham}{\ensuremath{\mathbf{H}}}
\newcommand{\Hamd}{\dot{\Ham}}
\newcommand{\Hrw}{\textup{H}^{\textup{RW}}}
\newcommand{\HrwN}{\textup{H}^{\textup{RW},N}}

\makeatletter

\let\c@table\c@figure
\makeatother

\begin{document}
\title[Tightness of discrete Gibbsian line ensembles]{Tightness of discrete Gibbsian line ensembles with exponential interaction Hamiltonians}

\author[Xuan Wu]{Xuan Wu}
\address{Xuan Wu, Columbia University,
Department of Mathematics,
2990 Broadway,
New York, NY 10027, USA}
\email{xuanw@math.columbia.edu}

\begin{abstract}
In this paper we introduce a framework to prove tightness of a sequence of discrete Gibbsian line ensembles $\cL^N = \{\cL_k^1(u), \cL_k^2(u),\dots  \}$, which is a countable collection of random curves. The sequence of discrete line ensembles $\cL^N$ we consider enjoys a resampling invariance property, which we call $(\Ham^N,\HrwN)$-Gibbs property. We also assume that $\cL^N$ satisfies technical assumptions A1-A4 on $(\Ham^N, \HrwN)$ and the assumption that the lowest labeled curve with a parabolic shift, $\cL_1^N(u) +  {u^2}/{2}$, converges weakly to a stationary process in the topology of uniform convergence on compact sets. Under these assumptions, we prove our main result Theorem~\ref{theorem:main} that $\cL^N$ is tight as a sequence of line ensembles and that the $\Ham$-Brownian Gibbs property holds for all subsequential limit line ensembles with $\Ham(x)= e^x$. 

As an application of Theorem~\ref{theorem:main}, under the weak noise scaling, we show that the scaled log-gamma line ensemble $\overbar{\mathcal{L}}^N$ is tight, which is a sequence of discrete line ensembles associated with the log-gamma polymer model via the geometric RSK correspondence. The $\Ham$-Brownian Gibbs property (with $\Ham(x) = e^x$) of its subsequential limits also follows. 
\end{abstract}
\maketitle

\section{Introduction}
There is a large class of stochastic integrable models from random matrix theory, last passage percolation, and more generally from the Kardar-Parisi-Zhang (KPZ) universality class that naturally carry the structure of random paths with some Gibbsian resampling invariance. One particularly interesting and central example is the Airy line ensemble \cite{CH14} $ \mathcal{A} = \{\mathcal{A}_1 > \mathcal{A}_2 > \cdots \}$, a collection of non-intersecting  continuous random curves indexed by $\mathbb{N}$, see Table~\ref{Airy} for some basic properties of the Airy line ensemble $\mathcal{A}$. 
\begin{table}[h!]
     \begin{center}
     \begin{tabular}{c p{8cm}}    
    \cmidrule(r){1-2}
     \raisebox{-\totalheight}{\includegraphics[width=6cm, height=4cm]{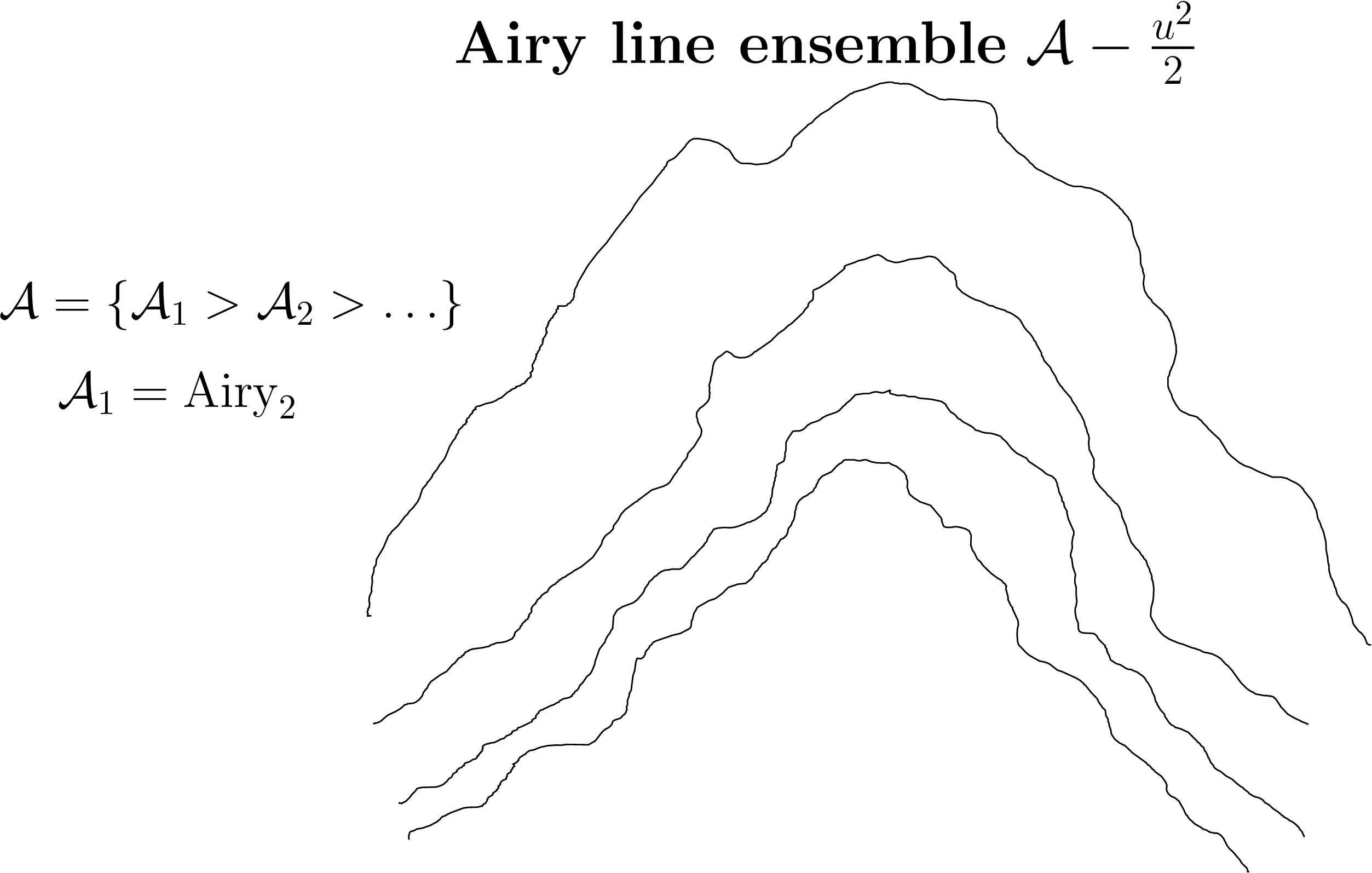}}
      &  
      \begin{itemize}
      \item  Non-intersecting paths, i.e. for any $u\in \mathbb{R}$, $\mathcal{A}_1(u) > \mathcal{A}_2(u) > \cdots$
      \item  Top curve $\mathcal{A}_1$ is $\textup{Airy}_2$ process.
      \item  $\mathcal{A}(u)$ is the Airy point process for $u \in \mathbb{R}$ fixed.
      \item Brownian Gibbs property of Airy line ensemble after subtracting a parabola $\frac{u^2}{2}$.
      \end{itemize}
        \\      
      \hline
      \end{tabular}
      \caption{Illustration of Airy line ensemble $\mathcal{A}$, with four curves drawn. }
      \label{Airy}
      \end{center}
      \end{table}
      
The Airy line ensemble $\mathcal{A}$ has been proven to be a universal edge scaling limit of a wide range of models, e.g. Gaussian unitary ensemble, Dyson Brownian motion, Brownian last passage percolation, polynuclear growth model, see \cite{OY, PS, DNV}. Another beautiful aspect about Airy line ensemble $\mathcal{A}$ is that the $\textup{Airy}_2$ process and the Airy point process are embedded together into $\mathcal{A}$. Moreover, after the subtraction of a parabola, $\tilde{\mathcal{A}}:= \mathcal{A}-\frac{u^2}{2}$ enjoys the Brownian Gibbs property (introduced in \cite{CH14}), which is a spatial Markov property and also a global resampling invariance property under the following resampling process. Taking any $a < b \in \mathbb{R}, k\in\mathbb{N}$, first remove the trajectory of the $k$-th curve $\tilde{\mathcal{A}}_k$ between $[a,b]$ and then resample a trajectory according to the law of a Brownian bridge which avoids the upper curve $\tilde{\mathcal{A}}_{k-1}$ and the lower curve $\tilde{\mathcal{A}}_{k+1}$ (note that we could run the same process for finite adjacent curves by resampling non-intersecting Brownian bridges). In other words, conditioned on the values of $\tilde{\mathcal{A}}$ outside a compact set $C = \{k_1, k_1+1, \cdots, k_2\} \times [a,b]$, the law of $\tilde{\mathcal{A}}$ inside $C$ only depends on the boundary data (i.e. independent of values of $\cL$ outside $C$). Furthermore, this conditional law of $\mathcal{A}$ on $C$ is equivalent to the law of Brownian bridges with endpoints to be $\vec{x}= (\mathcal{A}_{k_1}(a),\cdots, \mathcal{A}_{k_2}(a))$ and $\vec{y}= (\mathcal{A}_{k_1}(b),\cdots, \mathcal{A}_{k_2}(b))$ conditioned not to intersect (including not to touch upper and lower boundaries $\mathcal{A}_{k_1-1}$ and $\mathcal{A}_{k_2+1}$). 

The construction of the Airy line ensemble \cite{PS, AM, CH14} is a marriage of integrability and probability, through taking a functional limit of Brownian watermelon under edge scaling limit, see Figure~\ref{watermelon}.
\begin{figure}[h]
\begin{center}
\includegraphics[width=10cm, height = 6.5cm]{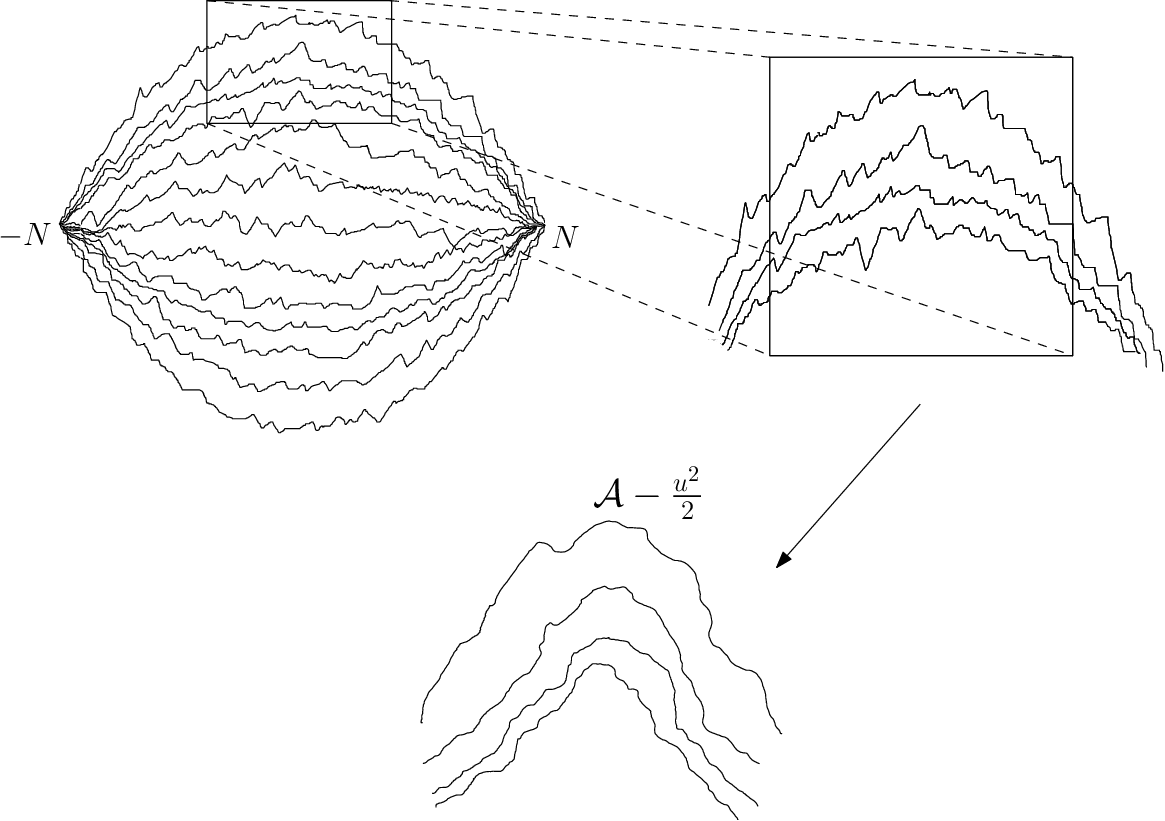}
\end{center}
\caption{Airy line ensemble as the edge scaling limit of Brownian watermelon.}
\label{watermelon}
\end{figure}
In the top left we have a Brownian watermelon $\mathcal{B} :=\{B_1, \cdots, B_N\}$, a collection of $N$ Brownian bridges $B_i: [-N, N] \rightarrow \mathbb{R}, B_i(-N) = B_i(N) = 0, 1\leq i \leq N$, conditioned not to intersect. The edge scaling limit of this system is obtained by taking a
weak limit as $N\rightarrow \infty$ of the collection of curves scaled so that the point $(0, 2
^{1/2}N)$ is fixed and space is squeezed, horizontally by a factor of $N
^{2/3}$ and vertically by $N^{1/3}$. Tightness under such edge scaling was established in \cite{CH14} by extensively exploiting the Brownian Gibbs property of Brownian watermelon (which naturally holds by the construction of Brownian watermelon) and showed the tightness. Moreover, it is demonstrated in \cite{CH14} that this Brownian Gibbs property survives under weak convergence of line ensembles, i.e. $\mathcal{A}-\frac{u^2}{2}$ enjoys the Brownian Gibbs property.

The KPZ equation is a central model in the KPZ universality class for random growth processes, interacting particle systems, and directed polymers, see \cite{Cor, QS}. Many discrete growth models have a tunable asymmetry and the KPZ equation appears as a continuum limit in the diffusive time scale (known as the weak scaling) as this parameter is critically tuned close to zero,  e.g. ASEP, directed polymers, stochastic vertex models, see \cite{BG, AKQ, CGHT, Lin}. 

Upon discovery, the Brownian Gibbs property has served as a powerful probabilistic tool. Recently, one of the authors of \cite{CH14}, Hammond, developed a more delicate treatment in \cite{Ham1} for Brownian Gibbs resampling invariance to estimate the modulus of continuity for line ensembles with Brownian Gibbs property (e.g. Airy line ensemble and the line ensemble associated with Brownian last passage percolation). Hammond also established $L^p$-norm bounds (for finite $p > 0$) on Radon-Nikodym derivative of the line ensemble curves (with an affine shift) with respect to Brownian bridges and other refined regularity properties. Furthermore in the subsequent papers \cite{Ham2,Ham3,Ham4}, the work in \cite{Ham1} was applied to understanding the geometry of last passage paths in Brownian last passage percolation with more general initial data. Another breakthrough is the construction of the Directed Landscape \cite{DOV}, the conjectural central limit of the KPZ universality class \cite{CQR}, where they found that the Airy sheet is already embedded in the Airy line ensemble and the Brownian Gibbs property is a key input of their estimates.

As the Brownian Gibbs property for Airy line ensemble has proven to be a powerful probabilistic resampling method, it is motivating to embed and study the KPZ equation as the top curve of some Gibbsian line ensemble $\mathcal{H}$. This is successfully constructed and explored in \cite{CH16}, called the KPZ line ensemble. The construction of the KPZ line ensemble in\cite{CH16} is based on subsequential extraction through O’Connell-Yor semi-discrete direct polymers and the characterization of the KPZ line ensemble through O’Connell-Warren’s \cite{OW} multilayer extension of the solution to the stochastic heat equation with narrow wedge initial data are established in \cite{Nic}.

While the non-intersecting property for Airy line ensemble $\mathcal{A}$ being a nature of the zero-temperature models, the curves of KPZ line ensemble $\mathcal{H}$ now could go out of order but subject to an exponential penalization. More specifically, the KPZ line ensemble enjoys the $\Ham$-Brownian Gibbs property, which is a more general type of Gibbs property compared to Brownian Gibbs property. The $\Ham$-Brownian Gibbs property for $\mathcal{H}$ specifies the law of $\mathcal{H}$ inside a compact set $C:= [k_1, k_2]\times [a,b]$ conditioned on the values of $\mathcal{H}$ outside $C$ such that the conditional law is equivalent to that of a few independent Brownian bridges reweighted by a penalization factor for being out order, see an illustration in Figure~\ref{KPZ_line}.
\begin{figure}
 \includegraphics[width=0.7\textwidth]{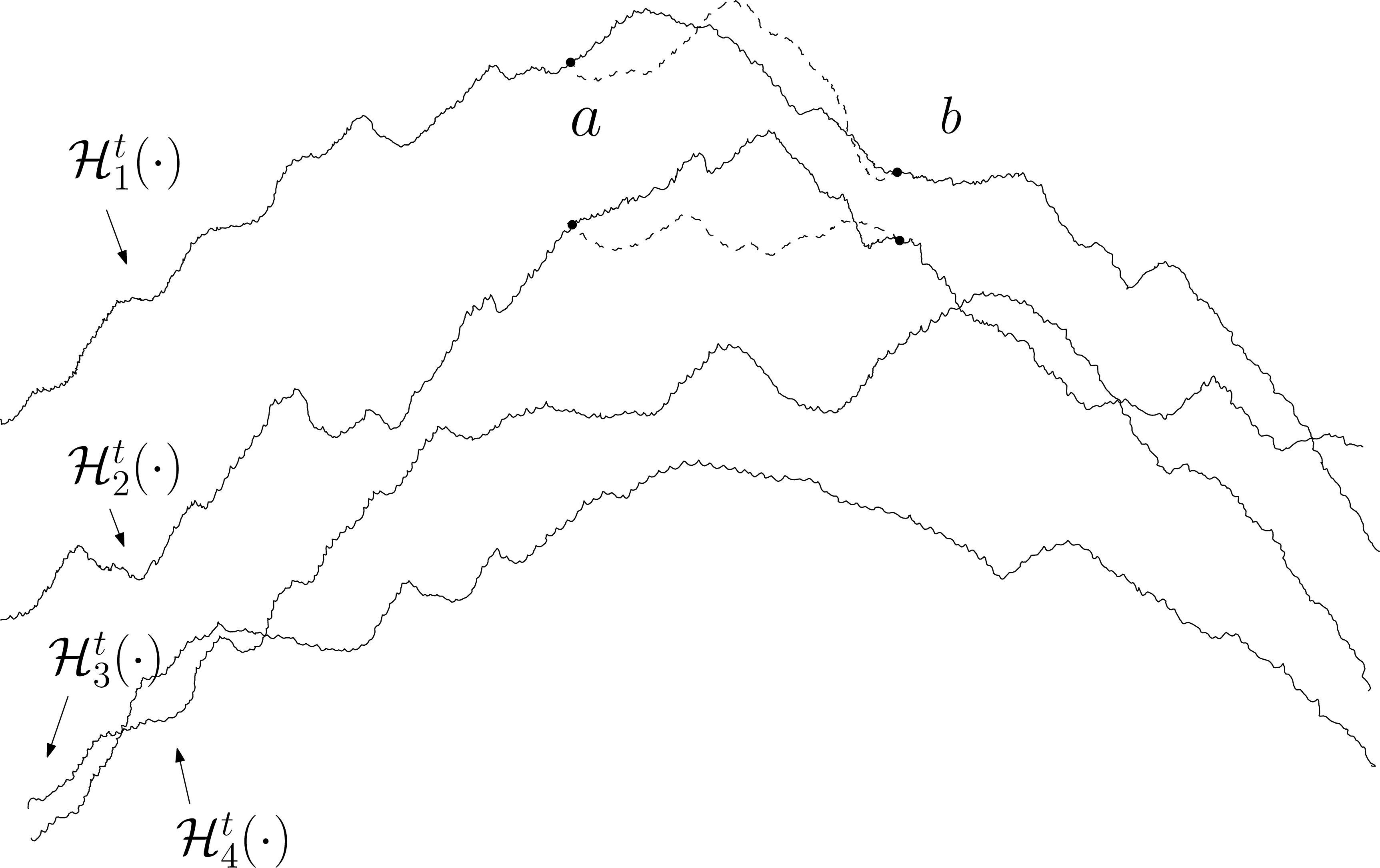}
 \caption{An overview of $\textrm{KPZ}_t$ line ensemble $\mathcal{H}^t$ for fixed $t$. Curves $\mathcal{H}^t_1(\cdot)$ through $\mathcal{H}^t_4(\cdot)$ are sampled. The lowest indexed curve $\mathcal{H}^t_1$ is distributed according to the time $t$ solution to KPZ equation with narrow wedge initial data. The dotted first two curves $\mathcal{H}^t_1$ and $\mathcal{H}^t_2$ between $a$ and $b$ indicate a possible resampling, as a demonstration for the $\Ham$-Brownian Gibbs property.}
 \label{KPZ_line}
\end{figure} 

A longstanding conjecture about the KPZ equation is that the solution of the narrow wedge initial data KPZ equation converges to the $\textrm{Airy}_2$ process (as $t$ goes to infinity) with a parabolic shift, under horizontal scaling by $t^{2/3}$ and vertical scaling by $t^{1/3}$. While the full conjecture is widely open, there has been breakthrough on the one point convergence, see \cite{SS, ACQ}. This conjecture was further strengthen in \cite{CH16} that the the KPZ line ensemble should converge to the Airy line ensemble under the same scaling. \cite{CH16} also provided a plausible route to this conjecture and one of the key steps is to characterize Airy line ensemble without relying on the determinantal formula of its finite dinmensional distributions, thus providing a new method for proving convergence in the KPZ universality class. A very recent work \cite{DM} showed that the Airy line ensemble could be characterized by the finite-dimensional marginals of its top curve and the Brownian Gibbs property.

While there have been many successes for the study of continuous Gibbs line ensembles in the KPZ universality class, the discrete Gibbsian line ensembles are also worth exploration and will be the focus of this paper, due to the richness of discrete integrable models in the KPZ universality class. The discrete Gibbsian line ensembles enjoy a discrete analogue resampling invariance of the previous Brownian Gibbs property. We call such resampling invariance random walk Gibbs property to emphasize that for line ensembles of Brownian Gibbs property or more generally, $\Ham$-Brownian Gibbs property, the underlying paths resemble Brownian bridges, while for discrete Gibbsian line ensembles, the underlying paths resemble random walk bridges. To give a few examples, through various versions of the Robinson-Schensted-Knuth (RSK) correspondence \cite{O'Con1}, one can link geometric last passage percolation to non-intersecting random walk bridges with geometric jumps, exponential last passage percolation to non-intersecting random walk bridges with exponential jumps, see \cite{DNV} for a recent study on the uniform convergence to Airy line ensemble for these two line ensembles. 

In this paper, we aim to study a sequence of log-gamma discrete line ensemble $\mathcal{L}$ with a discrete Gibbs resampling invariance property, which we call $(\Hamd,\Hrw)$-Gibbs property. The construction of this line ensemble $\mathcal{L}$ and its $(\Hamd,\Hrw)$-Gibbs property come from the study in \cite{COSZ} of a geometric RSK correspondence, when applied to the log-gamma directed polymers. The directed polymer model was introduced in the statistical physics literature by Huse and Henley \cite{HH} in 1985 and received first rigorous mathematical treatment in 1988 by Imbrie and Spencer \cite{IS}. The monograph \cite{Com} is a great resource for the foundational work in this area. Among directed polymers, the log-gamma directed polymer model was first introduced in \cite{Sep}, which is special in the same way as the last-passage percolation model with exponential or geometric weights is special among corner growth models, namely, both demonstrate integrable structures and permit explicit computations.

We further apply the weak noise scaling to the log-gamma line ensemble $\mathcal{L}$. This scaling regime is known as the intermediate disorder regime in \cite{AKQ}, where they showed the convergence of directed polymers with general random environment to the KPZ equation with narrow wedge initial data, hence establishing the weak KPZ universality for directed polymers. Denoting $\overbar{\mathcal{L}}^N$ as the scaled log-gamma line ensemble (see Figure~\ref{log-gamma-line} for an illustration), we want to take a functional limit of $\overbar{\mathcal{L}}^N$ as $N$ goes to infinity (in the topology of uniform convergence for continuous functions on compact sets). Moreover we prove that the $\Ham$-Brownian Gibbs property with $\Ham(x) = e^x$ (the same Gibbs property as enjoyed by the KPZ line ensemble $\mathcal{H}$) for all subsequential limits of $\overbar{\mathcal{L}}^N$.
\begin{figure}
 \includegraphics[width=0.7\textwidth]{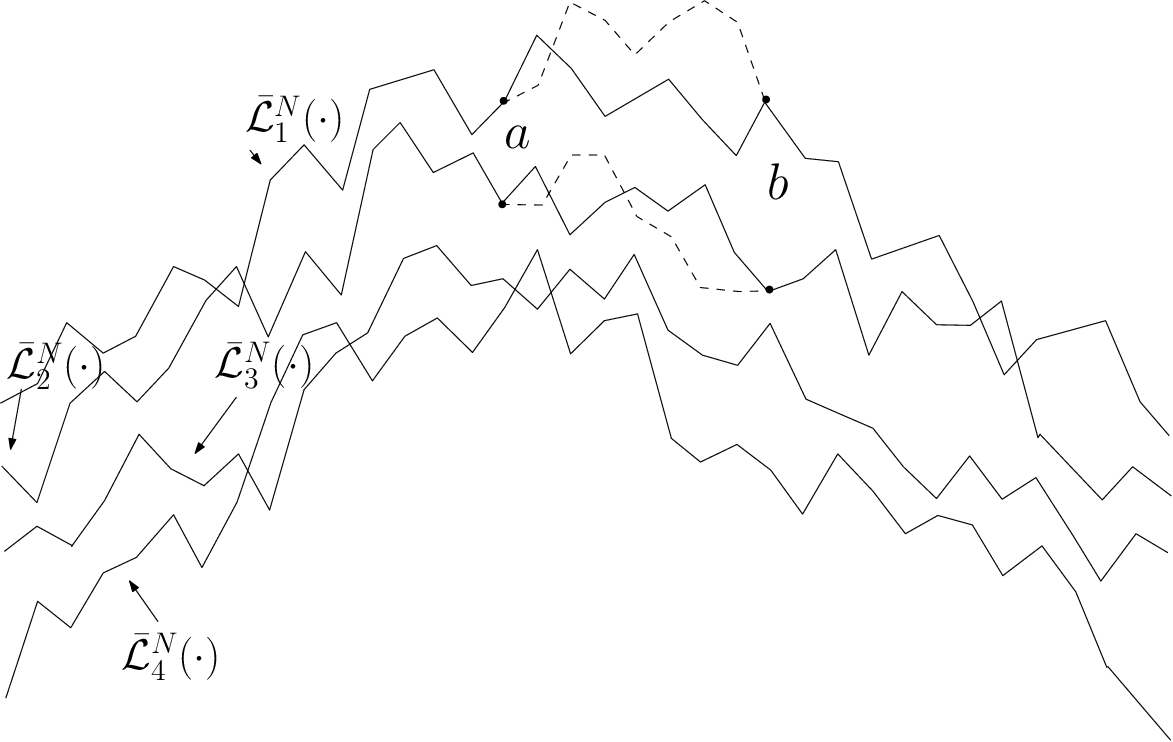}
 \caption{An overview of scaled log-gamma line ensemble $\overbar{\mathcal{L}}^N$. Discrete curves $\overbar{\mathcal{L}}^N_1(\cdot)$ through $\overbar{\mathcal{L}}^N_4(\cdot)$ are sampled on $\frac{2}{\sqrt{N}}\mathbb{Z}$ (linked through the linear interpolation between two adjacent points). The lowest indexed curve $\overbar{\mathcal{L}}^N_1(\cdot)$ converges weakly to time $t$ solution to the KPZ equation with narrow wedge initial data as $N$ goes to infinity. The dotted first two curves $\overbar{\mathcal{L}}^N_1(\cdot)$ and $\overbar{\mathcal{L}}^N_2(\cdot)$ between $a$ and $b$ indicate a possible resampling, as a demonstration for the random walk Gibbs property.}
 \label{log-gamma-line}
\end{figure}

Instead of working directly with the sequence of scaled log-gamma line ensembles $\overbar{\mathcal{L}}^N$, we introduce a general framework for studying the tightness of a sequence of $(\Hamd^N,\HrwN)$-Gibbs line ensembles, see our main result Theorem~\ref{theorem:main}. We propose assumptions A1-A4 that capture the properties enjoyed by $(\Hamd^N,\HrwN)$ that we rely on. To summarize here, consider a sequence of discrete line ensembles $\mathcal{L}^N_{i}(u), (i,u)\in \{1,\cdots, K\} \times \frac{1}{N}\mathbb{Z}$, which enjoys $(\Hamd^N,\HrwN)$-Gibbs property such that $(\Hamd^N,\HrwN)$ satisfy assumptions A1-A4, and assume that $\mathcal{L}^N_1(u)+u^2/2$ (defined through linear interpolation) converges weakly to a stationary process, we obtain the following result (Theorem~\ref{theorem:main} in the main text),
\begin{enumerate}
\item For any $T>0$ and $1\leq k \leq K$, the restriction of the line ensemble $\mathcal{L}^N$ to $\{1, \cdots, k\}\times [-T,T]$ is sequentially compact as $N$ varies. 

\item Any subsequential limit line ensemble $\mathcal{L}^{\infty}$ satisfies $\Ham$-Brownian Gibbs property with $\Ham(x)=e^x$.
\end{enumerate}

We then verify that the scaled log-gamma line ensembles $\overbar{\mathcal{L}}^N$ fall into the category where this theorem applies and obtain the above results for $\overbar{\mathcal{L}}^N$.

It is worth mentioning that another particularly successful instance of the discrete Gibbs line ensemble is studied in \cite{CD}, where the authors investigated a discrete Gibbsian line ensemble related to the ascending Hall-Littlewood process (a special case of the Macdonald processes \cite{BC}). By developing discrete analogues of the arguments in \cite{CH16}, \cite{CD} were successful in establishing the long-predicted 2/3 critical exponent for the asymmetric simple exclusion process (ASEP). 

\subsection*{Notation}
We describe notations used in this paper. Given $k_1<k_2$ with $k_1,k_2\in \mathbb{Z}$, we denote $[k_1,k_2]_{\mathbb{Z}}:=\{k_1,k_1+1,\dots, k_2\}$ and $(k_1,k_2)_{\mathbb{Z}}:\{k_1+1,\dots, k_2-1\}$. Given a subset

\subsection*{Outline}
Subsection~\ref{def:basics of line ensembles} and Subsection~\ref{def:Brownian and random walk gibbs property} describe the general setting of line ensembles and introduce the main objects studied in this paper, discrete line ensembles with random walk Gibbs property. Subsection~\ref{Assumptions} states assumptions A1-A4, under which the main theorem is also stated. The proof of main Theorem~\ref{theorem:main} is in Subsection~\ref{proof of main_thm (1)} and Subsection~\ref{proof of main_thm (2)}. Subsection~\ref{KMT_assumption} provides estimates on random walk bridges and discrete Gibbs line ensembles and Section~\ref{propproofs} contains the proofs of two key propositions. These are the main technical results in this paper. Section~\ref{sec:log_gamma} presents one interesting application of our main theorem to the scaled log-gamma line ensemble. Appendix~\ref{appendix} contains a proof for monotonicity Lemma~\ref{lem:monotone}.

\subsection*{Acknowledgements} The author is deeply grateful to Ivan Corwin for his consistent support and many useful suggestions. The author also thanks Evgeni Dimitrov for many helpful discussions and as well as Ivan Corwin and Vu Lan Nguyen for their efforts and initial contributions in a earlier draft of this project. The author was supported by Ivan Corwin through the NSF grants DMS-1811143, DMS-1664650 and also by the Minerva Foundation Summer Fellowship program.
	
\section{Gibbsian line ensembles and the main result}
We first introduce the basic notions of line ensembles in Subsection~\ref{def:basics of line ensembles} and then define the main objects of study in this paper -- Brownian and random walk Gibbsian line ensembles in Subsection~\ref{def:Brownian and random walk gibbs property}. Lastly in Subsection~\ref{Assumptions}, we list the assumptions A1-A4, under which we formulate the main result Theorem~\ref{theorem:main} of this paper.

\subsection{Basics about line ensembles}\label{def:basics of line ensembles}
\begin{definition}\label{def:line-ensemble}
Let $\Sigma$ be an interval of $\Z$ and let $\Lambda$ be a subset of $\R$. Consider the set $C(\Sigma\times \Lambda,\R)$ of continuous functions $f:\Sigma\times\Lambda \to \R$ endowed with the topology of uniform convergence on compact subsets of $\Sigma\times \Lambda$, and let $\mathcal{C}(\Sigma\times \Lambda,\R)$ denote the sigma-field generated by Borel sets in $C(\Sigma\times \Lambda,\R)$. A $\Sigma\times \Lambda$-indexed line ensemble $\mathcal{L}$ is a random variable on a probability space $(\Omega,\mathcal{B},\PP)$, taking values in $C(\Sigma\times \Lambda,\R)$ such that $\mathcal{L}$ is a measurable function from $\mathcal{B}$ to $\mathcal{C}(\Sigma\times \Lambda,\R)$.
\end{definition}

For integers $k_1<k_2$, let $[k_1,k_2]_{\Z} := \{k_1,k_1+1,\ldots,k_2\}$. When $\Lambda$ is a discrete subset of $\R$, it is possible to extend the line ensemble to one with $\Lambda$ replaced by its convex hull (i.e. the minimal interval of $\R$ containing all points of $\Lambda$). Under this extension, the lines of $\mathcal{L}$ are extended by linear interpolation and the convergence of $\mathcal{L}$ implies the convergence of the extension. We will sometimes abuse notations and conflate a discrete line ensemble defined on a discrete set $\Lambda$ with its linearly interpolated ensemble. Also we will generally write $\mathcal{L}:\Sigma\times \Lambda\to \R$ even though it is not $\mathcal{L}$, but rather $\mathcal{L}(\omega)$ for each $\omega\in \Omega$ which is such a function. We will also sometimes specify a line ensemble by only giving its law without reference to the underlying probability space.  We write $\mathcal{L}_i(\cdot):= \big(\mathcal{L}(\omega)\big)(i,\cdot)$ for the label $i\in \Sigma$ curve of the ensemble $\mathcal{L}$. 

\begin{definition}\label{def:CONVERGE}
Given a $\Sigma\times \Lambda$-indexed line ensemble $\mathcal{L}$ and a sequence of such ensembles $\left\{\mathcal{L}^N\right\}_{N\geq 1}$, we will say that $\mathcal{L}^N$ converges to $\mathcal{L}$ weakly
as a line ensemble if for all bounded continuous functions $F:C(\Sigma\times \Lambda,\R)\to \R$, as $N\to \infty$,
	\begin{equation*}
		\int F\big(\mathcal{L}^N(\omega)\big)d\mathbb{P}^N(\omega) \to \int F\big(\mathcal{L}(\omega)\big) d\mathbb{P}(\omega).
	\end{equation*}
This is equivalent to weak-$*$ convergence in $C(\Sigma\times \Lambda,\R)$ endowed with the topology of uniform convergence on compact subsets of $\Sigma\times \Lambda$.
\end{definition}

We will define two types of Gibbsian bridge line ensembles -- those whose underlying path measures are given by Brownian motions, and those given by discrete time random walks. We start with the Brownian case.
 
\begin{definition}\label{def:BrownianBridge}
For any $L>0$, we denote by $B_L:[0,L]\to\mathbb{R}$ the standard Brownian bridge with $B_L(0)=B_L(L)=0$. 
\end{definition}

\begin{definition}\label{def:H_Brownian}
Fix $k_1\leq k_2$ with $k_1,k_2 \in \mathbb{Z}$, an interval $[a,b]\subset \mathbb{R}$ and two vectors $\vec{x},\vec{y}\in \mathbb{R}^{k_2-k_1+1}$. We use $\mathbb{P}^{k_1,k_2,(a,b),\vec{x},\vec{y}}_{\free}$ to denote the law of a $[k_1,k_2]_{\Z}\times [a,b]$-indexed line ensemble $\mathcal{L} = (\mathcal{L}_{k_1},\ldots,\mathcal{L}_{k_2})$ in which each $\mathcal{L}_k$ is independent of other indexed curves and $$\mathcal{L}_k(u)\overset{(d)}{=} x_k+\frac{u-a}{b-a}(y_k-x_k)+B_{b-a}(u-a).$$ 

A Hamiltonian $\Ham$ is defined to be a continuous function $\Ham:\mathbb{R}\cup\{-\infty\}\to [0,\infty]$ with $\Ham (-\infty)=0$. Let $f:(a,b)\to \mathbb{R}\cup\{+\infty\}$ and $g:(a,b)\to \mathbb{R}\cup\{-\infty\}$ be measurables function with $\inf_{u\in (a,b)} f(u)>-\infty$ and $\sup_{u\in (a,b)} g(u)< \infty$. Let $\vec{x},\vec{y}\in \mathbb{R}^{k_{2}-k_1+1}$ be two vectors. For $\mathcal{L}=(\mathcal{L}_{k_1},\dots,\mathcal{L}_{k_2})\in C([k_1,k_2]_{\mathbb{Z}}\times [a,b],\mathbb{R}), $ define the \textit{Boltzmann weight} to be
\begin{equation}\label{def:Boltzmann_Brownian}
W_{\Ham}^{k_1,k_2,(a,b),\vec{x},\vec{y},f,g}(\mathcal{L}):= \exp\Bigg\{-\sum_{i=k_1-1}^{k_2}\int_a^b \Ham\Big(\mathcal{L}_{i+1}(u)-\mathcal{L}_{i}(u)\Big)du\Bigg\}.
\end{equation}
Here we adopt convention that $\mathcal{L}_{k_1-1}=f$, $\mathcal{L}_{k_2+1}=g$. The normalizing constant is defined by
\begin{equation}\label{def:normalcont_Brownian}
Z_{\Ham}^{k_1,k_2,(a,b),\vec{x},\vec{y},f,g} :=\mathbb{E}^{k_1,k_2,(a,b),\vec{x},\vec{y}}_{\free}\Big[W_{\Ham}^{k_1,k_2,(a,b),\vec{x},\vec{y},f,g}(\mathcal{L})\Big],
\end{equation}
where $\mathcal{L}$ in the above expectation is distributed according to the measure $\mathbb{P}^{k_1,k_2,(a,b),\vec{x},\vec{y}}_{\free}$. See Remark~\ref{rmk:Wmeasurable} for the measurability of the $W_{\Ham}^{k_1,k_2,(a,b),\vec{x},\vec{y},f,g}(\mathcal{L})$.\\

Suppose that 
\begin{equation}
 Z_{\Ham}^{k_1,k_2,(a,b),\vec{x},\vec{y},f,g} >0. 
\end{equation}
We define the $\Ham$-Brownian bridge line ensemble with entrance data $\vec{x}$, exit data $\vec{y}$ and boundary data $(f,g)$ to be a $[k_1,k_2]_{\Z}\times [a,b]$-indexed line ensemble with law $\mathbb{P}^{k_1,k_2,(a,b),\vec{x},\vec{y},f,g}_{\Ham}$ given according to the following Radon-Nikodym derivative relation:
\begin{equation*}
\frac{d\mathbb{P}_{\Ham}^{k_1,k_2,(a,b),\vec{x},\vec{y},f,g}}{d\mathbb{P}_{\free}^{k_1,k_2,(a,b),\vec{x},\vec{y}}}(\mathcal{L}) := \frac{W_{\Ham}^{k_1,k_2,(a,b),\vec{x},\vec{y},f,g}(\mathcal{L})}{Z_{\Ham}^{k_1,k_2,(a,b),\vec{x},\vec{y},f,g}}.
\end{equation*}
\end{definition}

\begin{remark}\label{rmk:Wmeasurable}
In this remark we discuss the measurability of the Boltzmann weight. By the continuity of $\mathbf{H}$ and $\mathbf{H}(-\infty)=0$, $\mathbf{H}(x)$ is uniformly continuous and bounded on $[-\infty,M]$ for any $M\in\mathbb{R}$. From this, it can be checked that for any functions $f$ and $g$ as in Definition~\ref{def:H_Brownian}, $W_{\Ham}^{k_1,k_2,(a,b),\vec{x},\vec{y},f,g}$ is a continuous function on $C([k_1,k_2]_{\mathbb{Z}}\times [a,b],\mathbb{R})$. This implies we can take the expectation in \eqref{def:normalcont_Brownian}. 
\end{remark}

In Definition \ref{def:H_Brownian}, we use Brownian bridges to build our line ensemble. We now describe how we may similarly construct discrete line ensembles in terms of random walk bridges. Random walks come in different flavors based on the choice of continuous versus discrete time, and continuous versus discrete jump distributions. In principle, for each such choice we can run the same type of construction as below. In this paper we focus on discrete time and continuous jump distributions, as it is suitable for our eventual application to study the line ensemble associated to the log-gamma directed polymers as introduced in \cite{Sep} and further studied in \cite{COSZ}.

We start by defining $\Hrw$-random walk bridges using the Hamiltonian function $\Hrw$, as well as various line ensembles built off of them.

\begin{definition}\label{def:RW}
A random walk Hamiltonian is a continuous function $\Hrw:\R \to (-\infty,\infty)$ such that
\begin{equation*}
\int_{\R} \exp\big(-\Hrw(x)\big)dx= 1.
\end{equation*}
Let $X_i$, $i\in\mathbb{N}$ i.i.d. random variables with probability density function given by $\exp(-\Hrw(x))$. For any $k\in\mathbb{N}$, we write $S(k)= X_1+X_2+\dots +X_k$ for the $k$-th step random walk. We adopt the convention that $S(0)\equiv 0$. The p.d.f. of $S_k$ can be inductively defined through 
\begin{align}\label{def:fk}
f_1(x):=\exp\left( -\Hrw (x) \right),\  f_k(x):=\int_{\mathbb{R}} f_1(x-y) f_{k-1}(y)\, dy.   
\end{align} 
Given $n\in\mathbb{N}$ and $z\in\mathbb{R}$, we denote by $\{S_{n,z}(k)\}_{k=0}^n$ the random walk $\{S(k)\}_{k=0}^n$ conditioned on $S(n)=z$. In other words,
\begin{equation}\label{equ:1225716}
S_{n,z}(k):=\, S(k)\, |\, S(L)=z.
\end{equation}
We often view $S_{n,z}(u)$ as continuous function on $u\in [0,n]$ through linear interpolation. 
\end{definition}
In the lemma below, we show that $S_{n,z}$ can be coupled continuously in one probability space. The proof can be found in Appendix~\ref{sec:rwb}.
\begin{lemma}\label{lem:RW_coupling}
Fix $n\in\mathbb{N}$. There exists a probability space $(\Omega_n,\mathcal{B}_n,\mathbb{P}_n)$ and a measurable map $\mathbf{G}_n:\mathbb{R}\times\Omega_n\to C([0,n],\mathbb{R})$ such that the following holds. For any $z\in\mathbb{R}$, the law of $\mathbf{G}_n(z,\cdot)$ is given by $S_{n,z}$ defined in Definition~\ref{def:RW}. Moreover, for any $\omega\in \Omega_n$, $\mathbf{G}_n(\cdot ,\omega)$ is continuous in the first variable. 
\end{lemma}

Next, we fix a discrete set and scale the random walk bridges accordingly.

\begin{definition}\label{def:RW_ensemble}
Fix $d>0$ and define a discrete set as $\Lambda_d := d\mathbb{Z} $. For $a<b$, let $\Lambda_d(a,b) = (a,b)\cap \Lambda_d$, $\Lambda_d[a,b] = [a,b]\cap \Lambda_d$ and likewise for half open / half close intervals.

Fix $L\in d\mathbb{N}$, $z\in\mathbb{R}$ and a random walk Hamiltonian $\Hrw$. Recall that $S_{d^{-1}L,z}$ is the random walk bridge on $[0,d^{-1}L]_{\mathbb{Z}}$ defined in Definition~\ref{def:RW}. We define	 
\begin{align*}
 \bar{S}_{L,z}(u):=  S_{d^{-1}L,z}(d^{-1}u ).  
\end{align*}  

Fix $k_1\leq k_2$ with $k_1, k_2\in \Z$, $a<b$ with $a,b \in \Lambda_d$ and two vectors  $\vec{x},\vec{y}\in \mathbb{R}^{k_2-k_1+1}$. We use $\PP_{\free,\Hrw}^{k_1,k_2,\Lambda_d(a,b),\vec{x},\vec{y}}$ to denote the law of a $[k_1,k_2]_{\Z}\times \Lambda_d[a,b]$-indexed line ensemble $\mathcal{L} = (\mathcal{L}_{k_1},\ldots,\mathcal{L}_{k_2})$ in which each $\mathcal{L}_k$ is independent of other indexed curves and $$\mathcal{L}_k(u)\overset{(d)}{=} x_k+\bar{S}_{b-a,y_k-x_k}(u-a).$$ 

\end{definition}
\begin{definition}\label{def:RWB_Gibbs_ensemble}
A local interaction Hamiltonian is a continuous function $\Hamd:(\R\cup \{\pm \infty\})^6 \to [0,\infty]$ (see Remark \ref{rem.local} for an explanation of the meaning of each slot of $\Hamd$). Fix $a<b$ with $a,b\in \Lambda_d$. Let $f:\Lambda_d(a,b)\to \mathbb{R}\cup\{+\infty\}$ and $g:\Lambda_d(a,b)\to \mathbb{R}\cup\{-\infty\}$ be functions. Let $\vec{x},\vec{y}\in \mathbb{R}^{k_{2}-k_1+1}$ be two vectors. For $\mathcal{L}=(\mathcal{L}_{k_1},\dots,\mathcal{L}_{k_2})$ in $ C([k_1,k_2]_{\Z}\times \Lambda_d[a,b],\mathbb{R}), $ define the \textit{Boltzmann weight} to be 
\begin{equation}\label{def:Boltzmann_randomwalk}
W_{\Hamd}^{k_1,k_2,\Lambda_d(a,b),\vec{x},\vec{y},f,g}(\mathcal{L}):= \exp\bigg\{-\sum_{k=k_1-1}^{k_2}\sum_{u \in \Lambda_d(a,b)} \Hamd\big(\rectangle(\mathcal{L},k,u )\big)\bigg\},
\end{equation}
where we adopt convention that $\mathcal{L}_{k_1-1}=f$, $\mathcal{L}_{k_2+1}=g$ and
\begin{equation}
\label{def:boxplus}
\rectangle(\mathcal{L},k,u ) = \big(\mathcal{L}_{k}(u-d ), \mathcal{L}_{k}(u ),\mathcal{L}_{k}(u+d),\mathcal{L}_{k+1}(u-d), \mathcal{L}_{k+1}(u ),\mathcal{L}_{k+1}(u+d)\big),
\end{equation}
see Figure~\ref{boxplus} below for an illustration of the inputs for $\rectangle(\mathcal{L},k,u)$. The normalizing is defined by 
\begin{equation}\label{def:normal_randomwalk}
Z_{\Hamd,\Hrw}^{k_1,k_2,\Lambda_d(a,b),\vec{x},\vec{y},f,g} =\mathbb{E}_{\free,\Hrw}^{k_1,k_2,\Lambda_d(a,b),\vec{x},\vec{y}}\Big[W_{\Hamd}^{k_1,k_2,\Lambda_d(a,b),\vec{x},\vec{y},f,g}(\mathcal{L})\Big],
\end{equation}
where $\mathcal{L}$ in the above expectation is distributed according to the measure $\mathbb{P}_{\free,\Hrw}^{k_1,k_2,\Lambda_d(a,b),\vec{x},\vec{y}}$.\\

Suppose that 
\begin{equation}\label{equ:Z>0}
Z_{\Hamd,\Hrw}^{k_1,k_2,\Lambda_d(a,b),\vec{x},\vec{y},f,g} >0. 
\end{equation}
We define the $(\Hamd,\Hrw)$-random walk bridge line ensemble with entrance data $\vec{x}$, exit data $\vec{y}$ and boundary data $(f,g)$ to be a $[k_1,k_2]_{\Z}\times \Lambda_d[a,b]$-indexed line ensemble $\mathcal{L}=(\mathcal{L}_{k_1},\ldots, \mathcal{L}_{k_2})$ with law $\mathbb{P}_{\Hamd,\Hrw}^{k_1,k_2,\Lambda_d(a,b),\vec{x},\vec{y},f,g}$ given according to the following Radon-Nikodym derivative relation:
\begin{equation*}
\frac{d\mathbb{P}_{\Hamd,\Hrw}^{k_1,k_2,\Lambda_d(a,b),\vec{x},\vec{y},f,g}}{d\mathbb{P}_{\free,\Hrw}^{k_1,k_2,\Lambda_d(a,b),\vec{x},\vec{y}}}(\mathcal{L}) = \frac{W_{\Hamd}^{k_1,k_2,\Lambda_d(a,b),\vec{x},\vec{y},f,g}(\mathcal{L})}{Z_{\Hamd,\Hrw}^{k_1,k_2,\Lambda_d(a,b),\vec{x},\vec{y},f,g}}.
\end{equation*}
\centering
\includegraphics[width= 0.4\textwidth, height=3cm]{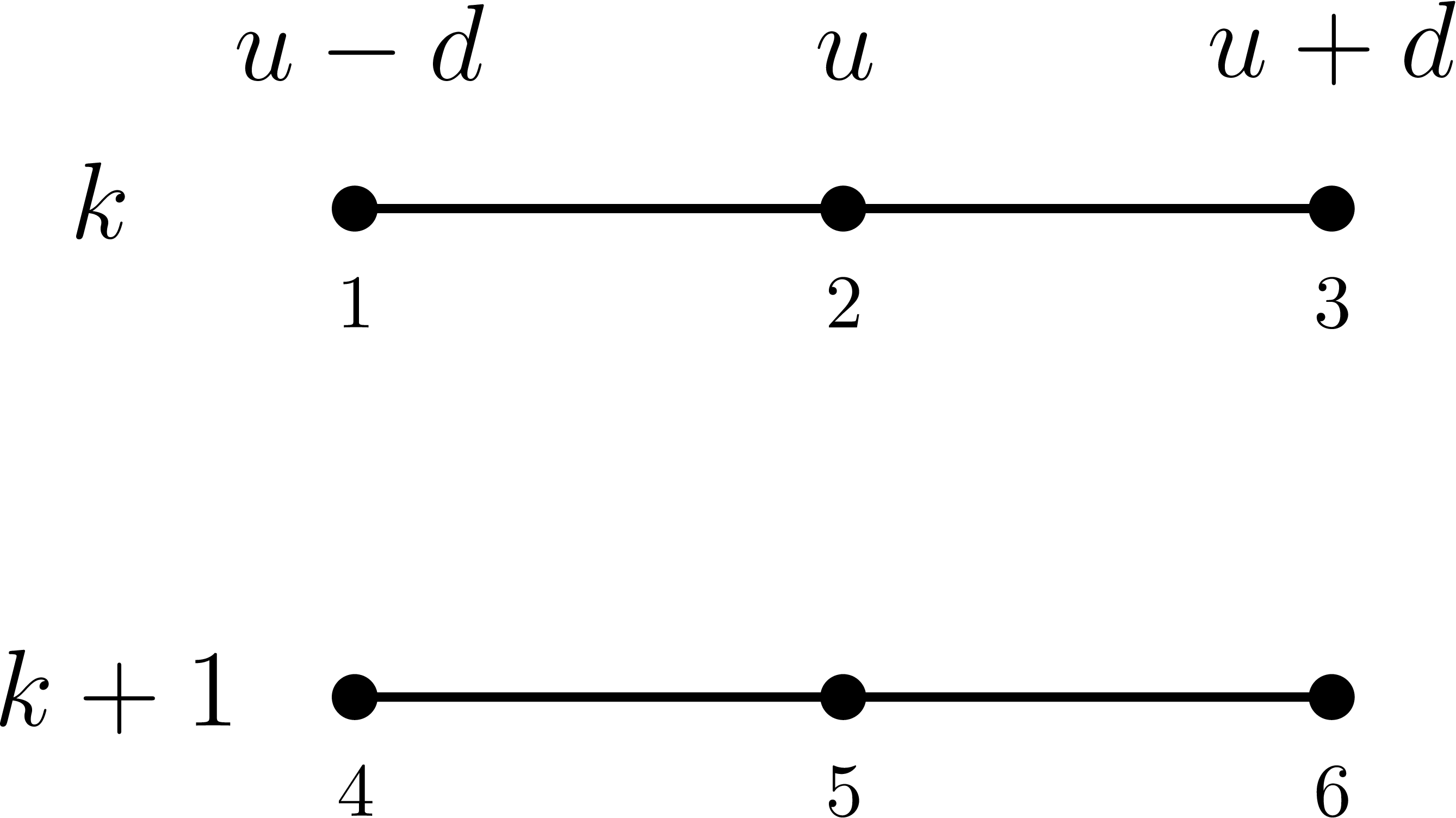}
\captionof{figure}{The points labeled with 1-6 correspond to the six inputs in \eqref{def:boxplus}.}
\label{boxplus}
\end{definition}

\begin{remark}\label{rem.local}
The function $\rectangle(\mathcal{L},k,u_m)$ defined in \eqref{def:boxplus} returns the point $\mathcal{L}_{k}(u_m)$ as well as five of its neighbors (in terms of the index space) corresponding to $k$ and $k+1$, as well as $u_{m-1}, u_{m}, u_{m+1}$. The interaction Hamiltonians we consider are nearest neighbor in that a single curve at a single time will only depend on the curve above and below it at that same time or plus or minus one time increment. The function $\rectangle(\mathcal{L},k,u_{m})$ is used in specifying the interaction felt by curve $k$ with respect to curve $k+1$. One could consider longer range interaction Hamiltonians, though it would require extending the boundary data necessary to specify the Gibbs property.
\end{remark}

In the next lemma, we show that the normalizing constants defined in \eqref{def:normalcont_Brownian} and \eqref{def:normal_randomwalk} are measurable with respect to the boundary data.
\begin{lemma}\label{lem:Zmeasurable}
Continue with the notations of Definition~\ref{def:H_Brownian}. For any bounded continuous function $F:C([k_1,k_2]_{\mathbb{Z}}\times [a,b],\mathbb{R})\to\mathbb{R}$, 
$$\EE^{k_1,k_2,(a,b),\vec{x},\vec{y} }_{\free}\Big[F(\mathcal{  \mathcal{L} })\cdot W^{k_1,k_2,(a,b),\vec{x},\vec{y},f,g}_{\mathbf{H}}(\mathcal{L}))\Big]$$
is a measurable function in $$(\vec{x},\vec{y},f,g)\in \mathbb{R}^{k_2-k_1+1}\times \mathbb{R}^{k_2-k_1+1}\times C([a,b],\mathbb{R}\cup\{\infty\} )\times C([a,b],\mathbb{R}\cup\{-\infty\}). $$

Continue with the notations of Definition~\ref{def:RWB_Gibbs_ensemble}. For any bounded continuous function $F:C([k_1,k_2]_{\mathbb{Z}}\times \Lambda_d[a,b],\mathbb{R})\to\mathbb{R}$, 
$$\EE^{k_1,k_2,\Lambda_d(a,b),\vec{x},\vec{y}}_{\free,\Hrw}\Big[F(\mathcal{\mathcal{L} })\cdot W_{\Hamd}^{k_1,k_2,\Lambda_d(a,b),\vec{x},\vec{y},f,g}(\mathcal{L}) \Big]$$
is a measurable function in $$(\vec{x},\vec{y},f,g)\in \mathbb{R}^{k_2-k_1+1}\times \mathbb{R}^{k_2-k_1+1}\times C(\Lambda_d[a,b],\mathbb{R}\cup\{\infty\})\times C(\Lambda_d[a,b],\mathbb{R}\cup\{-\infty\}). $$  
\end{lemma}
\begin{proof}
For simplicity, we denote
\begin{align*}
\mathfrak{X}:= \mathbb{R}^{k_2-k_1+1}\times C([a,b],\mathbb{R}\cup\{\infty\} )\times C([a,b],\mathbb{R}\cup\{-\infty\}).
\end{align*}
Let $\vec{0}$ be the zero vector in $\mathbb{R}^{k_2-k_1-1}$ and $\mathcal{L}'(\omega)=(\mathcal{L}'_{k_1}(\omega),\dots,\mathcal{L}'_{k_2}(\omega))$ be distributed according to $\mathbb{P}^{k_1,k_2,(a,b),\vec{0},\vec{0}}_{\free}$ defined in Definition~\ref{def:H_Brownian}. We will write $\Omega$ for the corresponding probability space. Given $\vec{x},\vec{y}\in \mathbb{R}^{k_2-k_1-1}$, define $\mathcal{L}(\vec{x},\vec{y},\omega)=(\mathcal{L}_{k_1}(\omega),\dots,\mathcal{L}_{k_2}(\omega))$ by
$$\big(\mathcal{L}_k(\omega)\big)(u)=x_k+\frac{u-a}{b-a}(y_k-x_k)+ \big(\mathcal{L}'_k(\omega) \big)(u).$$
Define the function $Q:\mathfrak{X} \times\Omega\to \mathbb{R}$ by
\begin{align*}
Q(\vec{x},\vec{y},f,g,\omega):=F(\mathcal{L}(\vec{x},\vec{y},\omega))\cdot W_{\Ham}^{k_1,k_2,(a,b),\vec{x},\vec{y},f,g}(\mathcal{L}(\vec{x},\vec{y},\omega)).
\end{align*}
By the continuity of $\mathbf{H}$, $Q$ is measurable on $\mathfrak{X} \times\Omega$. By the Fubini–Tonelli theorem,
\begin{align*}
\EE^{k_1,k_2,(a,b),\vec{x},\vec{y} }_{\free}\Big[F(\mathcal{  \mathcal{L} })\cdot W^{k_1,k_2,(a,b),\vec{x},\vec{y},f,g}_{\mathbf{H}}(\mathcal{L}))\Big]=\int_\Omega Q(\vec{x},\vec{y},f,g,\omega)\, d\mathbb{P}^{k_1,k_2,(a,b),\vec{0},\vec{0}}_{\free}(\omega)
\end{align*}
is measurable in $\mathfrak{X}$. The argument for the discrete case is analogous except that the construction of $\mathcal{L}(\vec{x},\vec{y},\omega)$ relies on Lemma~\ref{lem:RW_coupling}.
\end{proof}

\subsection{$\Ham$-Brownian Gibbs property and discrete $(\Hamd,\Hrw)$-Gibbs property }\label{def:Brownian and random walk gibbs property}
The Gibbs property for a line ensemble can be thought of as a spatial version of the Markov property whereby the distribution of a field in a given compact region depends entirely on the distribution of the field on the region's boundary. For a line ensemble which enjoys a Gibbs property, this distribution conditional on the boundary is specified exactly via the type of Radon-Nikodym derivative prescriptions in Definitions \ref{def:H_Brownian} and \ref{def:RWB_Gibbs_ensemble}.

\begin{definition}[$\Ham$-Brownian Gibbs property]\label{def:Brownian-H-Gibbsproperty}
Let $\Sigma$ be an interval in $\mathbb{N}$, $\Lambda$ an interval in $\mathbb{R}$ and $\Ham$ a Hamiltonian function. A $\Sigma\times \Lambda$-indexed line ensemble $\cL$ defined on a probability space $(\Omega, \cB, \PP)$ has the $\Ham$-Brownian Gibbs property if the following holds. Fix arbitrary $[k_1,k_2]_{\mathbb{Z}}\subset \Sigma$, $(a,b)\subset \Lambda$ and a bounded continuous function $F: C([k_1,k_2]\times (a,b), \mathbb{R})\rightarrow\mathbb{R}$. It holds $\PP$-almost surely that
\begin{align}\label{eqn:H-Brownian-Gibbs}
\EE\Big[F\big(\mathcal{L}\vert_{[k_1,k_2]_{\Z}\times (a,b)}\big)\big\vert \mathcal{F}_{\textrm{ext}}\big([k_1,k_2]_{\Z}\times  (a,b)\big)\Big]= \EE^{k_1,k_2,(a,b),\vec{x},\vec{y},f,g}_{\Ham}\Big[F(\mathcal{\tilde L})\Big].
\end{align}
Here 
\begin{equation*}
\mathcal{F}_{\textrm{ext}}\big([k_1,k_2]_{\Z}\times (a,b)\big):=\sigma\big(\mathcal{L}_k(u):(k,u)\in \Sigma\times\Lambda\backslash [k_1,k_2]_{\Z}\times (a,b)\big)
\end{equation*}
is the sigma field generated by the line ensemble outside $[k_1,k_2]_{\Z}\times (a,b)$. The ensemble $\mathcal{\tilde L}$ on the right-hand side of \eqref{eqn:H-Brownian-Gibbs} is independently drawn and the expectation is taken according to the $\PP^{k_1,k_2,(a,b),\vec{x},\vec{y},f,g}_{\Ham}$ measure, where we have entrance data $\vec{x} = \big(\mathcal{L}_{k_1}(a),\ldots, \mathcal{L}_{k_2}(a)\big)$, exit data $\vec{y} = \big(\mathcal{L}_{k_1}(b),\ldots, \mathcal{L}_{k_2}(b)\big)$, upper boundary curve $f=\mathcal{L}_{k_1-1}|_{(a,b)}$ and lower boundary curve $g=\mathcal{L}_{k_2+1}|_{(a,b)}$. By convention if $k_1-1\notin \Sigma$ then $\mathcal{L}_{k_1-1}$ is assumed to be everywhere $+\infty$ and likewise if $k_2+1\notin\Sigma$ then $\mathcal{L}_{k_2+1}$ is assumed to be everywhere $-\infty$. Note that from Lemma~\ref{lem:Zmeasurable}, the right hand side of \eqref{eqn:H-Brownian-Gibbs} is $\mathcal{F}_{\textrm{ext}}\big([k_1,k_2]_{\Z}\times \Lambda_d(a,b)\big)$-measurable. 

The equality in \eqref{eqn:H-Brownian-Gibbs} is $\PP$-almost surely as $\mathcal{F}_{\textrm{ext}}\big([k_1,k_2]_{\Z}\times (a,b)\big)$-measurable random variables. The Gibbs property can alternatively be stated in terms of conditional distributions. In that case, conditional on the above defined entrance data $\vec{x}$, exit data $\vec{y}$ and boundary data $(f,g)$, the law of $\mathcal{L}\vert_{[k_1,k_2]_{\Z}\times (a,b)}$ is given by $\PP^{k_1,k_2,(a,b),\vec{x},\vec{y},f,g}_{\Ham}$.
\end{definition}

\begin{definition}[Discrete $(\Hamd,\Hrw)$-Gibbs property]\label{def:Gibbs-property}
Fix an interval $\Sigma$ of $\N$ and a discrete set $\Lambda_d = d\mathbb{Z}$. Consider a $\Sigma\times\Lambda_d$-indexed line ensemble $\mathcal{L}$ defined on a probability space $(\Omega,\mathcal{B},\PP)$. For a random walk Hamiltonian $\Hrw$ and a local interaction Hamiltonian $\Hamd$ (see Definition \ref{def:RWB_Gibbs_ensemble}), we say that that the line ensemble $\mathcal{L}$ enjoys the $(\Hamd,\Hrw)$-random walk Gibbs property if for any $k_1<k_2$ with $k_1,k_2\in \Sigma$, any $a<b$ with $a,b\in \Lambda_d$ and any continuous bounded function $F:C([k_1,k_2]_{\Z}\times \Lambda_d(a,b),\R)\to \R$, $\PP$-almost surely
\begin{equation}\label{eqn:Gibbs-property}
\EE\Big[F\big(\mathcal{L}\vert_{[k_1,k_2]_{\Z}\times \Lambda_d(a,b)}\big)\big\vert \mathcal{F}_{\textrm{ext}}\big([k_1,k_2]_{\Z}\times \Lambda_d(a,b)\big)\Big]= \EE^{k_1,k_2,\Lambda_d(a,b),\vec{x},\vec{y},f,g}_{\Hamd,\Hrw}\Big[F(\mathcal{\tilde L})\Big].
\end{equation}
In the above,
\begin{equation*} 
\mathcal{F}_{\textrm{ext}}\big([k_1,k_2]_{\Z}\times \Lambda_d(a,b)\big):=\sigma\big(\mathcal{L}_k(u):(k,u)\in \Sigma\times\Lambda_d\backslash [k_1,k_2]_{\Z}\times \Lambda_d(a,b)\big)
\end{equation*}
is the sigma field generated by the (discrete) line ensemble outside $[k_1,k_2]_{\Z}\times \Lambda_d(a,b)$. The ensemble $\mathcal{\tilde L}$ on the right-hand side of \eqref{eqn:Gibbs-property} is independently drawn according to the $\PP^{k_1,k_2,\Lambda_d(a,b),\vec{x},\vec{y},f,g}_{\Hamd,\Hrw}$ measure, where we have entrance data $\vec{x} = \big(\mathcal{L}_{k_1}(a),\ldots, \mathcal{L}_{k_2}(a)\big)$, exit data $\vec{y} = \big(\mathcal{L}_{k_1}(b),\ldots, \mathcal{L}_{k_2}(b)\big)$, upper boundary curve $f=\mathcal{L}_{k_1-1}|_{\Lambda_d(a,b)}$ and lower boundary curve $g=\mathcal{L}_{k_2+1}|_{\Lambda_d(a,b)}$. By convention if $k_1-1\notin \Sigma$ then $\mathcal{L}_{k_1-1}$ is assumed to be everywhere $+\infty$ and likewise if $k_2+1\notin\Sigma$ then $\mathcal{L}_{k_2+1}$ is assumed to be everywhere $-\infty$. Note that from Lemma~\ref{lem:Zmeasurable}, the right hand side of \eqref{eqn:Gibbs-property} is $\mathcal{F}_{\textrm{ext}}\big([k_1,k_2]_{\Z}\times \Lambda_d(a,b)\big)$-measurable. 

The equality in \eqref{eqn:Gibbs-property} is $\PP$-almost surely as $\mathcal{F}_{\textrm{ext}}\big([k_1,k_2]_{\Z}\times \Lambda_d(a,b)\big)$-measurable random variables. The Gibbs property can alternatively be stated in terms of conditional distributions. In that case, conditional on the above defined entrance data $\vec{x}$, exit data $\vec{y}$ and boundary data $(f,g)$, the law of $\mathcal{L}\vert_{[k_1,k_2]_{\Z}\times \Lambda_d(a,b)}$ is given by $\PP^{k_1,k_2,\Lambda_d(a,b),\vec{x},\vec{y},f,g}_{\Hamd,\Hrw}$.
\end{definition}

\begin{remark}
An $\Ham$-Brownian bridge line ensemble naturally enjoys the $\Ham$-Brownian Gibbs property by definition. An $(\Hamd,\Hrw)$-random walk bridge line ensemble enjoys the $(\Hamd,\Hrw)$-random walk Gibbs property, which follows from the locality of the interaction Hamiltonian and the random walk bridge Hamiltonian.
\end{remark}

Just as the strong Markov property extends the Markov property to stopping times, we may (following \cite{CH14,CH16}) define stopping domains (Definition~\ref{def:stopping_time}) and appeal to the strong Gibbs property (Lemma~\ref{lem:strong_gibbs_property}). Note that in the case of discrete time $\Lambda_d$, the proof of this is considerably simpler than in continuous time (just as for the discrete versus continuous Markov processes). We do not provide the proof of this result as it is a simplified version of the proof of \cite[Lemma 2.5]{CH14}.
	
\begin{definition}\label{def:stopping_time}
Continue with the notations of Definition \ref{def:Gibbs-property}. A pair of random vairables $(l, r)$ which take values in $\Lambda_d$ is called a $[k_1,k_2]_{\Z}$-stopping domain if for all $a\leq  b$ with $a,b \in \Lambda_d$, it holds that
\begin{equation*}
\{l\leq a, r\geq b\}\in \mathcal{F}_{\textrm{ext}}\big([k_1,k_2]_{\Z}\times \Lambda_d(a,b)\big).
\end{equation*}
For a $[k_1,k_2]_{\Z}$-stopping domain $(l,r)$, let $\mathcal{F}_{\textrm{ext}}\big([k_1,k_2]_{\Z}\times \Lambda_d(l,r)\big)$ be the collection of measurable events $\mathsf{A}$ that satisfies
\begin{align*}
\mathsf{A}\cap \{l\leq a, r\geq b\}\in \mathcal{F}_{\textrm{ext}}\big([k_1,k_2]_{\Z}\times \Lambda_d(a,b)\big)
\end{align*}
for all $a\leq b$ with $a,b \in\Lambda_d$. Define the space
 \begin{equation*}
		\mathbb{C}^{k_1,k_2}:=\Big\{\big (l,r,f_{k_1},\cdots,f_{k_2}\big):l\leq r\in \Lambda_d \ \textrm{and}\ (f_{k_1},\cdots,f_{k_2})\in C([k_1,k_2]_{\Z}\times \Lambda_d [l,r],\R)\Big\}.
	\end{equation*}
We equip $\mathbb{C}^{k_1,k_2}$ with the topology induced by the restriction map $\Lambda_d\times\Lambda_d\times C([k_1,k_2]_{\Z}\times \Lambda_d,\R)\to\mathbb{C}^{k_1,k_2}$. Note that because $\Lambda_d$ is discrete, this topology is the same as the disjoint union of $C([k_1,k_2]_{\Z}\times \Lambda_d,\R)$ among different pairs of $(l, r)$.
\end{definition}

Because of the discreteness of $\Lambda_d$, the following strong $(\Hamd,\Hrw)$-Gibbs property property follows $(\Hamd,\Hrw)$-Gibbs property.
\begin{lemma}\label{lem:strong_gibbs_property}
Continuing with the notation of Definition \ref{def:Gibbs-property} and \ref{def:stopping_time}, if $(l,r)$ is a $[k_1,k_2]_\Z$-stopping domain for a line ensemble $\mathcal{L}$ which enjoys the $(\Hamd,\Hrw)$-random walk Gibbs property, then for any continuous bounded function $F:\mathbb{C}^{k_1,k_2}\to\mathbb{R} $, it holds $\PP$-almost surely that
\begin{equation}\label{equ:strong_gibbs_property}
\begin{split}
&\EE\Big[F\big(l,r,\mathcal{L}\vert_{[k_1,k_2]_{\Z}\times \Lambda_d(l,r)}\big)\big\vert \mathcal{F}_{\textrm{ext}}\big([k_1,k_2]_{\Z}\times \Lambda_d(l,r)\big)\Big]\\
=&\EE^{k_1,k_2,\Lambda_d[l,r],\vec{x},\vec{y},f,g}_{\Hamd,\Hrw}\Big[F(l,r, \mathcal{\tilde L})\Big].
\end{split}
\end{equation}
Here we have entrance data $\vec{x} = \big(\mathcal{L}_{k_1}(l),\ldots, \mathcal{L}_{k_2}(l)\big)$, exit data $\vec{y} = \big(\mathcal{L}_{k_1}(r),\ldots, \mathcal{L}_{k_2}(r)\big)$, upper boundary curve $f=\mathcal{L}_{k_1-1}|_{\Lambda_d[l,r]}$ and lower boundary curve $g=\mathcal{L}_{k_2+1}|_{\Lambda_d[l,r]}$. The ensemble $\mathcal{\tilde L}$ on the right-hand side above is independently draw according to the $\PP^{k_1,k_2,\Lambda_d[l,r],\vec{x},\vec{y},f,g}_{\Hamd,\Hrw}$ measure. From Lemma~\ref{lem:Zmeasurable}, for any $a,b\in \Lambda_d$, we have
\begin{align*}
\EE^{k_1,k_2,\Lambda_d[l,r],\vec{x},\vec{y},f,g}_{\Hamd,\Hrw}\Big[F(l,r, \mathcal{\tilde L})\Big]\cap\{l=a,r=b\}
\end{align*}
is measurable in $\mathcal{F}_{\textrm{ext}}\big([k_1,k_2]_{\Z}\times \Lambda_d(a,b)\big)$. Because $l$ and $r$ take discrete values, the right hand side of \eqref{equ:strong_gibbs_property} is measurable in $\mathcal{F}_{\textrm{ext}}\big([k_1,k_2]_{\Z}\times \Lambda_d(l,r)\big)$.
\end{lemma}

\subsection{Assumptions on $(\Hamd^N, \HrwN)$ and the main result}\label{Assumptions}
In this section we make four key assumptions on the interaction Hamiltonians and the random walk Hamiltonians. The convexity assumptions A1 and A2 are used to show the crucial monotone coupling Lemma~\ref{lem:monotone}. The assumption A3 ensures that interaction Hamiltonians approach $\Ham(x) = e^x$. The assumption A4 is about another important ingredient, the KMT type coupling between random walk bridges and Brownian bridges. The existence of such coupling for random walk bridges is the main topic studied in \cite{DW}.\\

{\raggedleft \bf Assumption A1.} $\Hamd$ satisfies the following properties: 

(1)\ $\Hamd(\vec{a})$ is non-increasing in terms of $a_1,a_2,a_3$ and is non-decreasing in terms of $a_4,a_5,a_6$. Moreover, $\Hamd(a_1,a_2,a_3,a_4,a_5,a_6)<\infty$ provided $a_1,a_2,a_3\in \mathbb{R}\cup\{\infty\}$ and $a_4,a_5,a_6\in \mathbb{R}\cup\{-\infty\}$.

(2)\ Let $\vec{a},\ \vec{b}\in (\R\cup \{\pm \infty\})^6$ and $\delta>0$. Suppose $a_i\geq b_i$ for $i=1,2,\dots, 6$ and $a_k=b_k$ for some $k\in \{1,2,\dots, 6\}$. Denote 
\begin{equation*}
a'_i=\left\{
\begin{array}{cc}
a_k+\delta & i=k\\
a_i & i\neq k
\end{array}
\right.,\ \ 
b'_i=\left\{
\begin{array}{cc}
b_k+\delta & i=k\\
b_i & i\neq k
\end{array}
\right..
\end{equation*}
Then for any $\vec{a},\vec{b}\in (\R\cup \{\pm \infty\})^6$ and any $\delta>0$, we require
\begin{align*}
-\Hamd(\vec{a}')+\Hamd(\vec{a})\geq -\Hamd(\vec{b}')+\Hamd(\vec{b}).
\end{align*}

{\raggedleft \bf Assumption A2.} The random walk Hamiltonian function $\Hrw: \mathbb{R} \rightarrow (-\infty, \infty)$ is convex. \\

For a convex Hamiltonian $\Ham$, it is known in the work of \cite{CH16} that the $\Ham$-Brownian bridge line ensemble has certain monotonicity  properties. For example, if $f$, $g$, $\vec{x}$ or $\vec{y}$ increase pointwise, then the resulting line ensemble can be coupled to the original one so as to dominate it pointwise. We extend these properties to line ensemble under Assumption A1 on $\Hamd$ and convexity of $\Hrw$. Without such convexity, the constructive proof we give for monotonicity fails. We remark that \cite{CD} involves a line ensemble which lacks this convexity. Therein they develop a new, weaker type of monotonicity (in terms of certain expectation values and up to certain constants) which turns out to be sufficient for proving tightness in the manner of \cite{CH16}.
	
We have the following lemma which allows us to couple different discrete line ensembles.

\begin{lemma}\label{lem:monotone}
Fix  $k_1\leq k_2$, $a < b \in \Lambda_d$. For $i=1,2$, fix vectors $(\vec{x}^i,\vec{y}^i)\in \mathbb{R}^{k_2-k_1+1}$, and functions $f^i:\Lambda_d(a,b)\to \mathbb{R}\cup \{\infty\}$, $g^i:\Lambda_d(a,b)\to \mathbb{R}\cup \{-\infty\}$. Suppose that for $i=1,2$, 
\begin{align*}
Z_{\Hamd,\Hrw}^{k_1,k_2,\Lambda_d(a,b),\vec{x}^i,\vec{y}^i,f^i,g^i}>0.
\end{align*}
See \eqref{def:normal_randomwalk} for the definition. For $i\in \{1,2\}$, let $\mathcal{Q}^i=\{\mathcal{Q}^i_j \}_{j=k_1}^{k_2}$ be a $[k_1,k_2]_{\mathbb{Z}}\times \Lambda_d[a,b]$-indexed line ensemble on  a probability space $(\Omega^i,\mathcal{B}^i,\mathbb{P}^i)$ where
$\mathbb{P}^i=\mathbb{P}^{k_1,k_2,\Lambda_d(a,b),\vec{x}^i,\vec{y}^i,f^i,g^i}_{\Hamd,\Hrw}$. 

Assume $\Hamd$ satisfies Assumption A1 and $\Hrw$ satisfies Assumption A2. Assume that the $i=1$ vectors and functions are pointwise greater than or equal to their $i=2$ counterparts (e.g. $f^1(u)\geq f^2(u)$ for all $u\in \Lambda_d(a,b)$). Then there exists a coupling of the probability measure
$\mathbb{P}^1$ and $\mathbb{P}^2$ such that almost surely $\mathcal{Q}^1_j(u)\leq \mathcal{Q}^2_j(u)$ for all $j\in [k_1,k_2]_{\mathbb{Z}}$ and $u\in \Lambda_d[a,b]$.
\end{lemma}

The proof of this lemma is given in Section~\ref{appendix} following the Glauber dynamics approach of \cite{CH14,CH16} which realizes the line ensemble as the invariant measure of a Markov chain on trajectories. Assumption A1 on $\Hamd$ and convexity of $\Hrw$ are sufficient conditions under which the Markov chains can be coupled, hence proving coupling of their invariant measures as well.\\

Next two assumptions are about large $N$ behavior of a sequence of interaction and random walk Hamiltonians $\{\Hamd^N,\HrwN\}_{N\in\mathbb{N}}$. Note that the underlying discrete set is also varying in $N$. Let $\{d_N\}_{N\in\mathbb{N}}$ be sequence which decreases to zero and set $\Lambda_d^N=d_N \mathbb{Z}$.\\

The assumption A3 compares $\Hamd^N$ and $\Ham(x) = e^x$. The Hamiltonian $\Ham(x) = e^x$ is particularly important because the KPZ line ensembles satisfy $\mathbf{H}$-Brownian Gibbs property \cite{CH16}. 

We first define modulus of continuity for continuous functions. Let $K\in \mathbb{N}$, $a<b$ and $r>0$. For a continuous function $ {f} =(f_1, f_2, \cdots, f_K)\in C([1,K]_{\mathbb{Z}}\times [a,b],\mathbb{R})$, the $r$-modulus of continuity is defined as
\begin{equation}\label{def:module}
\omega_{(a,b)}( {f} ,r)= \sup_{1\leq i \leq K}\sup_{
 \substack{s,t\in [a,b] \\\vert s-t \vert \leq r}}\vert f_i(u)-f_i(t)\vert.
\end{equation}
{\raggedleft \bf Assumption A3.} There exists a constant $C_1>0$ such that the following holds. Fix arbitrary $k\in\mathbb{Z}$, $a<b$ with $a,b\in\Lambda^N_d$ and $\mathcal{L}=(\cL_k,\cL_{k+1})\in C([k,k+1]_{\mathbb{Z}}\times [a,b],\mathbb{R})$, it holds that
\begin{align*}
\left|\frac{\displaystyle\sum_{u\in\Lambda^N_d (a,b)}\Hamd^N(\rectangle(\mathcal{L},k,u))}{\displaystyle\int_a^b \exp(\cL_{k+1}-\cL_k(u))du}-1 \right|\leq  e^{C_1\left( \omega_{(a,b)}(\cL ,d_N) +d_N \right)}-1.\\
\end{align*}

The last Assumption A4 is a strong (KMT) coupling between Brownian bridges and random walk bridges. Recall that $B_L:[0,L]\to\mathbb{R}$ is the standard Brownian bridge defined in Definition \ref{def:BrownianBridge}. Given $L\in d_N\mathbb{N}$ and $z\in\mathbb{R}$, we denote by $\bar{S}^N_{L,z}:[0,L] $ the random walk bridge defined in Definition~\ref{def:RW_ensemble} using $\HrwN$ and $\Lambda_d^N=d_N\mathbb{Z}$. 

Under suitable requirements on $\HrwN$, Donsker invariance principle says that $\bar{S}^N_{L,z}(u)$ converges weakly to $B_L(u)+L^{-1}z  {u}$ as $N$ goes to infinity, while KMT coupling provides a quantitative estimate for this convergence rate. For the original classical result on the case of random walks with exponential moment, see \cite{KMT} and a recent treatment for the case of random walk bridges is considered in \cite{DW}.\\

{\raggedleft \bf Assumption A4.}\ \ For any $b_1,b_2>0$ there exist constants $0<a_1,a_2<\infty$ (depending on $b_1, b_2$ but not on $N$) such that the following statement holds. For any $N\in\mathbb{N}$, any $L>0$ with $d_N^{-1} L \in\mathbb{N}$, there exists a probability space on which a Brownian bridge $B_L(u)$ and a family of random walk bridges $\{\bar{S}^N_{L,z}(u)\}_{z\in\mathbb{R}}$ are defined. For all $z\in\mathbb{R}$, one has the following estimate
\begin{align*}
\mathbb{P}\left[ \sup_{0\leq u\leq L}\left\vert B_L(u)+L^{-1}z {u} -\bar{S}^N_{L,z}(u)\right\vert\geq a_1 d_N^{1/2}\log(d_N^{-1}L) \right]\leq a_2 (d_N^{-1}L )^{-b_1}e^{b_2z^2/L}.
\end{align*}
Moreover, $\{\bar{S}^N_{L,z}(u)\}_{z\in\mathbb{R}}$ is measurable in $z$.\\

Under assumptions A1-A4, we are ready to state the main result of this paper.
\begin{theorem}\label{theorem:main}
Let $\{\Hamd^N,\HrwN\}_{N\in\mathbb{N}}$ be a sequence of interaction and random walk Hamiltonians, $\{d_N\}_{N\in\mathbb{N}}$ be a sequence which decreases to zero and let $\Lambda_d^N=d_N \mathbb{Z}$. Fix $K\in \N\cup \{\infty\}$, let $\mathcal{L}^N$ be a $[1,K]_{\Z} \times \Lambda_d^N$-indexed discrete line ensemble that enjoys the discrete $(\Hamd^N,\HrwN)$-Gibbs property. Here we adopt the convention that $\cL^N_0 = +\infty$ and $\cL^N_{K+1} = -\infty$.

Assume that $(\Hamd^N,\HrwN,\Lambda^N_d)$ satisfies assumptions A1-A4. Moreover, assume that $\mathcal{L}^N_1(u)+u^2/2$, defined through linear interpolation, converges weakly to stationary process as a $\{1\}\times\mathbb{R}$-indexed line ensembles. Then the following statements hold.
\begin{enumerate}
\item For any $T>0$ and $1 \leq k \leq K$, the restriction of the line ensemble $\mathcal{L}^N$ to $[1,k]_{\Z}\times [-T,T]$ is sequentially compact as $N$ varies. 

\item Furthermore, any subsequential limit line ensemble $\cL^{\infty}$ satisfies $\Ham$-Brownian Gibbs property with $\Ham(x)=e^x$.
\end{enumerate}
\end{theorem}

\section{Proof of Theorem \ref{theorem:main}}
We first present in Subsection~\ref{three propositions} two propositions concerning upper bounds and lower bounds of $\mathcal{L}^N$ in Theorem~\ref{theorem:main}. Next, we record a few estimates on random walk bridges and $(\Hamd^N, \HrwN)$-discrete Gibbs line ensembles in Subsection~\ref{KMT_assumption}. Then we prove Theorem \ref{theorem:main} in the last two subsections.\\


\subsection{Two key propositions}\label{three propositions}
\begin{proposition}\label{pro:1}
Continue the notations ans assumptions in Theorem~\ref{theorem:main}. Fix arbitrary $k\in [1,K]_{\mathbb{Z}}, \varepsilon\in (0,1)$ and $x_0>0$. There exists $R_k= R_k(\varepsilon)>0$ and $N_1=N_1(k,x_0,\varepsilon)\in\mathbb{N}$ such that the following holds. For any $\overline{x}\in [-x_0,x_0]$ and $N\geq N_1$, we have 
\begin{equation*}
			\mathbb{P}\Big(\inf_{u\in \Lambda_d^N[\overline{x}-1,\overline{x}+1]}\big(\mathcal{L}^N_{k}(u)+ {u^2}/{2}\big)<-R_k\Big)<\epsilon.
\end{equation*}
\end{proposition}
 
\begin{proposition}\label{pro:2}
Continue the notations ans assumptions in Theorem~\ref{theorem:main}. Fix arbitrary $k\in [1,K]_{\mathbb{Z}}, \varepsilon\in (0,1)$ and $x_0>0$. There exists $\hat{R}_k= \hat{R}_k(\varepsilon)>0$ and $N_2=N_2(k,x_0,\varepsilon)\in\mathbb{N}$ such that the following holds. For any $\overline{x}\in [-x_0,x_0]$ and $N\geq N_2$, we have 
\begin{equation*}
			\mathbb{P}\Big(\sup_{u\in \Lambda_d^N[\overline{x}-1,\overline{x}+1]}\big(\mathcal{L}^N_{k}(u)+ {u^2}/{2}\big)>\hat{R}_k\Big)<\epsilon.
\end{equation*}
\end{proposition}

The proofs of the above propositions are postponed to Section \ref{propproofs}.

\subsection{Estimates for random walk bridges and $(\Hamd^N, \HrwN)$ line ensembles}\label{KMT_assumption}
In this subsection we prove a few lemmas which we need in the proof of main Theorem~\ref{theorem:main}. Recall that $B_L$ is the Brownian bridge defined in Definition~\ref{def:BrownianBridge}. The following lemma can be found in \cite[Chapter 4, (3.40)]{KS}.
	
\begin{lemma}\label{Bbridge_sup}
For all $s>0$, it holds that
\begin{align*} 
\mathbb{P}\left(\sup_{u\in[0,L]} B_L(u) >s \right) = e^{-2s^2/L}.
\end{align*}
\end{lemma}

The next lemma is an analogue for random walk bridges which satisfy Assumption A4. 

\begin{lemma}\label{lem:sup_est}
Fix $L_2\geq L_1> 0,\ z_0\geq 0$ and $s_0\geq 1$. Suppose that $\HrwN$ satisfies Assumption A4 with respect to $\Lambda^N_d=d_N\mathbb{Z}$. Then there exists $N_0=N_0(L_1,L_2,s_0,z_0)$ such that the following holds. For any $|z|\leq z_0$, $1\leq s\leq s_0$, $L_1\leq L\leq L_2 $ and $N\geq N_0$, let $\bar{S}^N_{L,z}(u)$ be the random walk bridge defined in Definition~\ref{def:RW_ensemble} with respect to $\HrwN$ and $\Lambda^N_d$. It holds that
\begin{equation*}
\mathbb{P}\left(\sup_{u\in[0,L]} \left( \bar{S}^N_{L,z}(u)-L^{-1} z {u}  \right) >s\right)\leq e^{-s^2/L}.
\end{equation*}
\end{lemma}

\begin{proof}
 Let $b_1=b_2=1$ and $a_1,\ a_2$ be the constants determined in Assumption A4. We take $N_0$ to be the smallest number such that the following two inequalities hold for all $N\geq N_0$:
 \begin{equation}\label{equ:1225}
\begin{split}
a_1 d_N^{1/2}  \log (d_N^{-1}L_2)&\leq  1-2^{-1}  \sqrt{3} ,\\
 a_2 (d_N^{-1}L_1)^{-1}e^{z_0^2/L_1}&\leq  \min_{1\leq s\leq s_0}\left( e^{-s^2/L}-e^{-3s^2/(2L)} \right).
\end{split} 
 \end{equation}
Combining Assumption A4, Lemma \ref{Bbridge_sup}, \eqref{equ:1225} and $s\geq 1$, we get
\begin{align*}
&\mathbb{P}\left(\sup_{u\in[0,L]} \left( \bar{S}^N_{L,z}(u)-L^{-1}z {u}  \right) >s\right)\\
\leq &\mathbb{P}\left(\sup_{u\in[0,L]} B_L(u) > 2^{-1}\sqrt{3}s\right)+\mathbb{P}\left(\sup_{u\in[0,L]} \left| B_L(u) + Lz  u   - \bar{S}^N_{L,z}(u) \right|  >\left(1-2^{-1} \sqrt{3} \right) s\right)\\
\leq &e^{-s^2/L}.
\end{align*}
\end{proof}

Next, we compare the normalizing constants defined in \eqref{def:normalcont_Brownian} and \eqref{def:normal_randomwalk}. This is the content of Proposition~\ref{pro:Z_comparison}. To prove Proposition~\ref{pro:Z_comparison}, we need the following two lemmas. 
\begin{lemma}\label{change_Hamiltonian}
Let $\mathbf{H}(x)=e^x$, $\{\Hamd^N\}_{N\in  \mathbb{N}}$ be a sequence of Hamiltonians, $\{d_N\}_{N\in\mathbb{N}}$ be a sequence which decreases to $0$ and set $\Lambda^N_d=d_N\mathbb{Z}$. We assume that $(\Hamd^N,\Lambda_d^N)$ satisfies Assumption A3. There exists a constant $C_2$ depending on the constant $C_1$ in Assumption A3 such that the following holds.

 Fix arbitrary  $k_1\leq k_2$ with $k_1,k_2\in\mathbb{Z}$, $a<b$ with $a,b\in\Lambda^N_d$ and $  \tilde{\mathcal{L}}=(\mathcal{L}_{k_1-1},\mathcal{L}_{k_1  },\dots, \mathcal{L}_{k_2+1})$ in $C([k_1-1,k_2+1]_{\mathbb{Z}}\times \Lambda_d[a,b],\mathbb{R})$. We set $\vec{x}=(\mathcal{L}_{k_1}(a),\dots,\mathcal{L}_{k_2}(a))$, $\vec{y}=(\mathcal{L}_{k_1}(b),\dots,\mathcal{L}_{k_2}(b))$, $f=\mathcal{L}_{k_1-1}$, $g=\mathcal{L}_{k_2+1}$ and $\mathcal{L}=	(\mathcal{L}_{k_1},\dots,\mathcal{L}_{k_2})$. Then it holds that
\begin{equation*}
\left|W_{\Hamd^N}^{k_1,k_2,\Lambda^N_d(a,b),\vec{x},\vec{y},f,g}(\mathcal{L})-W_{\Ham}^{k_1,k_2,(a,b),\vec{x},\vec{y},f,g}(\mathcal{L})\right|\leq C_2\left(\omega_{(a,b)}(\tilde{\mathcal{L}},d_N)+d_N \right),
\end{equation*}
See \eqref{def:module} for the definition of $\omega_{(a,b)}$. 
\end{lemma}  
\begin{proof}
To simplify the notation, we denote 
\begin{align*}
W={W}_{\Ham}^{k_1,k_2,(a,b),\vec{x},\vec{y},f,g}(\mathcal{L})\ \textup{and}\ W^N={W}_{\Hamd^N}^{k_1,k_2,\Lambda^N_d(a,b),\vec{x},\vec{y},f,g}(\mathcal{L}). 
\end{align*}
From \eqref{def:Boltzmann_Brownian} and \eqref{def:Boltzmann_randomwalk}, we have
\begin{align*}
-\log{W}&=\displaystyle\sum_{ i=k_1-1}^{k_2} \displaystyle\int_a^b \exp\left( \cL_{i+1}(u)-\cL_{i}(u) \right) du,\\
-\log{W}^N&=\displaystyle\sum_{i=k_1-1}^{k_2}\displaystyle\sum_{u\in \Lambda^N_d(a,b) }\Hamd^N(\rectangle(\mathcal{L},i,u)).
\end{align*} 
Under Assumption A3, we obtain
\begin{align*}
\left| {\log{W}^N }/{\log{W}}-1 \right| &\leq e^{C_1\left( \omega_{(a,b)}(\tilde{\mathcal{L}},d_N)+d_N \right)}-1.
\end{align*}
If we further assume $\omega_{(a,b)}(\tilde{\mathcal{L}},d_N)+d_N\leq 1$, then by the mean value theorem it holds that
\begin{align*}
e^{C_1\left( \omega_{(a,b)}(\tilde{\mathcal{L}},d_N)+d_N \right)}-1\leq C_1e^{C_1}\left( \omega_{(a,b)}(\tilde{\mathcal{L}},d_N)+d_N \right).
\end{align*}
Hence
\begin{align}\label{equ:1225847}
\left| {\log{W}^N }/{\log{W}}-1 \right| \leq C_1e^{C_1}\left( \omega_{(a,b)}(\tilde{\mathcal{L}},d_N)+d_N \right).
\end{align}
Further assuming  $\omega_{(a,b)}(\tilde{\mathcal{L}},d_N)+d_N \leq   (2C_1)^{-1}e^{-C_1}$, the right hand side of \eqref{equ:1225847} is bounded from above by $2^{-1}$. Applying the elementary inequality (which we proof at the end)
\begin{align}\label{inq:exp_function}
|\alpha^{1+\beta }-\alpha |\leq |\beta| \ \ \textup{for all}\ \alpha \in (0,1]\ \textup{and}\ |\beta|\leq 2^{-1} 
\end{align}
with $\alpha = W$ and $\ 1+\beta={\log{W}^N }/{\log{W}}$, we conclude from \eqref{equ:1225847} that 
\begin{align}\label{equ:1225915}
 \left|\log W^N-\log W \right| \leq  C_1e^{C_1}\left( \omega_{(a,b)}(\tilde{\mathcal{L}},d_N)+d_N \right),
\end{align}
provided $\omega_{(a,b)}(\tilde{\mathcal{L}},d_N)+d_N \leq \min\{1, (2C_1)^{-1}e^{-C_1}\}$. 

Suppose $\left( \omega_{(a,b)}(\tilde{\mathcal{L}},d_N)+d_N \right)> \min\{1, (2C_1)^{-1}e^{-C_1}\}$. From \eqref{def:Boltzmann_Brownian} and \eqref{def:Boltzmann_randomwalk}, the Boltzmann weights take values in $[0,1]$. Therefore, 
\begin{equation}\label{equ:1225914}
\left| W_N-W \right|\leq    1 \leq   \max\{ 1,2C_1 e^{C_1} \} \left( \omega_{(a,b)}(\tilde{\mathcal{L}},d_N)+d_N \right).
\end{equation}
Combining \eqref{equ:1225915} and \eqref{equ:1225914}, the assertion holds for by taking $C_2=\max\{1,2C_1 e^{C_1}\}$.\\

It remains to prove \eqref{inq:exp_function}. For any $\alpha > 0$ and $|\beta |\leq 1/2$, from the Mean value theorem, it holds that
\begin{align*}
|\alpha ^{1+\beta }-\alpha |=|\alpha ^{1+\beta'}\log a| |\beta|
\end{align*}
for some $|\beta'|\leq 1/2$. Then \eqref{inq:exp_function} follows the fact that
\begin{align*}
\sup_{\alpha \in (0,1]} |\alpha ^{1+\beta'}\log \alpha |\leq \sup_{a\in (0,1]} |\alpha^{1/2}\log \alpha| \leq 1.
\end{align*}
\end{proof}


\begin{lemma}\label{change_ensemble}
Let $\mathbf{H}(x)=e^x$. There exists a constant $C_3$ such that the following holds.

 Fix arbitrary  $k_1\leq k_2$ with $k_1,k_2\in\mathbb{Z}$, $a<b$ and $  \tilde{\mathcal{L}}^i=(\mathcal{L}^i_{k_1-1},\mathcal{L}^i_{k_1  },\dots, \mathcal{L}^i_{k_2+1})$ in $C([k_1-1,k_2+1]_{\mathbb{Z}}\times [a,b],\mathbb{R})$ with $i=1,2$. We set $\vec{x}^i=(\mathcal{L}^i_{k_1}(a),\dots,\mathcal{L}^i_{k_2}(a))$, $\vec{y}^i=(\mathcal{L}^i_{k_1}(b),\dots,\mathcal{L}^i_{k_2}(b))$, $f^i=\mathcal{L}^i_{k_1-1}$, $g^i=\mathcal{L}^i_{k_2+1}$ and $\mathcal{L}^i=	(\mathcal{L}^i_{k_1},\dots,\mathcal{L}^i_{k_2})$. Let
 \begin{align*}
 \|\tilde{\mathcal{L}}^1-\tilde{\mathcal{L}}^2\|=\sup_{\substack{a\leq u\leq b\\k_1-1\leq k\leq k_2+1}}\left| \mathcal{L}_k^1(u)-\mathcal{L}^2_k(u) \right|.
 \end{align*}
 Then it holds that
\begin{equation*}
\left| W_{\Ham}^{k_1,k_2,(a,b),\vec{x}^1,\vec{y}^1,f^1,g^1}(\mathcal{L}^1)-W_{\Ham}^{k_1,k_2,(a,b),\vec{x}^2,\vec{y}^2,f^2,g^2}(\mathcal{L}^2) \right|\leq C_3\|\tilde{\mathcal{L}}^1-\tilde{\mathcal{L}}^2\|.
\end{equation*}

\begin{proof}
To simplify the notation, we denote 
\begin{align*}
W^i=W_{\Ham}^{k_1,k_2,(a,b),\vec{x}^i,\vec{y}^i,f^i,g^i}(\mathcal{L}^i). 
\end{align*}
For any $ k\in [k_1-1,k_2]_{\mathbb{Z}}$, we have 
\begin{align*}
\left|\frac{\displaystyle \int_a^b \exp(\cL^1_{k+1}(u)-\cL^1_{k}(u))du}{\displaystyle\int_a^b \exp(\cL^2_{k+1}(u)-\cL^2_{k}(u))du}-1\right| &\leq \left|\sup_{a\leq u\leq b }\frac{ \exp(\cL^1_{k+1}(u)-\cL^1_{k}(u))}{\exp(\cL^2_{k+1}(u)-\cL^2_{k}(u))}-1\right| \leq  e^{2 \| \tilde{\mathcal{L}}^1-\tilde{\mathcal{L}}^2 \|}-1.
\end{align*}
Together with
\begin{align*}
-\log W^i =\displaystyle\sum_{ k=k_1-1}^{k_2} \displaystyle\int_a^b \exp\left( \cL^i_{k+1}(u)-\cL^i_{k}(u) \right) du,
\end{align*}
we have 
\begin{align*}
\left| \log W^1/\log W^2-1 \right|\leq e^{2 \| \tilde{\mathcal{L}}^1-\tilde{\mathcal{L}}^2 \|}-1.
\end{align*}
If we assume $\| \tilde{\mathcal{L}}^1-\tilde{\mathcal{L}}^2 \|\leq 1,$ by the mean value theorem, it holds that
\begin{align*}
\left| \log W^1/\log W^2-1 \right|\leq 2e^2 \| \tilde{\mathcal{L}}^1-\tilde{\mathcal{L}}^2 \|.
\end{align*}
Further assuming that $\|\tilde{\mathcal{L}}^1-\tilde{\mathcal{L}}^2\|\leq 4^{-1}e^{-2}$, we can apply \eqref{inq:exp_function} with $\alpha =W^2$ and $1+\beta=\log W^1/\log W^2,$ to get
\begin{align*}
|W^1-W^2|\leq 2e^2 \| \tilde{\mathcal{L}}^1-\tilde{\mathcal{L}}^2 \|.
\end{align*}

Suppose $\|\tilde{\mathcal{L}}^1-\tilde{\mathcal{L}}^2\|> 4^{-1} e^{-2}$. We can simply use
\begin{align*}
\left|W^1-W^2 \right|
\leq 1 \leq 4e^2\|\tilde{\mathcal{L}}^1-\tilde{\mathcal{L}}^2\|.
\end{align*}
Thus the desired result follows by taking $C_3=4e^2$.
\end{proof}
\end{lemma}

The following proposition shows under suitable assumptions, the normalizing constants $Z_{\Hamd^N,\HrwN}^{k_1,k_2,\Lambda^N_d(a,b),\vec{x},\vec{y},f,g}$ converge to $Z_{\Ham}^{k_1,k_2,(a,b),\vec{x},\vec{y},f,g}$ with $\Ham(x)=e^x$. See Definitions~\ref{def:H_Brownian} and \ref{def:RWB_Gibbs_ensemble} for the definitions of these normalizing constants.
\begin{proposition}\label{pro:Z_comparison}
Let $\mathbf{H}(x)=e^x$, $\{\Hamd^N,\HrwN\}_{N\in\mathbb{N}}$ be a sequence of interaction and random walk Hamiltonians, $\{d_N\}_{N\in\mathbb{N}}$ be a sequence which decreases to zero and let $\Lambda_d^N=d_N \mathbb{Z}$. Suppose $(\Hamd^N,\HrwN,\Lambda^N_d)$ satisfy Assumption A3 and Assumption A4. Then for any  $k_0\in\mathbb{N}$, $0<L_1<L_2$, $z_0>0$ and $\epsilon>0$, there exists $N_0$ and $\delta>0$ such that the following statement holds.

Fix arbitrary $N\in\mathbb{N}$ with $N\geq N_0$, $k_1, k_2\in\mathbb{Z}$ with $0\leq k_2-k_1\leq k_0$, $a,b \in\Lambda_d^N $ with $L_1\leq b-a\leq L_2$, $\vec{x}=(x_{k_1},\dots x_{k_2}), \vec{y}=(y_{k_1},\dots y_{k_2}), $ with $ \sup_{k_1\leq k\leq k_2}|x_k-y_k|\leq z_0$ and two continuous functions $f,g\in C([a,b],\mathbb{R})$ with $\omega_{(a,b)}(f,d_N)+\omega_{(a,b)}(g,d_N)<\delta $. 

Let $B:[k_1,k_2]_{\mathbb{Z}}\times [a,b]\to \mathbb{R}$ and $\bar{S}^N:[k_1,k_2]_{\mathbb{Z}}\times \Lambda^N_d[a,b]\to \mathbb{R}$ be the line ensembles with laws $\mathbb{P}^{k_1,k_2,(a,b),\vec{x},\vec{y}}_{\free}$ and $\PP_{\free,\HrwN}^{k_1,k_2,\Lambda^N_d(a,b),\vec{x},\vec{y}}$ respectively. (See Definitions~\ref{def:H_Brownian} and \ref{def:RW_ensemble}.)  Then $\bar{S}^N$ and $B$ can be coupled in one probability space. Moreover, given two events $\mathsf{J}$ and $\mathsf{J}'$ with $\mathbb{P}(\mathsf{J}\Delta \mathsf{J}')<\epsilon'$, we have

\begin{equation*}
\left| \mathbb{E}\left[ W_{\Hamd^N}^{k_1,k_2,\Lambda^N_d(a,b),\vec{x},\vec{y},f,g}(\bar{S}^N) \mathbb{1}_{\mathsf{J}} \right]-\mathbb{E}\left[ W_{\Ham}^{k_1,k_2,(a,b),\vec{x},\vec{y},f,g}(B) \mathbb{1}_{\mathsf{J}'} \right] \right|<\epsilon+\epsilon'.
\end{equation*}
In particular, by taking $\mathsf{J}=\mathsf{J}'$ to be the whole probability space, we have

\begin{equation*}
\left| Z_{\Hamd^N,\HrwN}^{k_1,k_2,\Lambda^N_d(a,b),\vec{x},\vec{y},f,g}-Z_{\Ham}^{k_1,k_2,(a,b),\vec{x},\vec{y},f,g}  \right|<\epsilon.
\end{equation*}
\end{proposition}

\begin{proof}
Let $\delta_1, \delta_2$ be two small numbers to be determined. By taking $b_1=1$ and $b_2=1$ in Assumption A4, we can couple $\bar{S}^N $ and $B$ in the same probability space such that for each $k\in [k_1,k_2]_{\mathbb{Z}}$,
\begin{equation}\label{equ:1226220}
 \mathbb{P}\left( \sup_{a\leq u\leq b}\left| B_k (u)-\bar{S}^N_k(u) \right|>a_1 d_N^{1/2}\log ( d_N^{-1} L_2 ) \right)\leq  a_2(d_N^{-1} L_1)^{-1}e^{z_0^2/L_1}.
\end{equation}

Define the events
\begin{align*}
\mathsf{A}_{1,N}=&\{ \omega_{(a,b)}(B, d_N) <\delta_1\},\\
\mathsf{A}_{2,N}=&\left\{ \sup_{a\leq u\leq b, k_1\leq k\leq k_2} \left| B_k(u)-\bar{S}^N_k(u) \right| <\delta_2\right\},\\
\mathsf{A}_N=&A_{1,N}\cap A_{2,N}.
\end{align*}
Take $N_0$ large enough such $a_1d_{N_0}^{1/2}\log (d_{N_0}^{-1} L_2)<\delta_2$. Then \eqref{equ:1226220} implies for all $N\geq N_0$, it holds that $\mathbb{P}(\mathsf{A}^{\neg}_{2,N})\leq k_0    a_2(d_{N_0}^{-1} L_1)^{-1}e^{z_0^2/L_1}.$ Hence through taking $N_0$ large enough, we have $\mathbb{P}(\mathsf{A}^{\neg}_{2,N})\leq \varepsilon.$ Also, for $N_0$ large enough depending on $L_2, z_0, k_0$ and $\delta_1$, we have $\mathbb{P}(\mathsf{A}^{\neg}_{1,N})\leq \varepsilon$ for all $N\geq N_0$. In short, we obtain that 
\begin{equation}\label{equ:1226225}
\mathbb{P}(\mathsf{A}^{\neg}_{N})\leq 2\varepsilon. 
\end{equation}
When the event $\mathsf{A}_N$ occurs, it holds that $\omega_{(a,b)}(\bar{S}^N,d_N)\leq \delta_1+2\delta_2.$ Hence, by applying Lemma \ref{change_Hamiltonian} and \ref{change_ensemble}, we get
\begin{align*}
&\mathbbm{1}_{\mathsf{A}_N} \left|W_{\Hamd^N}^{k_1,k_2,\Lambda^N_d(a,b),\vec{x},\vec{y},f,g}(\bar{S}^N)-W_{\Ham}^{k_1,k_2,(a,b),\vec{x},\vec{y},f,g}(B)\right|\\
\leq & \mathbbm{1}_{\mathsf{A}_N} \left|W_{\Hamd^N}^{k_1,k_2,\Lambda^N_d(a,b),\vec{x},\vec{y},f,g}(\bar{S}^N)-W_{\Ham}^{k_1,k_2,(a,b),\vec{x},\vec{y},f,g}(\bar{S}^N)\right|\\
&+ \mathbbm{1}_{\mathsf{A}_N}\left|W_{\Ham}^{k_1,k_2,(a,b),\vec{x},\vec{y},f,g}(\bar{S}^N)-W_{\Ham}^{k_1,k_2,(a,b),\vec{x},\vec{y},f,g}(B)\right|\\
\leq & \mathbbm{1}_{\mathsf{A}_N}\big(C_2(\delta+\delta_1+2\delta_2+d_N)+C_3\delta_2\big).
\end{align*}
By choosing $\delta$, $\delta_1$, $\delta_2$ and $d_{N_0}$ small enough, we have
\begin{align}\label{equ:1226231}
 \mathbbm{1}_{\mathsf{A}_N} \left|W_{\Hamd^N}^{k_1,k_2,\Lambda^N_d(a,b),\vec{x},\vec{y},f,g}(\bar{S}^N)-W_{\Ham}^{k_1,k_2,(a,b),\vec{x},\vec{y},f,g}(B)\right|<\mathbbm{1}_{\mathsf{A_N}} \epsilon.
\end{align}
Combining \eqref{equ:1226225} and \eqref{equ:1226231} and the fact that Boltzmann weights take values in $[0,1]$, we conclude that 
\begin{align*}
&\left| \mathbb{E}\left[ W_{\Hamd^N}^{k_1,k_2,\Lambda^N_d(a,b),\vec{x},\vec{y},f,g}(\bar{S}^N) \mathbb{1}_{\mathsf{J}} \right]-\mathbb{E}\left[ W_{\Ham}^{k_1,k_2,(a,b),\vec{x},\vec{y},f,g}(B) \mathbb{1}_{\mathsf{J}'} \right] \right|\\
\leq & \left| \mathbb{E}\left[ W_{\Hamd^N}^{k_1,k_2,\Lambda^N_d(a,b),\vec{x},\vec{y},f,g}(\bar{S}^N) \mathbb{1}_{\mathsf{J}\cap \mathsf{J}'\cap \mathsf{A}_N} \right]-\mathbb{E}\left[ W_{\Ham}^{k_1,k_2,(a,b),\vec{x},\vec{y},f,g}(B)  \mathbb{1}_{\mathsf{J}\cap \mathsf{J}'\cap \mathsf{A}_N} \right] \right|\\
+&\left| \mathbb{E}\left[ W_{\Hamd^N}^{k_1,k_2,\Lambda^N_d(a,b),\vec{x},\vec{y},f,g}(\bar{S}^N) \mathbb{1}_{\mathsf{J}\cap \mathsf{J}'\cap \mathsf{A}_N^{\neg}} \right]\right|+\left| \mathbb{E}\left[ W_{\Ham}^{k_1,k_2, (a,b),\vec{x},\vec{y},f,g}(B) \mathbb{1}_{\mathsf{J}\cap \mathsf{J}'\cap \mathsf{A}_N^{\neg}} \right]\right|\\
+&\left| \mathbb{E}\left[ W_{\Hamd^N}^{k_1,k_2,\Lambda^N_d(a,b),\vec{x},\vec{y},f,g}(\bar{S}^N) \mathbb{1}_{\mathsf{J} \setminus \mathsf{J}'} \right]\right|+\left| \mathbb{E}\left[ W_{\Ham}^{k_1,k_2,(a,b),\vec{x},\vec{y},f,g}(B) \mathbb{1}_{\mathsf{J}' \setminus \mathsf{J}} \right]\right|\\
\leq &5\epsilon+\epsilon'.
\end{align*}
\end{proof}

\subsection{Proof of Theorem \ref{theorem:main} (1)}\label{proof of main_thm (1)}
In this subsection, we prove Theorem~\ref{theorem:main} (1). The core of the proof is an estimate on the normalizing constants for $(\Hamd^N,\HrwN)$-Gibbs line ensembles, which we show in the lemma below. The analogous result for $\Ham$-Brownian Gibbs line ensembles in proved in \cite[Proposition 6.4]{CH16}.

 \begin{lemma}\label{lemma:Z}
 Continue the notations ans assumptions of Theorem~\ref{theorem:main}. Fix $k_1\leq k_2$ with $k_1,k_2\in [1,K]_{\mathbb{Z}}$, $0<L_1< L_2$ and $\varepsilon>0$. Then there exists $\delta=\delta(k_2,L_1,L_2,\varepsilon) $, and $N_0=N_0(k_2,L_1,L_2,\epsilon)$ such that the following holds.
 
Fix arbitrary $N\geq  N_0$, and $a,b\in\Lambda^N_d[-L_2,L_2]$ with $L_1\leq b-a$. We have
 	\begin{equation*}
 		\mathbb{P}\Big(Z^{k_1,k_2,\Lambda_d^N(a,b),\vec{x},\vec{y},f,g}_{\Hamd^N,\HrwN}<\delta \Big)<\epsilon,
 	\end{equation*}
 	where $\vec{x}=(\mathcal{L}^N_i(a))_{i=k_1}^{k_2}$, $\vec{y}=(\mathcal{L}^N_i(b))_{i=k_1}^{k_2}$, $f=\mathcal{L}^N_{k_1-1}$, $g=\mathcal{L}^N_{k_2+1}$.
 \end{lemma}
 
\begin{proof}
Propositions \ref{pro:1} and \ref{pro:2} imply that there exists $M>0$, depending on  $k_2,L_2$ and $\epsilon$, such that the event
	\begin{equation*}
\begin{split}
\mathsf{E}=&\Big\{ \min_{u\in \Lambda_d^N[a,b]}\mathcal{L}_{k_1-1}^N(u) >-M \Big\} \cap \Big\{ \max_{u\in \Lambda_d^N[a,b]}\mathcal{L}_{k_2+1}^N(u) <M \Big\}\\
	 &\cap \Big\{\vert \mathcal{L}^N_i(u) \vert \leq M, \ \textup{for all}\  u\in\{a,b\},\ i\in [k_1,k_2]_{\mathbb{Z}} \Big\} 
 \end{split}
	\end{equation*}
has probability $\mathbb{P}(\mathsf{E})\geq 1- \epsilon $.
	
Let $\delta>0$ be a small number to be specified soon. Define the event (with $\vec{x}, \vec{y}, f, g$ as in the statement of the lemma)
	\begin{equation*}
 \mathsf{D}=\Big\{ Z^{k_1,k_2,\Lambda_d^N(a,b),\vec{x},\vec{y},f,g}_{\Hamd^N,\HrwN}<\delta \Big\}.
	\end{equation*}
Since $\mathbb{P}(\mathsf{E}^{\neg})< \epsilon$ and $\mathbb{P}(\mathsf{D})\leq \mathbb{P}(\mathsf{D}\cap \mathsf{E}) +\mathbb{P}(\mathsf{E}^{\neg})$, it suffice to show that $ \mathbb{P}(\mathsf{D} \cap \mathsf{E} )=0$.\\

Due to the monotonicity of Assumption A1 (1), we have
 \begin{equation*}
 \mathbbm{1}_{\mathsf{E}} Z^{k_1,k_2,\Lambda_d^N(a,b),\vec{x},\vec{y},f,g}_{\Hamd^N,\HrwN}\geq \mathbbm{1}_{\mathsf{E}} Z^{k_1,k_2,\Lambda_d^N(a,b),\vec{x},\vec{y},-M,M}_{\Hamd^N,\HrwN}.
 \end{equation*}
 Let $\delta=2^{-1}\inf Z^{k_1,k_2,(a,b),\vec{x}',\vec{y}',-M,M}_{\Ham},$ where the infimum is taking over $\vec{x}',\vec{y}'\in [-M,M]^{k_2-k_1+1}$. By Proposition~\ref{pro:Z_comparison}, there exists $N_0$ depending on $k_2, L_1,L_2, M$ and $\delta$ such that for $N\geq N_0$,
 \begin{align*}
\mathbbm{1}_{\mathsf{E}} \left|  Z^{k_1,k_2,\Lambda_d^N(a,b),\vec{x},\vec{y},-M,M}_{\Hamd^N,\HrwN}-Z^{k_1,k_2,(a,b),\vec{x},\vec{y},-M,M}_{\Ham}\right|\leq \delta\mathbbm{1}_{\mathsf{E}}.  
\end{align*}
Under the above arrangement, $ \mathsf{D} \cap \mathsf{E}=\varnothing$. 
\end{proof}

Before giving the proof of Theorem \ref{theorem:main}(1),  let us recall a tightness criterion for random continuous functions. For a continuous function $f =(f_1,f_2,\dots,f_k)$ in $C([1,k]_{\mathbb{Z}}\times [a,b],\mathbb{R})$ and $r>0$, $\omega_{(a,b)}(f,r)$ defined in \eqref{def:module} measures the modulus of continuity of $f$. Consider a sequence of probability measures $\mathbb{P}_N$ on $C([1,k]_{\mathbb{Z}}\times [a,b],\mathbb{R})$ and define event
\begin{equation}\label{def:eventU}
\mathsf{U}_{[a,b]} (\rho,r )=\Big\{\omega_{(a,b)}\big( {f} ,r\big)\leq \rho\Big\}.
\end{equation}
As an immediate generalization of \cite[Theorem 8.2]{Bi}, a sequence $\mathbb{P}_N$ of probability measures on $C([1,k]_{\mathbb{Z}}\times [a,b],\mathbb{R})$ is tight if, for each $i\in [1,k]_{\mathbb{Z}}$, the one-point distribution of $f_i(u_0)$ at a fixed $u_0\in[a,b]$ is tight and if, for each positive $\rho$ and $\eta$, there exists $r>0$ and integer $N_0$ such that for $N\geq N_0$,
\begin{equation*}
\mathbb{P}_N \big(\mathsf{U}_{[a,b]} (\rho,r ) \big)\geq 1-\eta.
\end{equation*}
This tightness criterion will be applied to $\mathcal{L}^N$ restricted on a compact interval.\\

\begin{proof}[Proof of Theorem~\ref{theorem:main} (1)]
We present only the argument of the case $1\leq k<K$. The case $k=K$ can be proved analogously. Fix $T>0$. Recall that $\{\cL^N\}_{N\in\mathbb{R}}$ is a sequence of $[1,K]_{\mathbb{Z}}\times \Lambda^N_d$-indexed line ensemble which has $(\Hamd^N,\HrwN)$-Gibbs property. We denote $\mathbb{P}_N$ $\mathbb{E}_N$ as the corresponding probability measures on $C([1,k]_{\mathbb{Z}}\times [-T,T],\mathbb{R})$. We denote $\mathbb{E}_N$ as the expectations of $\mathbb{P}_N$. We aim to show that $\mathbb{P}_N$ is tight in $N$.\\ 

From Propositions \ref{pro:1} and \ref{pro:2}, the sequence of random variables $\{\mathcal{L}^N_i(0) \}_{N\in\mathbb{N}}$ is tight in $N$ for all  $i\in [1,k]_{\mathbb{N}} $. This gives the one point tightness. It remains to verify that, for any $\rho,\eta>0$, there exist $r$ and $N_0$ depending on $k,T,\rho$ and $r$ such that for $N\geq N_0 $, it holds that
	\begin{equation}\label{eq:P_N}
	\mathbb{P}_N\big(\mathsf{U}_{[-T,T]} (\rho,r )\big)\geq 1-\eta,
	\end{equation}
where $\mathsf{U}_{[-T,T]} (\rho,r )$ is defined in \eqref{def:eventU}.\\
	
For $M>0$, we define the event
\begin{equation*}
	\mathsf{S}_{N,M}=\bigcap_{i=1}^k \big\{ \mathcal{L}^N_i(\pm T)\in [-M,M]\ \textup{for all}\ i\in [1,k]_{\mathbb{Z}}  \big\}.
 \end{equation*}
Propositions \ref{pro:1} and \ref{pro:2} imply there exist $M$ and $N_0$ depending on $k,T$ and $\eta$ such that 
\begin{equation}\label{equ:SNM}
 \mathbb{P}(\mathsf{S}_{N,M}^{\neg})\leq 4^{-1}\eta.
\end{equation} 
Let $\vec{x}^N=(\mathcal{L}^N_1(T),\dots,\mathcal{L}^N_k(-T))$, $\vec{y}^N=(\mathcal{L}^N_1(T),\dots,\mathcal{L}^N_k(T))$ and $g^N(u)=\mathcal{L}^N_{k_2+1}(u)$. We denote $Z_N$ as a shorthand for
\begin{equation}\label{equ:def_someZN} 
Z_N :=Z_{\Hamd^N,\HrwN}^{1,K,\Lambda_d^N(-T,T),\vec{x}^N,\vec{y}^N,\infty,g^N}.
\end{equation}
By Lemma \ref{lemma:Z}, there exists $\delta$ and $N_0$ depending on $k,T$ and $\eta$ such that for all $N>N_0$, it holds that
\begin{equation}\label{equ:Zeta}
  \mathbb{P}_N(\{Z_N<\delta\})\leq 4^{-1}\eta .
\end{equation}

It is enough to prove that for any $\rho, \eta >0$, there exists $M, r$ and $N_0$ depending on $\rho,\eta$ such that for all $N\geq N_0$, it holds that
\begin{equation}\label{eq:W(rho,eta)}
\mathbb{P}_N\Big(\mathsf{U}_{[-T, T]}(\rho,r)\cap\{Z_N\geq \delta\}\cap \mathsf{S}_{N,M}\Big)> 1-\eta.
\end{equation}
	
Observe that the event $\{Z_N\geq \delta\}\cap \mathsf{S}_{N,M}$ is $\mathcal{F}_{\textrm{ext}}([1,K]_{\mathbb{Z}}\times\Lambda_d^N(-T,T))$-measurable. This implies the left-hand side of (\ref{eq:W(rho,eta)}) equals
\begin{equation}\label{eq.noleft}
	\mathbb{E}_N\Big[\mathbbm{1}_{Z_N\geq \delta}\mathbbm{1}_{\mathsf{S}_{N,M}}\mathbb{E}_N\big[\mathbb{1}_{\mathsf{U}_{[-T,T]}(\rho,r)}\big\vert \mathcal{F}_{\textrm{ext}}([1,K]_{\mathbb{Z}}\times\Lambda_d^N(-T,T))\big]\Big].
	\end{equation}
Denote \begin{equation}\label{eq.Pnn}
 {\tilde{\mathbb{P}}}_N:={\mathbb{P}}^{1,K,\Lambda_d^N(-T,T),\vec{x}^N,\vec{y}^N,+\infty,g^N}_{\Hamd^N,\HrwN}.
\end{equation}
From the $(\Hamd^N,\HrwN)$-Gibbs property enjoyed by $\mathcal{L}^N$, it holds $\mathbb{P}_N$-almost surely that,
	\begin{align}\label{eq:cond_W}
	\mathbb{E}_N\big[\mathbb{1}_{\mathsf{U}_{[-T,T]}(\rho,r)}\big\vert \mathcal{F}_{\textrm{ext}}([1,K]_{\mathbb{Z}}\times\Lambda_d^N(-T,T))\big]={\tilde{\mathbb{P}}}_N\big(\mathsf{U}_{[-T,T]}(\rho,r)\big).
	\end{align}
\begin{lemma}\label{lem:u(-T,T)}
Fix arbitrary positive numbers $\rho,\eta,\delta$ and $M$. There exists $r$ and $N_0$ depending on $k,T, \rho,\eta,\delta$ and $M$ such that for all $N\geq N_0$,
	\begin{equation*}
	\mathbbm{1}_{Z_N\geq \delta}\mathbbm{1}_{\mathsf{S}_{N,M}} \tilde{\mathbb{P}}_N\big(\mathsf{U}_{[-T,T]}(\rho,r)\big)\geq (1-2^{-1}\eta)\mathbbm{1}_{Z_N\geq \delta}\mathbbm{1}_{\mathsf{S}_{N,M}}.
	\end{equation*}
	\end{lemma}
Let us assume Lemma~\ref{lem:u(-T,T)} holds for the moment and proceed to derive \eqref{eq:W(rho,eta)}. Recall that $\delta$ and $M$ are determined by $k,T$ and $\eta$ such that \eqref{equ:SNM} and \eqref{equ:Zeta} hold. Combining Lemma \ref{lem:u(-T,T)} and \eqref{eq:cond_W}, there exists $r$ and $N_0$ depending on  $k,T,\rho$ and $\eta$ such that for all $N\geq N_0$, it holds that
 \begin{equation}\label{equ:3.14}
	\eqref{eq.noleft} \geq (1-2^{-1}\eta )\mathbb{E}_N[\mathbb{1}_{Z_N\geq \delta}\mathbb{1}_{\mathsf{S}_{N,M}}].
 \end{equation}
Combining \eqref{equ:SNM}, \eqref{equ:Zeta} and \eqref{equ:3.14}, we conclude that 
\[
\eqref{eq.noleft} \geq (1-\eta/2)^2 > 1-\eta.
\]
This completes the proof of \eqref{eq:W(rho,eta)}.
\end{proof}  


\begin{proof}[Proof of Lemma \ref{lem:u(-T,T)}]
The proof is based on the KMT coupling (Assumption A4) and the modulus of continuity bound for free Brownian bridges.

Recall that $\tilde{\mathbb{P}}_N$ is defined in \eqref{eq.Pnn}. We write $\tilde{\mathbb{P}}_{\free,N}$ for ${\mathbb{P}}^{1,K,\Lambda_d^N(-T,T),\vec{x}^N,\vec{y}^N}_{\free,\HrwN}$ and $\tilde{\mathbb{E}}_{\free,N}$ for its expectation. Also we denote $W_N$ as the $N$-th Boltzmann weight corresponding to $(\Hamd^N,\Lambda^N_d)$. See Definition~\ref{def:RWB_Gibbs_ensemble} for more details.
From Definition~\ref{def:RWB_Gibbs_ensemble}, we have
\begin{align*}
\tilde{\mathbb{P}}_N\big(\mathsf{U}^{\neg}_{[-T,T]}(\rho,r)\big)& = Z_N^{-1}   {\tilde{\mathbb{E}}}_{\free,N} [\mathbb{1}_{\mathsf{U}^{\neg}_{[-T,T]}(\rho,r)} W_N].  
\end{align*}
Here $Z_N$ is given in \eqref{equ:def_someZN}. Because $W_N$ takes values in $[0,1]$, we have
\begin{align}\label{{modulus_P_upper}}
\mathbbm{1}_{Z_N\geq \delta} \tilde{\mathbb{P}}_N\big(\mathsf{U}^{\neg}_{[-T,T]}(\rho,r)\big)  \leq   \delta^{-1} \tilde{\mathbb{P}}_{\free,N}\big(\mathsf{U}^{\neg}_{[-T,T]}(\rho,r)\big)\mathbbm{1}_{Z_N\geq \delta}.
\end{align}
Therefore, the proof of Lemma~\ref{lem:u(-T,T)} is reduced to show that for any $\rho, \eta,\delta$ and $M$, there exit $r$ and $N_0$ such that for all $N \geq N_0$, it holds that
\begin{equation}\label{equ:1226511}
 \mathbbm{1}_{\mathsf{S}_{N,M}} \tilde{\mathbb{P}}_{\free,N}\big(\mathsf{U}^{\neg}_{[-T,T]}(\rho,r)\big)  \leq 2^{-1}\delta \eta \mathbbm{1}_{\mathsf{S}_{N,M}}. 
\end{equation}

The rest of the proof is devoted to show \eqref{equ:1226511}. From now one we assume each component of $\vec{x}^N$ and $\vec{y}^N$ is contained in $[-M,M]$. We write $\bar{S}^N=(\bar{S}^N_1,\dots,\bar{S}^N_k)$ for random curve distributed as $\tilde{\mathbb{P}}_{\free,N}$. We denote $B=(B_1,B_2,\dots, B_k)$ as Brownian bridges distributed as $\mathbb{P}^{1,k,(-T,T),\vec{x}^N,\vec{y}^N}_{\free}$. See Definition~\ref{def:H_Brownian} for detailed definition. We will use $\mathsf{U}_{[-T,T]}(\bar{S}^N, \rho,r)$ and $\mathsf{U}_{[-T,T]}(B,\rho,r)$ to denote the event defined in \eqref{def:eventU} for $\bar{S}^N$ or $B$ respectively.\\ 

From Assumption A4, $\bar{S}^N$ and $B$ can be coupled in one probability space and we write $\mathbb{P}^N_{\textup{cpl}}$ for the coupling measure. Take $b_1=1$ and $b_2=1$ in Assumption  A4. There exist $a,c >0$ depending on $T,M$ and $k$ such that 
\begin{equation*}
\mathbb{P}^N_{\cpl}\left(\mathsf{D}_N^{\neg} \right) \leq cd_N,
\end{equation*}
where
\begin{align*}
\mathsf{D}_N=\left\{ \sup_{i\in [1,k]_{\mathbb{Z}}, u\in[-T,T]} \left\vert \bar{S}^N_i(u)-B(u)\right\vert < ad_N^{1/2}\log (2d_N^{-1}T) \right\}.
\end{align*}

We can now fix $N_0$. Let $N_0$ be the smallest integer such that for all $N\geq N_0$, we have $ad_N^{1/2}\log (2d_N^{-1}T)< 4^{-1}\rho$ and $cd_N <4^{-1}\delta\eta$. This implies that
\begin{align}\label{equ:1226554}
 \mathsf{D}_N\cap  \mathsf{U}^{\neg}_{[-T,T]}(\bar{S}^N,\rho,r) \subset \mathsf{U}^{\neg}_{[-T,T]}(B,2^{-1}\rho,r). 
\end{align}
and that
\begin{equation}\label{equ:1226543}
\mathbb{P}^N_{\cpl}(\mathsf{D}_N^{\neg}) <   4^{-1}  \delta\eta .
\end{equation}

Next, we fix the value of $r$. From the modulus of continuity estimate for free Brownian bridges, there exists $r$ depending on $k,T, \rho,\eta,\delta$ and $M$ such that
\begin{align}\label{BBridge_modulus}
\mathbb{P}^N_{\cpl}\left(U^{\neg}_{[-T,T]}\left(B, 2^{-1}  \rho  ,r\right)\right)\leq 4^{-1}  {\delta\eta}.
\end{align}

Combining \eqref{equ:1226554}, \eqref{equ:1226543} and \eqref{BBridge_modulus}, we get
\begin{align*}
\mathbb{P}^N_{\cpl}\big(\mathsf{U}^{\neg}_{[-T,T]}(\bar{S}^N, \rho,r)\big)\leq \mathbb{P}^N_{\cpl}\left( \mathsf{D}^{\neg}_N  \right)+\mathbb{P}^N_{\cpl}\big(\mathsf{U}^{\neg}_{[-T,T]}(B, 2^{-1}\rho,r)\big)\leq 2^{-1}\delta\eta.
\end{align*}
This finishes the proof of \eqref{equ:1226511}.
\end{proof}
\subsection{Proof of Theorem \ref{theorem:main} (2)}\label{proof of main_thm (2)}
In this subsection, we demonstrate that all sub-sequential limiting line ensembles enjoy the {H}-Brownian Gibbs property.

\begin{proof}[Proof of Theorem \ref{theorem:main} (2)] Fix $T>0$ and $k\in [1,K]_{\mathbb{Z}}$. Without loss of generality, we assume that $d_N$ is summable and that $\mathcal{L}^N|_{[1,k]_{\mathbb{Z}}\times [-T,T]}$ converges weakly to a line ensemble $\mathcal{L}^\infty|_{[1,k]_{\mathbb{Z}}\times [-T,T]}$. See Definition~\ref{def:CONVERGE} for the definition of the weak convergence. We denote will denote them by $\mathcal{L}^N$ and $\mathcal{L}^\infty$ for simplicity. Fix an index $i\in [1,k-1]_{\mathbb{Z}}$ and two times $a,b\in [-T,T]$ with $a<b$. We will show that the law of $\mathcal{L}^{\infty}$ is unchanged if $\mathcal{L}^{\infty}_i$ is resampled between $a$ and $b$ according to the law $\mathbb{P}_{\mathbf{H}}^{i,i,(a,b),x^ ,y^\infty ,f^\infty ,g^\infty }$ with $x^\infty=\mathcal{L}^\infty_i(a), y^\infty=\mathcal{L}^\infty_i(b)$, $f^\infty=\mathcal{L}^\infty_{i-1}$, $g^\infty=\mathcal{L}^\infty_{i+1} $ and $\mathbf{H}=e^x$. The argument can easily be generalized to multiple consecutive curves. Note that the {H}-Brownian Gibbs property is equivalent to this resampling invariance, hence finishing the proof.

Since the Banach space $C([1,k]_{\mathbb{Z}}\times [-T,T],\mathbb{R})$ equipped with supremum norm is separable, the Skorohod representation theorem applies. Therefore there exists a probability space $(\Omega,\mathcal{B},\mathbb{P})$ on which all of $\mathcal{L}^N$ for $N\in\mathbb{N}\cup\{\infty\}$ are defined and almost surely $\mathcal{L}^N(\omega)\to\mathcal{L}^\infty(\omega)$ in the topology of $C([1,k]_{\mathbb{Z}}\times [-T,T],\mathbb{R})$.

Let $L=b-a$. Recall that $\bar{S}^N_{L,z}$ is the random walk bridge defined in in Definition~\ref{def:RW_ensemble} using $\HrwN$ and $\Lambda_d^N=d_N\mathbb{Z}$. From Assumption A4, there exists a probability space $(\Omega_{\textup{cpl}}^N,\mathcal{B}_{\textup{cpl}}^N,\mathbb{P}_{\textup{cpl}}^N)$ on which all of random walk bridges $\bar{S}^N_{L,z},\ z\in\mathbb{R}$ and a Brownian bridge $B_{L}$ are defined. $\{\bar{S}^N_{L,z}(u)\}_{z\in\mathbb{R}}$ is measurable in $z$. Moreover, by taking $b_1=1$ and $b_2=1$ in  Assumption A4, there exit  $0<a_1,a_2<\infty$ such that
\begin{align*}
\mathbb{P}^N_{\textup{cpl}} \left( \sup_{0 \leq u\leq L}\left|B_{L}(u)+L^{-1}z  {u}  - \bar{S}^N_ {L,z}(u) \right| > a_1d_N^{1/2} \log\left(d_N^{-1} L \right) \right)\leq a_2(d_N^{-1} L )^{-1}e^{z^2/L}.
\end{align*}
We further put all such coupling together and construct a probability space $(\Omega_{\textup{cpl}},\mathcal{B}_{\textup{cpl}},\mathbb{P}_{\textup{cpl}})$ on which all of $\bar{S}^N_{L,z},\ z\in\mathbb{R}, N\in\mathbb{N}$ and a Brownian bridge $B_{L}$ are defined. Moreover, the above estimates hold with $\mathbb{P}_{\textup{cpl}}^N$ replaced by $\mathbb{P}_{\textup{cpl}}$. Suppose we have a bounded sequence $z_N\in\mathbb{R}$ converging to $z_\infty$, then
\begin{align*}
\sum_{N }\mathbb{P}_{\textup{cpl}} \left( \sup_{0\leq u\leq L}\left|B_{L}(u)+L^{-1}z_N  u   - \bar{S}^N_{L,z_N}(u) \right|>a_1d_N^{1/2} \log\left(d_N^{-1}  L\right) \right)< \infty.
\end{align*}
Through the Borel-Cantelli lemma, it holds $\mathbb{P}_{\textup{cpl}}$-{almost surely} that
\begin{equation}\label{converge_bridges}
\begin{split}
&\sup_{u\in [0,L]} \left|B_{L}(u)+L^{-1}z_{\infty} u - \bar{S}^N_{L,z_N}(u) \right|\\
\leq & \sup_{u\in [0,L]} \left|B_{L}(u)+L^{-1}z_{N} u - \bar{S}^N_{L,z_N}(u) \right| + |z_N-z_\infty|\to 0.
\end{split}
\end{equation}

Let $\{(\bar{S}^{N,\ell}_{L,z},B_{L}^\ell)\}_{\ell\in\mathbb{N}}$ be a sequence of such coupling and independent between different $\ell$. Let $\{U_\ell\}_{\ell\in\mathbb{N}}$ be a sequence of independent random variables, each having the uniform distribution on $[0,1]$. We further augment the probability space $(\Omega,\mathcal{B},\mathbb{P})$ to include all such random variables in an independent manner. \\

In the first step, we define the $\ell$-th candidate for the resampled bridge.  As $u\in [a,b]$, define
\begin{align*}
\mathcal{L}^{N,\ell}_i(u)=\mathcal{L}^{N}_i(a)+ \bar{S}^{N,\ell}_{L,\mathcal{L}^{N}_i(b)-\mathcal{L}^{N}_i(a)}(u-a),
\end{align*}
and $\mathcal{L}^{N,\ell}_i(u)=\mathcal{L}^{N}_i(u)$ for $u\in [-T,a)\cup (b,T]$. Because $\bar{S}^{N,\ell}_{L,z}$ is measurable in $z$, $\mathcal{L}^{N,\ell}_i$ is measurable. Similarly, as $u\in [a,b]$, define
\begin{align*}
\mathcal{L}^{\infty,\ell}_i(u)=\mathcal{L}^{\infty}_i(a)+B^{\ell}_{L}(u-a)+\frac{u-a}{b-a}\cdot (\mathcal{L}^{\infty}_i(b)-\mathcal{L}^{\infty}_i(a)),
\end{align*}
and $\mathcal{L}^{\infty,\ell}_i(u)=\mathcal{L}^{\infty}_i(u)$ for $u\in [-T,a)\cup (b,T]$.\\

For $N\in\mathbb{N}\cup\{\infty\}$, we set $x^N=\mathcal{L}^N_i(a), y^N=\mathcal{L}^N_i(b)$, $f^N=\mathcal{L}^N_{i-1}$ and $g^N=\mathcal{L}^N_{i+1} $. In the second step, we check whether 
\begin{equation}
U_\ell\leq W(N,\ell):=W_{\Hamd^N}^{i-1,i+1,\Lambda^N_d(a,b),x^N ,y^N ,f^N ,g^N}(\mathcal{L}^{N,\ell}_i),
\end{equation}
and {\bf accept} the candidate resampling $\mathcal{L}^{N,\ell}_i$ if this event occurs. We define accordingly
\begin{equation}
W(\infty,\ell):=W_{\Ham}^{i-1,i+1,(a,b),x^{\infty} ,y^{\infty} ,f^{\infty} ,g^{\infty} }(\mathcal{L}^{ \infty,\ell}_i).
\end{equation}

For $N\in\mathbb{N}\cup\{\infty\}$, define $\ell(N)$ to be the minimal value of $\ell$ for which we accept $\mathcal{L}^{N,\ell}_i$. In Lemma~\ref{lem:lfinite} below, we show that $\ell(\infty)$ is almost surely finite. The argument for $\ell(N)$ is analogous. Write $\mathcal{L}^{N,\textup{re}}$ for the line ensemble with the $i$-th line replaced by $\mathcal{L}^{N,\ell(N)}_i$. Because $\ell(N)$ takes discrete values, $\mathcal{L}^{N,\textup{re}}$ is measurable.  The random walk Gibbs property is equivalent to the fact that for $N\in\mathbb{N}$,
\begin{align}\label{resample}
\mathcal{L}^{N,\textup{re}}\overset{(d)}{=}\mathcal{L}^{N}.
\end{align}

Our {\bf goal} is to show the same equality holds for $N=\infty$, which verifies the $\Ham$-Brownian Gibbs property for the limiting line ensembles. For the moment we assume $\ell(N)$ converges to $\ell(\infty)$ with $\ell(\infty)$ bounded almost surely (which we will prove in the lemmas following later) and we proceed to finish the proof of Theorem~\ref{theorem:main}(2).  

From \eqref{converge_bridges} and the independence among $\mathcal{L}^N$ and $\{  \bar{S}^{N,\ell}_{L,z},B_{L}^\ell\}_{\ell\in\mathbb{N}}$, one obtains almost surely
\begin{align}\label{KMT_application}
\sup_{u\in [a,b]} \left| B^{\ell(\infty)}_{L}(u-a)+\frac{u-a}{b-a}\cdot (\mathcal{L}^{\infty}_i(b)-\mathcal{L}^{\infty}_i(a))- \bar{S}^{N,\ell(N)}_{L,\mathcal{L}^{N}_i(b)-\mathcal{L}^{N}_i(a)}(u-a)\ \right|\to 0
\end{align}

Here we used the independence among $\mathcal{L}^N$ and $\{ {S}^{N,\ell}_{L,z},B_{L}^\ell\}_{\ell\in\mathbb{N}}$\ to ensure \eqref{converge_bridges} can be applied to $z_N=\mathcal{L}^N_i(b)-\mathcal{L}^N_i(a)$, $z_{\infty}=\mathcal{L}^{\infty}_i(b)-\mathcal{L}^{\infty}_i(a)$ and the convergence still holds almost surely. Then $\mathcal{L}^{N,\textup{re}}$ converges to $\mathcal{L}^{\infty,\textup{re}}$ in $C( [-T,T],\mathbb{R})$ almost surely and thus $\mathcal{L}^{\infty,\textup{re}}\overset{(d)}{=}\mathcal{L}^{\infty}$. We complete the proof of Theorem \ref{theorem:main} (2).
\end{proof}

\begin{lemma}\label{lem:lfinite}
Almost surely $\ell(\infty)$ is finite.
\end{lemma}

\begin{proof}
For fixed $\mathcal{L}^{\infty}_i(a),\mathcal{L}^{\infty}_i(b),\mathcal{L}^{\infty}_{i-1},\mathcal{L}^{\infty}_{i+1}$, $W(\infty,\ell)$ (randomness coming from $B_{L}^\ell$) are i.i.d. in $\ell$ and are supported in $(0,1)$. Hence, for some $\epsilon>0$, $W(\infty,\ell)$ is at least $\epsilon$ with probability at least $\epsilon$, which implies that $\ell(\infty)$ is finite almost surely. 
\end{proof}

\begin{lemma}
Almost surely for all $\ell$, $\displaystyle\lim_{N\to\infty}W(N,\ell)=W(\infty,\ell)$.
\end{lemma}

\begin{proof}
Let $\mathsf{A}$ be the intersection of the following events:
\begin{itemize}
\item $\displaystyle\sum_{i=1}^{k-1} \sup_{u\in [-T,T]}|\mathcal{L}^N_i(u)-\mathcal{L}^{\infty}_i(u)|\to 0.$

\item For all $\ell\in\mathbb{N}$,

$\displaystyle\sup_{u\in [a,b]} \left| B^{\ell}_{L}(u-a)+\frac{u-a}{b-a}\cdot (\mathcal{L}^{\infty}_i(b)-\mathcal{L}^{\infty}_i(a))- \bar{S}^{N,\ell}_{L,\mathcal{L}^{N}_i(b)-\mathcal{L}^{N}_i(a)}(u-a) \right|\to 0.$ 
\end{itemize}

In view of \eqref{KMT_application}, $\mathsf{A}$ occurs with probability $1$. A direct consequence is that as $\mathsf{A}$ occurs, $\mathcal{L}^{N,\ell}$ converges uniformly to $\mathcal{L}^{\infty,\ell}$ for all $\ell$. In below we show that when $\mathsf{A}$ happens, $W(N,\ell)\to W(\infty,\ell)$. Recall that $x^N=\mathcal{L}^N_i(a), y^N=\mathcal{L}^N_i(b)$, $f^N=\mathcal{L}^N_{i-1}$ and $g^N=\mathcal{L}^N_{i+1} $.  We estimate

\begin{align*}
&|W(N,\ell)-W(\infty,\ell)|\\
\leq & \left| W_{\Hamd^N}^{i-1,i+1,\Lambda^N_d(a,b),x^N ,y^N ,f^N ,g^N }(\mathcal{L}^{N,\ell})-W_{\Ham}^{i-1,i+1,(a,b),,x^N ,y^N ,f^N ,g^N }(\mathcal{L}^{ N,\ell}) \right|\\
+&\left| W_{\Ham}^{i-1,i+1,(a,b),x^N ,y^N ,f^N ,g^N  }(\mathcal{L}^{ N,\ell})-W_{\Ham}^{i-1,i+1,(a,b),x^{\infty}  ,y^{\infty} ,f^N ,g^N  }(\mathcal{L}^{ \infty,\ell}) \right|\\
+&\left| W_{\Ham}^{i-1,i+1,(a,b),x^{\infty}  ,y^{\infty} ,f^N ,g^N  }(\mathcal{L}^{ \infty,\ell})-W_{\Ham}^{i-1,i+1,(a,b),x^{\infty}  ,y^{\infty} ,f^\infty ,g^\infty }(\mathcal{L}^{ \infty,\ell}) \right|.
\end{align*}

From Lemma \ref{change_Hamiltonian}, the first term is bounded by $$C_2\big(\omega_{(a,b)}(\mathcal{L}_{i-1}^{N },d_N)+\omega_{(a,b)}(\mathcal{L}_{i}^{N,\ell},d_N)+\omega_{(a,b)}(\mathcal{L}_{i+1}^{N },1/N)+d_N\big).$$
From
\begin{align*}
\omega_{(a,b)}(\mathcal{L}_{i}^{N,\ell},d_N)\leq \omega_{(a,b)}(\mathcal{L}_{i}^{\infty,\ell},d_N)+2\| \mathcal{L}_{i}^{N,\ell}-\mathcal{L}_{i}^{\infty,\ell} \|_{C ([a,b],\mathbb{R})},
\end{align*}
the first terms goes to zero.

From Lemma \ref{change_ensemble}, the second term is bounded by $$C_3 \| \mathcal{L}_{i}^{N,\ell}-\mathcal{L}_{i}^{\infty,\ell} \|_{C( [a,b],\mathbb{R})}$$ which converges to zero. 

From Lemma \ref{change_ensemble}, the last terms is bounded by
$$C_3\big( \| \mathcal{L}_{i-1}^{N }-\mathcal{L}_{i-1}^{\infty} \|_{C( [a,b],\mathbb{R})}+ \| \mathcal{L}_{i+1}^{N }-\mathcal{L}_{i+1}^{\infty} \|_{C( [a,b],\mathbb{R})}\big)$$
which again converges to zero.
\end{proof}

\begin{lemma}
Almost surely $\displaystyle\lim_{N\to\infty} \ell(N)=\ell(\infty)$.
\end{lemma}

\begin{proof}
Let $\mathsf{A}'$ be the intersection of the event $\mathsf{A}$ above and 

\begin{itemize}
\item $\ell(\infty)<\infty$
\item $W(\infty,\ell(\infty))>U_{\ell(\infty)}$
\end{itemize}

The last condition occurs with probability $1$ since $W(\infty,\ell(\infty))\in (0,1)$ and, conditioned on $\{W(\infty,\ell(\infty))\}_{j=1}^{\ell(\infty)}$, $U_{\ell(\infty)}$ is the uniform distribution in $[0, W(\infty,\ell(\infty))]$. Then from $W(N,\ell(\infty))\to W(\infty,\ell(\infty))$, we have for $N$ large enough $W(N,\ell(\infty))>U_{\ell(\infty)}$ and then $\ell(N)\leq \ell(\infty)$. In particular,
\begin{align*}
\limsup_{n\to\infty}\ell(N)\leq \ell(\infty).
\end{align*}

On the other hand, for all $1\leq j\leq \ell(\infty)-1$, one has $W(\infty,j)<U_j$. Therefore $W(N,j)<U_j$ for $N$ large enough and hence
\begin{align*}
\liminf_{n\to\infty}\ell(N)\geq \ell(\infty).
\end{align*}
\end{proof}

\section{Proof of Two key Propositions}\label{propproofs}
	In this section, we will prove Propositions \ref{pro:1}, \ref{pro:2} by an induction on the index $k\in \mathbb{N}$. 
\subsection{Proof of Proposition \ref{pro:1}}
The $k=1$ case follows from the assumption of Theorem~\ref{theorem:main} since we assume that $\mathcal{L}^N_1(u)+u^2/2$ converges weakly as a process on $\mathbb{R}$ to a stationary process. From now on, we set $k\geq 2$ and assume Propositions~\ref{pro:1} and \ref{pro:1} have been verified for $k-1$. 

Now we set the constants used in this subsection. Throughout this subsection, we fix $\varepsilon\in (0,1)$, $x_0>0$ and $\overline{x}\in [-x_0,x_0] $.  Let ${R}_{k-1}( 8^{-1} \varepsilon)$ and $\hat{R}_{k-1}( 8^{-1} \varepsilon)$ be the constants in the Proposition~\ref{pro:1} and Proposition~\ref{pro:2} respectively. We set $\tau>0$ be the smallest number which satisfies the following conditions:
\begin{equation}\label{equ:tau}
\begin{split}
&\tau\geq 1,\ 4^{-1}\tau^2\geq R_{k-1}(8^{-1}\varepsilon)+\hat{R}_{k-1}(8^{-1}\varepsilon),\ \\
&\frac{2}{1-e^{-1/\tau}} e^{2e^{-1}\tau-\tau^{3}/32}\leq 2^{-1},\  \frac{2}{1-e^{-1/(4\tau)}} e^{8e^{-1}\tau-\tau^{3}/8} \leq 2^{-1}\varepsilon.
\end{split}
\end{equation}
The value of $\tau$ is also fixed throughout this subsection. We remark that $\tau$ depends only on $k$ and $\varepsilon$ but not on $x_0$.\\

We consider the linear functions $\ell_{\pm }(u)$ which agree with the parabola $-u^2/2$ at $u= \overline{x} \pm 2\tau$ and $u=\overline{x}\pm 4\tau$. Explicitly,
\begin{align*}
\ell_{\pm} (u)= &-(\overline{x}\mp 3\tau)u+(\overline{x}\mp 4\tau)(\overline{x}\mp 2\tau)/2.
\end{align*}
Define the events
\begin{align*}
\mathsf{Low}^N_{k,-} :=&\left\{ \sup_{u\in \Lambda_d^N[\overline{x}-4\tau,\overline{x}-2\tau] }\left(\mathcal{L}^N_k(u) -\ell_{\tau}(u)\right)\leq 0 \right\},\\
\mathsf{Low}^N_{k,+} :=&\left\{ \sup_{u\in \Lambda_d^N[\overline{x}+2\tau,\overline{x}+4\tau] }\left(\mathcal{L}^N_k(u) -\ell_{\tau}(u)\right)\leq 0 \right\},\\
\mathsf{Low}^N_{k}:=&\mathsf{Low}^N_{k,-}(\tau)\cup\mathsf{Low}^N_{k,+}(\tau).
\end{align*}
\begin{lemma}\label{lem:2.1}
There exists $N_{11}=N_{11}(k,x_0,\varepsilon)$ such that for all $N\geq N_{11}$, we have
\begin{align*}
\mathbb{P}\left( \mathsf{Low}^N_{k}  \right)<\varepsilon.
\end{align*}
\end{lemma}
\begin{proof}
We present only the proof for $\mathbb{P}(\mathsf{Low}^N_{k,-})<2^{-1}\varepsilon $ because the argument for $\mathsf{Low}^N_{k,+}$ is analogous. Let $ {R}_{k-1}(8^{-1}\varepsilon)$ and $\hat{R}_{k-1}(8^{-1}\varepsilon)$ be the constants in Proposition~\ref{pro:1} and Proposition~\ref{pro:2} respectively. Define the events
\begin{align*}
&\mathsf{G}^N_{k-1} :=\left\{ \mathcal{L}^N_{k-1}(u)-\ell_{ -}(u)\leq \hat{R}_{k-1}(8^{-1}\varepsilon)\ \textup{for}\ u=\overline{x}-2\tau\ \textup{and}\ u=\overline{x}-4\tau  \right\},\\
&\mathsf{H}_{k,k-1}^N :=\mathsf{Low}^N_{k,-} \cap \mathsf{G}^N_{k-1}, \ \mathsf{Fall}^N_{k-1}:=\left\{ \mathcal{L}^N_{k-1}(\overline{x})+\overline{x}^2/2\leq -R_{k-1}(8^{-1}\varepsilon)  \right\}.
\end{align*} 
From Propositions~\ref{pro:1} and \ref{pro:2}, for $N\geq \max\{N_1(k-1,x_0,8^{-1}\varepsilon), N_2(k-1,x_0+4\tau,8^{-1}\varepsilon)\}$, we have
\begin{align}\label{equ:2.-1}
 \mathbb{P}(\mathsf{G}^N_{k-1})\geq 1-4^{-1}\varepsilon,\  \mathbb{P}(\mathsf{Fall}^N_{k-1})< 8^{-1}\varepsilon. 
\end{align}

It is clear that the event $\mathsf{H}^N_{k,k-1}$ is $\mathcal{F}_{\textrm{ext}}\Big(\{k-1\},\Lambda_d^N(\overline{x}-4\tau,\overline{x}-2\tau)\Big)$-measurable. Therefore,
	\begin{equation*}
		\mathbb{P}\Big(\mathsf{H}^N_{k,k-1} \cap \mathsf{Fall}^N_{k-1} \Big )=\mathbb{E}\bigg[\mathbb{1}_{\mathsf{H}^N_{k,k-1}  }\mathbb{E}\Big[\mathbb{1}_{\mathsf{Fall}^N_{k-1} }\big\vert\mathcal{F}_{\textrm{ext}}\Big(\{k-1\},\Lambda_d^N(\overline{x}-4\tau,\overline{x}-2\tau)\Big)\Big]\bigg].
	\end{equation*}
Due to the $(\Hamd^N,\HrwN)$-Gibbs property of $\mathcal{L}^N$, it holds that
	\begin{equation*}
		\mathbb{E}\Big[\mathbb{1}_{\mathsf{Fall}^N_{k-1} }\big\vert\mathcal{F}_{\textrm{ext}}\Big(\{k-1\},\Lambda_d^N(\overline{x}-4\tau,\overline{x}-2\tau)\Big)\Big]=\mathbb{P}_{\Hamd^N,\HrwN}^{k-1,k-1,\Lambda_d^N(\overline{x}-4\tau,\overline{x}-2\tau),x,y,f,g}\big(\mathsf{Fall}^N_{k-1} \big),
	\end{equation*}
	where $x=\mathcal{L}^N_{k-1}(\overline{x}-4\tau)$, $y=\mathcal{L}^N_{k-1}(\overline{x}-2\tau)$ , $f=\mathcal{L}^N_{k-2}$ and $g=\mathcal{L}^N_{k}$. We claim that for $N$ large enough depending on $x_0$, we have
\begin{align}\label{equ:2.1}
	\mathbb{1}_{\mathsf{H}^N_{k,k-1} }\cdot\mathbb{P}_{\Hamd^N,\HrwN}^{k-1,k-1,\Lambda_d^N(\overline{x}-4\tau,\overline{x}-2\tau),x,y,f,g}\big(\mathsf{Fall}^N_{k-1} \big)\geq 2^{-1}\mathbb{1}_{\mathsf{H}^N_{k,k-1}}.
\end{align}
From \eqref{equ:2.1}, we have
\begin{align}\label{equ:2.0}
 \mathbb{P}\Big(\mathsf{Fall}^N_{k-1} \Big )\geq  \mathbb{P}\Big(\mathsf{H}^N_{k,k-1} \cap \mathsf{Fall}^N_{k-1} \Big )\geq 2^{-1} \mathbb{P}\Big(\mathsf{H}^N_{k,k-1} \Big ).
\end{align}
Combining \eqref{equ:2.-1} and \eqref{equ:2.0}, we conclude that
\begin{align*}
\mathbb{P}\left( \mathsf{Low}^N_{k,-}(\tau) \right)\leq  &\mathbb{P}\left( \left( \mathsf{G}^N_{k-1}(\tau) \right)^{\neg} \right)+\mathbb{P}\left(  \mathsf{H}^N_{k,k-1}(\tau) \right)\\
\leq   & \mathbb{P}\left( \left( \mathsf{G}^N_{k-1}(\tau) \right)^{\neg} \right)+2\mathbb{P}\left(  \mathsf{Fall}^N_{k-1}  \right)<2^{-1}\varepsilon.
\end{align*}
This is the desired result.\\

It remains to prove \eqref{equ:2.1}. We now use the stochastic monotonicity to simplify the boundary condition. Set
\begin{align*}
x_*=\ell_{-}(\overline{x}-4\tau)+\hat{R}_{k-1}(8^{-1}\varepsilon),&\ y_*=\ell_{-}(\overline{x}-2\tau)+\hat{R}_{k-1}(8^{-1}\varepsilon),\\ 
f_*(u)\equiv +\infty,& \ g_*(u)= \ell_{-}(u).
\end{align*}
  From the stochastic monotonicity,
\begin{align*}
&\mathbb{1}_{\mathsf{H}^N_{k,k-1} } \mathbb{P}_{\Hamd^N,\HrwN}^{k-1,k-1,\Lambda_d^N(\overline{x}-4\tau,\overline{x}-2\tau),x,y,f,g}\big(\mathsf{Fall}^N_{k-1} \big)\\
\geq & \mathbb{1}_{\mathsf{H}^N_{k,k-1} } \mathbb{P}_{\Hamd^N,\HrwN}^{k-1,k-1,\Lambda_d^N(\overline{x}-4\tau,\overline{x}-2\tau),x_*,y_*,f_*,g_*}\big(\mathsf{Fall}^N_{k-1} \big). 
\end{align*} 
It is then sufficient to show that
\begin{equation}\label{equ:2.2}
\mathbb{P}_{\Hamd^N,\HrwN}^{k-1,k-1,\Lambda_d^N(\overline{x}-4\tau,\overline{x}-2\tau),x_*,y_*,f_*,g_*}\big(\mathsf{Fall}^N_{k-1} \big)\geq 2^{-1}.
\end{equation}

\vspace{0.3cm}

To ease the notation, we write $\mathbb{P}_N$ and $\mathbb{E}_N$ respectively for the probability measure and the expectation for $ \mathbb{P}_{\free,\HrwN}^{k-1,k-1,\Lambda_d^N(\overline{x}-4\tau,\overline{x}-2\tau),x_*,y_*,f_*,g_*}.$ Also, we denote by $W_N$ for the functional $W^{k-1,k-1,\Lambda^N_d(\overline{x}-4\tau,\overline{x}-2\tau),x_*,y_*,f_*,g_*}_{\Hamd^N}.$ See Definition~\ref{def:RW_ensemble} for details. Moreover, we denote by $\mathbb{P}_{\infty}$, $\mathbb{E}_\infty$ and $W_\infty$ for the corresponding objects with random walk bridges replaced by Brownian bridges and $\Hamd^N$ replaced by $\Ham (x)=e^{x}$. See Definition~\ref{def:H_Brownian}. 
 
Since $W_N\leq 1$, we have
\begin{equation}\label{equ:2.3}
\begin{split}
\mathbb{P}_{\Hamd^N,\HrwN}^{k-1,k-1,\Lambda_d^N(\overline{x}-4\tau,\overline{x}-2\tau),x_*,y_*,f_*,g_*}\big(\big(\mathsf{Fall}^N_{k-1}\big)^{\neg} \big)  = &\mathbb{E}_N \left[  W_N \right]^{-1}   \mathbb{E}_N \left[ \mathbbm{1}_{(\mathsf{Fall}^N_{k-1})^{\neg}}\cdot W_N \right]\\
\leq&\mathbb{E}_N \left[  W_N \right]^{-1}   \mathbb{P}_N \left(   (\mathsf{Fall}^N_{k-1})^{\neg}  \right). 
\end{split}
\end{equation}

Under the law $\mathbb{P}_N$, the curve is distributed according to a $\HrwN$ random walk bridge with height $x_*$ and $y_*$ at the left and right end points respectively. Using $z_*=y_*-x_*$, it can be expressed explicitly as 
\begin{align*}
\overline{S}^N_{2\tau,z_*}(u-\overline{x}+4\tau)+x_*.
\end{align*}
See in Definition~\ref{def:RW_ensemble} for the definition of $\overline{S}^N_{2\tau,z_*}$. Through a direct calculation,
\begin{align*}
\mathbb{P}_N \left(   (\mathsf{Fall}^N_{k-1})^{\neg}  \right) =&\mathbb{P}\left( \overline{S}^N_{2\tau,z_*}(\tau)-2^{-1}z_*> 2^{-1}\tau^2 -R_{k-1}(8^{-1}\varepsilon)-\hat{R}_{k-1}(8^{-1}\varepsilon)\right)
\end{align*}
Let $N_0=N_0(2\tau,4^{-1}\tau^2,6\tau^2+6\tau x_0)$ be the constant in Lemma~\ref{lem:sup_est}. 
In view of \eqref{equ:tau} and Lemma~\ref{lem:sup_est},  for $N\geq N_0$ we have
\begin{align}\label{equ:2.4}
\mathbb{P}_N \left(   (\mathsf{Fall}^N_{k-1})^{\neg}  \right) \leq \mathbb{P}\left( \overline{S}^N_{2\tau,z_*}(\tau)-2^{-1}z_*> 4^{-1}\tau^2\right)\leq e^{-\tau^3/32}.
\end{align}

Next, we give a lower bound on $\mathbb{E}_N \left[  W_N \right]$. Recall that $B_L$ is a standard Brownian bridge on $[0,L]$. Under the law $\mathbb{P}_\infty$, the curve $S(u)$ is distributed as
$$ B_{2\tau}(u-\overline{x}+4\tau)+\ell_{-}(u)+\hat{R}_{k-1}(8^{-1}\varepsilon).$$ 
Together with $g_*(u)=\ell_{-}(u)$, $\hat{R}_{k-1}(8^{-1}\varepsilon)\geq 2$ and Lemma~\ref{Bbridge_sup}, this implies the event 
\begin{align*}
\mathsf{Sep}=\left\{  S(u)>g_*(u)+1\ \textup{  for all}\ u\in [\overline{x}-4\tau,\overline{x}-2\tau]  \right \}
\end{align*}
occurs with a probability at least $  1-e^{-1/\tau}$. Moreover, we have $$\mathbbm{1}_{\mathsf{Sep}}\cdot W_\infty\geq  e^{-2e^{-1}\tau}\mathbbm{1}_{\mathsf{Sep}}.$$ Taking expectation $\mathbb{E}_{\infty}$, it holds that 
$$\mathbb{E}_\infty[W_\infty]\geq ( 1-e^{-1/\tau}) e^{-2e^{-1}\tau} .$$
In view of Proposition \ref{pro:Z_comparison}, for $N$ large enough depending on $x_0$ and $\tau$, 
\begin{equation}\label{equ:2.5}
 \mathbb{E}_\N[W_N]\geq 2^{-1}( 1-e^{-1/\tau}) e^{-2e^{-1}\tau}. 
\end{equation}

Combining \eqref{equ:2.3}, \eqref{equ:2.4}, \eqref{equ:2.5} and \eqref{equ:tau}, we conclude that
\begin{align*}
\mathbb{P}_{\Hamd^N,\HrwN}^{k-1,k-1,\Lambda_d^N(\overline{x}-4\tau,\overline{x}-2\tau),x_*,y_*,f_*,g_*}\big(\big(\mathsf{Fall}^N_{k-1}\big)^{\neg} \big)\leq \frac{2}{1-e^{-1/\tau}} e^{2e^{-1}\tau-\tau^{3}/32} \leq 2^{-1}.
\end{align*}
This finishes the derivation of \eqref{equ:2.2}.
\end{proof}
We are now ready to prove Proposition~\ref{pro:1}.
\begin{proof}[Proof of Proposition~\ref{pro:1}]
We begin setting relevant constants. Recall that $\tau$ is the minimum number which satisfies \eqref{equ:tau}. Let $R_{k-1}(\varepsilon/(4\tau))$ be the constant in Proposition~\ref{pro:1} for the index $k-1$. Set the constants
\begin{align}\label{def:Rk}
\bar{R}=\max\{\tau^2/2,R_{k-1}(\varepsilon/(4\tau))+2\},\ R_k(\varepsilon)=\bar{R}+9\tau^2.
\end{align} 

Define the events
\begin{align*}
\mathsf{LB}^N_{k-1}=&\left\{ \inf_{u\in\Lambda^N_d[\overline{x}-2\tau,\overline{x}+2\tau]} (\mathcal{L}^N_{k-1}(u)+u^2/2)\geq -R_{k-1}(\varepsilon/(4\tau)) \right\},\\
\mathsf{E}^N_k=&\left\{ \inf_{u\in\Lambda^N_d[\overline{x}-2\tau,\overline{x}+2\tau]} (\mathcal{L}^N_k(u)+u^2/2)\leq -R_k(\varepsilon) \right\}.
\end{align*}
The goal is to show that for $N$ large enough depending on $k,x_0$ and $\varepsilon$, we have
\begin{equation}\label{equ:2.goal}
\mathbb{P}(\mathsf{E}^N_k)<3\varepsilon.
\end{equation}
Once we have \eqref{equ:2.goal}, the assertion in Proposition~\ref{pro:1} holds with $\varepsilon$ replaced by $3\varepsilon$. Then Proposition~\ref{pro:1} follows by replace $\varepsilon$ by $3^{-1}\varepsilon$.\\ 

We show that $\mathsf{LB}^N_{k-1}$ is typical. $[\overline{x}-2\tau,\overline{x}+2\tau]$ can be covered by $\lfloor 4\tau \rfloor$ intervals with length $2$. Applying Proposition~\ref{pro:1} to each interval, for $N\geq N_1(k-1,x_0+2\tau,\varepsilon/(4\tau))$ we have  
\begin{align}\label{equ:2.b}
\mathbb{P}((\mathsf{LB}^N_{k-1})^{\neg})< \varepsilon.
\end{align}
We claim that for $N$ large enough,
\begin{equation}\label{equ:2.6}
\mathbb{P}\left(\big( \mathsf{Low}^N_k \big)^{\neg}\cap\mathsf{LB}^N_{k-1}\cap\mathsf{E}^N_k \right)<\varepsilon.
\end{equation}
Combining \eqref{equ:2.b}, \eqref{equ:2.6} and Lemma~\ref{lem:2.1}, we have
\begin{align*}
\mathbb{P}(\mathsf{E}^N_k)\leq \mathbb{P}\left(\big( \mathsf{Low}^N_k \big)^{\neg}\cap\mathsf{LB}^N_{k-1}\cap\mathsf{E}^N_k \right)+\mathbb{P}\left( \mathsf{Low}^N_k  \right)+\mathbb{P}((\mathsf{LB}^N_{k-1})^{\neg})< 3\varepsilon.
\end{align*}
This concludes \eqref{equ:2.goal}.\\

It remains to prove \eqref{equ:2.6}. Define $\sigma_{- }$ to be the infimum over those $u\in\Lambda_d^N[\overline{x}-4\tau,\overline{x}-2\tau]$ such that $\mathcal{L}^N_{k}(u)-\ell_{-}(u)\geq 0 $. Likewise define $\sigma_{+ }$ to be the supremum over those $u\in\Lambda_d^N[\overline{x}+2\tau,\overline{x}+4\tau]$ such that $\mathcal{L}^N_{k}(u)-\ell_{+}(u)\geq 0$. It is easy to see that the interval $(\sigma_{-},\sigma_{+})$ forms a $\{k\}$-stopping domain (Definition \ref{def:stopping_time}) and the event $\big( \mathsf{Low}^N_k \big)^{\neg}\cap\mathsf{LB}^N_{k-1} $ is $\mathcal{F}_{\textrm{ext}}\Big(\{k\},\Lambda_d^N(\sigma_{- },\sigma_{+ })\Big)$-measurable. These imply that
 \begin{equation*}
 \mathbb{P}\left(\big( \mathsf{Low}^N_k \big)^{\neg}\cap\mathsf{LB}^N_{k-1}\cap\mathsf{E}^N_k \right)=\mathbb{E}\Bigg[\mathbb{1}_{ ( \mathsf{Low}^N_k )^{\neg}\cap\mathsf{LB}^N_{k-1}}\mathbb{E}\Big[\mathbb{1}_{\mathsf{E}^N_k}\Big \vert \mathcal{F}_{\textrm{ext}}\Big(\{k\},\Lambda_d^N(\sigma_{- },\sigma_{+ })\Big)\Big]\Bigg].
 \end{equation*}
From the strong Gibbs property,
 \begin{equation*}
 \mathbb{E}\Big[\mathbb{1}_{\mathsf{E}^N_k}\Big \vert \mathcal{F}_{\textrm{ext}}\Big(\{k\},\Lambda_d^N(\sigma_{- },\sigma_{+ })\Big)\Big]=\mathbb{P}_{\Hamd^N,\HrwN}^{k,k,\Lambda_d^N(\sigma_{- },\sigma_{+ }),x,y,f,g}(\mathsf{E}^N_k),
 \end{equation*}
where $x=\mathcal{L}^N_k(\sigma_{-})$, $y=\mathcal{L}^N_k(\sigma_{+})$, $f(u)=\mathcal{L}^N_{k-1}(u)$ and $g(u)=\mathcal{L}^N_{k+1}(u)$.

We now use the stochastic monotonicity to simplify the boundary condition. Let $\ell(u)$ be the linear function which agrees with $-u^2/2$ at $u=\sigma_\pm$. When $( \mathsf{Low}^N_k )^{\neg}$ occurs, we have
\begin{align*}
\mathcal{L}^N_k(\sigma_{\pm})\geq \ell_{\pm}(\sigma_\pm)\geq -\sigma_\pm^2/2-\tau^2/2=\ell(\sigma_\pm)-\tau^2/2. 
\end{align*} 
 When $\mathsf{LB}^N_{k-1}$ occurs, we have for all $u\in \Lambda^N_d[\sigma_-,\sigma_+]$ 
\begin{align*}
\mathcal{L}^N_{k-1}(u)\geq -u^2/2-R_{k-1}(\varepsilon/(4\tau))\geq \ell(u)-R_{k-1}(\varepsilon/(4\tau)). 
\end{align*}
In view of \eqref{def:Rk}, $( \mathsf{Low}^N_k(\tau))^{\neg}\cap \mathsf{LB}^N_{k-1}$ implies that
\begin{align*}
\mathcal{L}^N_k(\sigma_{\pm})\geq \ell(\sigma_\pm) -\bar{R},
\end{align*}
and that 
\begin{align*}
\mathcal{L}^N_{k-1}(u)\geq \ell(u)-\bar{R}+2\ \textup{ for all}\  u\in \Lambda^N_d[\sigma_-,\sigma_+].
\end{align*}
Set
\begin{align*}
x_*=\ell(\sigma_-) -\bar{R},\ y_*=\ell(\sigma_+) -\bar{R},\ f_*(u)=\ell(u) -\bar{R}+2,\ g_*(u)\equiv -\infty.
\end{align*} 
From the stochastic monotonicity, we have
\begin{align*}
\mathbb{1}_{ ( \mathsf{Low}^N_k(\tau))^{\neg}\cap\mathsf{LB}^N_{k-1}}\mathbb{E}\Big[\mathbb{1}_{\mathsf{E}^N_k}\Big \vert \mathcal{F}_{\textrm{ext}}\Big(\{k\},\Lambda_d^N(\sigma_{- },\sigma_{+ })\Big)\Big] 
\leq \mathbb{1}_{ ( \mathsf{Low}^N_k(\tau))^{\neg}\cap\mathsf{LB}^N_{k-1}}\mathbb{P}_{\Hamd^N,\HrwN}^{k,k,\Lambda_d^N(\sigma_{- },\sigma_{+ }),x_*,y_*,f_*,g_*}(\mathsf{E}^N_k). 
\end{align*}
It is then sufficient to show that
\begin{equation}\label{equ:2.7}
\mathbb{P}_{\Hamd^N,\HrwN}^{k,k,\Lambda_d^N(\sigma_{- },\sigma_{+ }),x_*,y_*,f_*,g_*}(\mathsf{E}^N_k)<\varepsilon.
\end{equation}
Since $ \sigma_+-\sigma_-\leq 8\tau$, we have $-u^2/2-8\tau^2\leq \ell(u)$. This implies
\begin{align*}
\mathsf{E}^N_k\subset \bar{\mathsf{E}}^N_k= \left\{ \inf_{u\in\Lambda^N_d[\sigma_- ,\sigma_+ ]} (\mathcal{L}^N_k(u)-\ell(u) )\leq -R_k(\varepsilon)+8\tau^2 \right\}.
\end{align*}
It is then sufficient to show that
\begin{equation}\label{equ:2.7}
\mathbb{P}_{\Hamd^N,\HrwN}^{k,k,\Lambda_d^N(\sigma_{- },\sigma_{+ }),x_*,y_*,f_*,g_*}(\bar{\mathsf{E}}^N_k)<\varepsilon.
\end{equation}

To ease the notation, we write $\mathbb{P}_N$ and $\mathbb{E}_N$ respectively for the probability measure and the expectation for $\mathbb{P}_{\free,\HrwN}^{k ,k ,\Lambda_d^N(\sigma_- ,\sigma_+ ),x_*,y_*,f_*,g_*}.$ Also, we denote by $W_N$ for the functional $W^{k ,k ,\Lambda_d^N(\sigma_- ,\sigma_+ ),x_*,y_*,f_*,g_*}_{\Hamd^N}.$ See Definition~\ref{def:RW_ensemble} for details. Moreover, we denote by $\mathbb{P}_{\infty}$, $\mathbb{E}_\infty$ and $W_\infty$ for the corresponding objects with random walk bridges replaced by Brownian bridges and $\Hamd^N$ replaced by $\Ham (x)=e^{x}$. See Definition~\ref{def:H_Brownian}.\\

Since $W_N\leq 1$, we have
\begin{equation}\label{equ:2.8}
\begin{split}
\mathbb{P}_{\Hamd^N,\HrwN}^{k,k,\Lambda_d^N(\sigma_{- },\sigma_{+ }),x_*,y_*,f_*,g_*}(\bar{\mathsf{E}}^N_k)= &\mathbb{E}_N \left[  W_N \right]^{-1}   \mathbb{E}_N \left[ \mathbbm{1}_{\bar{\mathsf{E}}^N_k }\cdot W_N \right]\\
\leq&\mathbb{E}_N \left[  W_N \right]^{-1}   \mathbb{P}_N \left( \bar{\mathsf{E}}^N_k  \right). 
\end{split}
\end{equation}

Under the law $\mathbb{P}_N$, the curve is distributed according to a $\HrwN$ random walk bridge with height $x_*$ and $y_*$ at the left and right end points respectively. Using $z_*=y_*-x_*$, it can be expressed explicitly as 
\begin{align*}
\overline{S}^N_{\sigma_+-\sigma_- ,z_*}(u-\sigma_- )+x_*.
\end{align*}
Through a direct calculation and \eqref{def:Rk},
\begin{align*}
\mathbb{P}_N \left( \bar{\mathsf{E}}^N_k  \right) =&\mathbb{P}\left(\inf_{u\in\Lambda^N_d[\sigma_- ,\sigma_+ ]} (\bar{S}^N_{\sigma_+-\sigma_-,z_*}(u-\sigma_-) -\ell(u)+\ell(\sigma_-) )\leq -\tau^2 \right).
\end{align*}
Let $N_0=N_0(8\tau,\tau^2,2^{-1}(x_0+4 \tau)^2)$ be the constant in Lemma~\ref{lem:sup_est}. From Lemma~\ref{lem:sup_est}, for $N\geq N_0$, we have
\begin{align}\label{equ:2.9}
\mathbb{P}_N \left( \bar{\mathsf{E}}^N_k  \right)\leq  e^{-\tau^3/8}.
\end{align}

Next, we give a lower bound on $\mathbb{E}_N \left[  W_N \right]$. Under the law $\mathbb{P}_\infty$, the curve $S(u)$ is distributed  as
$$ B_{\sigma_+-\sigma_- }(u-\sigma_- )+\ell(u)-\bar{R}.$$ 
Together with $f_*(u)=\ell(u)-\bar{R}+2$ and Lemma~\ref{Bbridge_sup}, this implies the event 
\begin{align*}
\mathsf{Sep}=\left\{  S(u)<f_*(u)-1\ \textup{  for all}\ u\in [\sigma_- ,\sigma_+]  \right \}
\end{align*}
occurs with probability $  1-e^{-2/(\sigma_+-\sigma_-) }\geq 1-e^{-1/(4\tau ) }$.  Moreover, we have 
$$\mathbbm{1}_{\mathsf{Sep}}\cdot W_\infty\geq  e^{-8e^{-1}\tau}\mathbbm{1}_{\mathsf{Sep}}.$$ Taking expectation $\mathbb{E}_{\infty}$, it holds that 
$$\mathbb{E}_\infty[W_\infty]\geq ( 1-e^{-1/(4\tau)}) e^{-8e^{-1}\tau} .$$
In view of Proposition \ref{pro:Z_comparison}, for $N$ large enough depending on $x_0$ and $\tau$, 
\begin{equation}\label{equ:2.10}
 \mathbb{E}_\N[W_N]\geq 2^{-1}( 1-e^{-1/(4\tau)}) e^{-8e^{-1}\tau}. 
\end{equation}
Combining \eqref{equ:2.8}, \eqref{equ:2.9}, \eqref{equ:2.10} and \eqref{equ:tau}, we conclude that
\begin{align*}
\mathbb{P}_{\Hamd^N,\HrwN}^{k-1,k-1,\Lambda_d^N(\sigma_- ,\sigma_+ ),x_*,y_*,f_*,g_*}\big( \bar{\mathsf{E}}^N_k\big)\leq \frac{2}{1-e^{-1/(4\tau)}} e^{8e^{-1}\tau-\tau^{3}/8} \leq \varepsilon.
\end{align*}
This finishes the derivation of \eqref{equ:2.7}.
\end{proof}

\subsection{Proof of Proposition \ref{pro:2}}

Our proof proceeds by induction on the curve index $k$. For the case $k=1$, Proposition \ref{pro:2} follows from assumption in Theorem~\ref{theorem:main}. The general case is $k\geq 3$, and the case $k=2$ is a specialization of the $k\geq 3$ proof. So, from here on we will assume that $k\geq 3$. 

We will apply Proposition \ref{pro:1} for indices $k-2,k-1$ and $k$ and Proposition \ref{pro:2} for index $k-1$. The idea is to show that should the index $k$ curve be too high at some time $u\in [\bar{x},\bar{x}+\tfrac{1}{2}]$ then so too must the index $k-1$ curve be high at some point between $[x,\bar{x}+2]$. This violates the index $k-1$ result of Proposition \ref{pro:2} assumed by the induction, and hence proves the index $k$ case.\\

Now we set the constants and events which will be used in this subsection. Throughout this subsection, we fix $\varepsilon\in (0,1)$, $x_0>0$ and $\overline{x}\in [-x_0,x_0] $. We use $M,r_k, \hat{r}_{k-1}$ and $\hat{r}_k$ as positive parameters. Consider the event  
\[
\mathsf{E}^N_k(\hat{r}_k)=\bigg\{\sup_{u\in \Lambda_d^N[\overline{x},\overline{x}+\frac{1}{2}]}(\mathcal{L}^N_k(u)+u^2/2)\geq \hat{r}_k\bigg\}.
\]
The goal is to show that for suitable $\hat{R}_k(\varepsilon)$ and $N$ large enough, we have
\begin{equation}\label{equ:3.goal}
\mathbb{P}\left(\mathsf{E}^N_k(\hat{R}_k (\varepsilon) )\right)<8\varepsilon.
\end{equation}
Then Proposition~\ref{pro:2} follows easily from \eqref{equ:3.goal}. 

Set
\[
\chi^N(\hat{r}_k)=\inf\Big\{u\in \Lambda_d^N[\overline{x},\overline{x}+\frac{1}{2}]:(\mathcal{L}^N_k(u)+u^2/2)\geq \hat{r}_k\Big\},
\]
with the convention that if the infimum is not attained then $\chi(\hat{r}_k)=\overline{x}+1/2$. We will generally shorten $\chi^N(\hat{r}_k)$ by just writing $\chi$. 

Let us further define events (which we will show to be typical)
\begin{align*}
\mathsf{Q}^N_{k-2}(M)&=\bigg\{\inf_{u\in \Lambda_d^N[\overline{x},\overline{x}+2]}(\mathcal{L}^N_{k-2}(u)+u^2/2)\geq -M\bigg\},\\
\mathsf{A}^N_{k-1,k}(r_k)&=\Big\{\mathcal{L}^N_{k-1}(\chi)+\chi^2/2\geq -r_k\Big \}\cap\Big\{\mathcal{L}^N_{j}(\overline{x}+2)+(\overline{x}+2)^2/2\geq -r_k\ \textrm{for}\ j=k,k-1\Big\}.
\end{align*}
Lastly, define the event (which will be shown to be atypical)
\begin{equation*}
\mathsf{B}^N_{k-1}(\hat{r}_{k-1})=\bigg\{\sup_{u\in \Lambda_d^N[\chi,\overline{x}+2]}(\mathcal{L}^N_{k-1}(u)+u^2/2)\geq \hat{r}_{k-1}\bigg\}.
\end{equation*}
For the above $N$-dependent events, we will typically drop the $N$ superscript. For instance, we will denote $\mathsf{E}^N_k(\hat{r}_k)$ simply by $\mathsf{E}_k(\hat{r}_k)$.\\

We now discuss the first set of requirements on the parameters. Assuming
\begin{align}\label{equ:ASS1}
  M\geq R_{k-2}(\varepsilon),  r_k\geq R_{k-1}(\varepsilon)+R_{k}(\varepsilon),\ \textup{and}\ \hat{r}_{k-1}\geq R_{k-1}(\varepsilon), 
\end{align}
from Proposition~\ref{pro:1} and Proposition~\ref{pro:2}, it holds that that
\begin{align}\label{equ:666}
\mathbb{P}(\mathsf{Q}_{k-2}(M))\geq 1-\varepsilon,\ \mathbb{P}(\mathsf{A}_{k-1,k-2}(r_k))\geq 1-3\varepsilon,\ \textup{and}\ \mathbb{P}(\mathsf{B}_{k-1}(\hat{r}_{k-1}))<\varepsilon, 
\end{align}  
provided $N$ is large enough depending on $k$, $x_0$ and $\varepsilon$. \\

Observe that the interval $[\chi,\overline{x}+2]$ forms a $\{k-1,k\}$-stopping domain for $\mathcal{L}^N$.  Observe also that the events $\mathsf{E}_k(\hat{r}_k),\mathsf{Q}^N_{k-2}(M)$ and $\mathsf{A}^N_{k-1,k}(r_k)$ are all $\mathcal{F}_{\textrm{ext}}\big(\{k-1,k\},\Lambda_d^N(\chi,\overline{x}+2)\big)$-measurable. This implies $\mathbb{P}\big(\mathsf{E}_k(\hat{r}_k)\cap \mathsf{Q}^N_{k-2}(M)\cap \mathsf{A}^N_{k-1,k}(r_k))\cap\mathsf{B}_{k-1}(\hat{r}_{k-1})\big)$ equals
\begin{align*}
\mathbb{E}\bigg[\mathbbm{1}_{ \mathsf{E}_k(\hat{r}_k)\cap \mathsf{Q}^N_{k-2}(M)\cap \mathsf{A}^N_{k-1,k}(r_k) }  \mathbb{E}\bigg[\mathsf{B}_{k-1}(\hat{r}_{k-1})\Big \vert\mathcal{F}_{\textrm{ext}}\big(\{k-1,k\},\Lambda_d^N(\chi,\overline{x}+2)\big) \bigg] \bigg].
\end{align*}
By the strong Gibbs property, we have that $\mathbb{P}$-almost surely:
\begin{align*}
	&\mathbb{E}\bigg[\mathsf{B}_{k-1}(\hat{r}_{k-1})\Big \vert\mathcal{F}_{\textrm{ext}}\big(\{k-1,k\},\Lambda_d^N(\chi,\overline{x}+2)\big) \bigg] = \mathbb{P}_{\Hamd^N,\HrwN}^{k-1,k,\Lambda_d^N(\chi,\overline{x}+2),\vec{x} ,\vec{y},f,g}(\mathsf{B}_{k-1}(\hat{r}_{k-1})),
\end{align*}
where $\vec{x}=(\mathcal{L}^N_{k-1}(\chi),\mathcal{L}^N_{k}(\chi))$, $\vec{y}=(\mathcal{L}^N_{k-1}(\overline{x}+2),\mathcal{L}^N_{k}(\overline{x}+2))$, $f(u)=\mathcal{L}^N_{k-2}(u)$ and $g(u)=\mathcal{L}^N_{k+1}(u)$.

We now use the stochastic monotonicity to simplify the boundary condition.	Given that the event $\mathsf{E}_k(\hat{r}_k)\cap \mathsf{Q}_{k-2}(M) \cap \mathsf{A}_{k-1,k}(r_k)$ occurs, it follows that
	\begin{align*}
		\mathcal{L}^N_{k-1}(\chi)&\geq -r_k-\chi^2/2\\
		\mathcal{L}^N_{k}(\chi)&\geq \hat{r}_k-\chi^2/2\\
		\mathcal{L}^N_{k-1}(\overline{x}+2)&\geq -r_k-(\overline{x}+2)^2/2\\
		\mathcal{L}^N_{k}(\overline{x}+2)&\geq -r_k-(\overline{x}+2)^2/2\\
		\mathcal{L}^N_{k-2}(u)&\geq M-u^2/2 \quad \textrm{for all}\ u\in \Lambda_d^N[\chi,\overline{x}+2].
	\end{align*}
Set
\begin{equation*} 
\begin{split}
\vec{x}_*=(-r_k-\chi^2/2,\hat{r}_k-\chi^2/2)&,\ \vec{y}_*=(-r_k-(\overline{x}+2)^2/2,-r_k-(\overline{x}+2)^2/2)\\
f_*(u)=M-u^2/2&,\ g_*(u)\equiv -\infty.
\end{split}
\end{equation*}
By the stochastic monotonicity,
\begin{equation}\label{eq:proof_3}
\begin{split}
&\mathbbm{1}_{ \mathsf{E}_k(\hat{r}_k)\cap \mathsf{Q}^N_{k-2}(M)\cap \mathsf{A}^N_{k-1,k}(r_k) }  \mathbb{P}_{\Hamd^N,\HrwN}^{k-1,k,\Lambda_d^N(\chi,\overline{x}+2),\vec{x} ,\vec{y},f,g}(\mathsf{B}_{k-1}(\hat{r}_{k-1}))\\
 \geq &\mathbbm{1}_{ \mathsf{E}_k(\hat{r}_k)\cap \mathsf{Q}^N_{k-2}(M)\cap \mathsf{A}^N_{k-1,k}(r_k) }  p(M,r_k,\hat{r}_{k-1},\hat{r}_k),
\end{split}
\end{equation}
where $p(M,r_k,\hat{r}_{k-1},\hat{r}_k)$ is a shorthand for $\mathbb{P}_{\Hamd^N,\HrwN}^{k-1,k,\Lambda_d^N(\chi,\overline{x}+2),\vec{x}_* ,\vec{y}_*,f_*,g_*}(\mathsf{B}_{k-1}(\hat{r}_{k-1}))$.\\ 

We would like to set parameters such that $p(M,r_k,\hat{r}_{k-1},\hat{r}_k)$ has a lower bound. This is achieved in Lemma~\ref{lem:key} below.
\begin{lemma}\label{lem:key}
Suppose $\Hamd^N$ and $\HrwN$ satisfy assumptions A3 and A4 respectively. Then there exist positive numbers $\delta>0$, $R^0>0$, and functions $K^0:\mathbb{R}\to\mathbb{R}_+ $, and $\hat{R}^0:\mathbb{R}^2 \to\mathbb{R}_+ $ and $N_0:\mathbb{R}^3 \to\mathbb{N} $ such that the following holds. Given the data
\begin{align*}
R\geq R^0,\  K\geq  K^0(R), \hat{R}\geq \hat{R}^0(R,K), N\geq    N_0(R,K,\hat{R}),
\end{align*}
and
\begin{align*}
\overline{x}\in \mathbb{R},  \chi\in   \Lambda_d^N[\overline{x},\overline{x}+1/2], 
\end{align*}    
it holds that
\begin{equation*}
	     \mathbb{P}_{N}^{1,2} \left(\sup_{u\in \Lambda_d^N[\chi,\overline{x}+2]}(\bar{S}_1(u)+u^2/2)\geq 2^{-1}\delta\hat{R}\right)\geq 4^{-1},
	\end{equation*}
where $\mathbb{P}_{N}^{1,2}$ is a shorthand for the measure below on the curves $\bar{S}_1$ and $\bar{S}_2$,
\[
\mathbb{P}_{\Hamd^N,\HrwN}^{1,2,\Lambda_d^N(\chi,\overline{x}+2),(-R-\chi^2/2,\hat{R}-\chi^2/2),(-R-(\overline{x}+2)^2/2,-R-(\overline{x}+2)^2/2),-K-x^2/2,-\infty}.
\]
\end{lemma}
We postpone the proof of Lemma~\ref{lem:key} to the end of this section. Suppose that
\begin{align}\label{equ:ASS2}
 r_k\geq R^0,\ M\geq K^0(r_k), \hat{r}_k\geq\hat{R}^0(r_k,M),\  2^{-1}\delta \hat{r}_k\geq\hat{r}_{k-1} .
\end{align} 
Then we have $p(M,r_k,\hat{r}_{k-1},\hat{r}_k) \geq 4^{-1}$. By taking the expectation in (\ref{eq:proof_3}), we obtain
  \begin{equation}\label{equ:444}
	 \mathbb{P}(\mathsf{E}_k(\hat{r}_k)\cap \mathsf{Q}_{k-2}(M) \cap \mathsf{A}_{k-1,k}(r_k))\leq 4\mathbb{P}(\mathsf{B}_{k-1}(\hat{r}_{k-1})).   
  \end{equation}
Combining \eqref{equ:666} and \eqref{equ:444}, we obtain
	\begin{align*}
		\mathbb{P}(\mathsf{E}_k(\hat{r}_k))&\leq \mathbb{P}(\mathsf{E}_k(\hat{r}_k)\cap \mathsf{Q}_{k-2}(M) \cap \mathsf{A}_{k-1,k}(r_k))+\mathbb{P}(( \mathsf{Q}_{k-2}(M))^{\neg})+\mathbb{P}(  (\mathsf{A}_{k-1,k}(r_n))^{\neg})\\
		&\leq 4\mathbb{P}(\mathsf{B}_{k-1}(\hat{r}_{k-1})) +\mathbb{P}(( \mathsf{Q}_{k-2}(M))^{\neg})+\mathbb{P}(  (\mathsf{A}_{k-1,k}(r_n))^{\neg})< 8\varepsilon.
	\end{align*}
	 
We are ready to determine the value of $\hat{R}_{k}(\varepsilon)$ and prove \eqref{equ:3.goal}. Let $\hat{R}_{k}(\varepsilon)$ be the minimum number which satisfy the following condition. There exists $(M,r_k,\hat{r}_{k-1},\hat{r}_k)$ such that $\hat{r}_k=\hat{R}_{k}(\varepsilon)$ and \eqref{equ:ASS1} and \eqref{equ:ASS2} are satisfied. This finishes the derivation of \eqref{equ:3.goal}.\\

In the rest of the section, we prove Lemma~\ref{lem:key}. The analogue of Lemma~\ref{lem:key} for Brownian bridge measure has been proved in \cite[Proposition 7.6]{CH16}. Our strategy is to use Assumption A4 to compare random walk bridges and Brownian bridges, hence finishing the proof.
\begin{proof}[Proof of Lemma \ref{lem:key}]
Recall that $\Lambda^N_d= d_N\mathbb{Z}$ for some $d_N$ decreasing to $0$. For simplicity we take $d_N=N^{-1}$ in the prove. For any triple $(R,K,\hat{R})$, we consider
\begin{equation*} 
\begin{split}
\vec{x}_0=(-R-\chi^2/2,\hat{R}_k-\chi^2/2)&,\ \vec{y}_0=(-R-(\overline{x}+2)^2/2,-R-(\overline{x}+2)^2/2)\\
f_0(u)=K-u^2/2&,\ g_0(u)\equiv -\infty.
\end{split}
\end{equation*} 	 
To simplify the notation we denote the measures:
	$$\overline{\mathbb{P}}_N:=\mathbb{P}_{\Hamd^N,\HrwN}^{1,2,\Lambda_d^N(\chi,\overline{x}+2),\vec{x}_0 ,\vec{y}_0 ,f_0,g_0},\   \mathbb{P}_N:=\mathbb{P}_{\free,\HrwN}^{1,2,\Lambda_d^N(\chi,\overline{x}+2),\vec{x}_0,\vec{y}_0},$$
and
$$\overline{\mathbb{P}}_\infty=\mathbb{P}_{\Ham}^{1,2,(\chi,\overline{x}+2),\vec{x}_0 ,\vec{y},f_0,g_0},\   \mathbb{P}_\infty= \mathbb{P}_{\free}^{1,2,(\chi,\overline{x}+2),\vec{x}_0,\vec{y}_0}.$$  	 
See Definitions~\ref{def:H_Brownian} and \ref{def:RWB_Gibbs_ensemble} for more details. We write $\overline{\mathbb{E}}_N,  \mathbb{E}_N,\overline{\mathbb{E}}_\infty,  \mathbb{E}_\infty$ for the corresponding expectations. We use $\bar{S}^N= (\bar{S}^N_1,\bar{S}^N_2)$ to denote curves with laws $\mathbb{P}_N$ or $\bar{\mathbb{P}}_N$. Similarly, we use $B=(B_1,B_2)$ to denote curves with laws $\mathbb{P}_\infty$ or $\bar{\mathbb{P}}_\infty$. Furthermore, we simplify the Boltzmann weights as
\begin{align*}
W_N(\bar{S}^N )&=\exp\Big[-\sum_{k=0}^{2}\sum_{u\in \Lambda_d^N[\chi,\overline{x}+2]} \Hamd^N[\rectangle(\bar{S}^N,k,u)]\Big],\\
W_\infty(B)&=\exp\Big[-\sum_{i=0}^{2}\int_{\chi}^{\overline{x}+2} du \exp[B_{i+1}(u)-B_{i}(u)] \Big],
\end{align*}
 with the convention that the index $0$ curve is $K-u^2/2$ and the index $3$ curve is identical to $-\infty$.\\

From \cite[Proposition 7.6 ]{CH16}, there exists $\delta>0$, $R^0>0$, and functions $K^0:\mathbb{R}\to\mathbb{R}_+ $, and $\hat{R}^0:\mathbb{R}^2 \to\mathbb{R}_+ $ such that the following is true. For all
  \begin{align}\label{equ:ASS3}
 R\geq R^0,\ K \geq K^0(R),\ \hat{R}\geq \hat{R}^0(R,K)  
\end{align}   
and all
$\overline{x}\in \mathbb{R}$ and $\chi\in[\overline{x},\overline{x}+1/2]$, it holds that 		
\begin{equation}\label{equ:kmu}
 		\overline{\mathbb{P}}_\infty\left(\sup_{u\in[\chi,\overline{x}+2]}(B_1(u)+u^2/2)\geq 2^{-1} \delta\hat{R}\right) \geq  2^{-1}.
\end{equation}
We adapt the same numbers $\delta, R^0$ and functions $K^0, \hat{R}^0$. We also fix $(R,K,\hat{R})$ such that \eqref{equ:ASS3} is satisfied. We aim to show that the same event, under the measure $\mathbb{Q}$, occurs with probability at least $4^{-1}$.\\

Consider the event 
\begin{align*}
\mathsf{I}^N=&\left\{\sup_{u\in\Lambda_d^N[\chi,\overline{x}+2]}(\bar{S}_1(u)+u^2/2)\geq 2^{-1} \delta\hat{R}\right\}.
\end{align*}
Our goal is to show that
\begin{align}\label{equ:4.goal}
\liminf_{N\to\infty}\overline{\mathbb{P}}_N(\mathsf{I}^N)\geq 2^{-1}.
\end{align}
Once we have \eqref{equ:4.goal}, we know  $\overline{\mathbb{P}}_N(\mathsf{I}^N)\geq 4^{-1}$ for $N$ large enough. This finishes the proof of Lemma~\ref{lem:key}. The rest of the proof is devoted to show \eqref{equ:4.goal} \\

From the definition of $\overline{\mathbb{P}}_N$, we have
\begin{align*}
\overline{\mathbb{P}}_N \left(\mathsf{I}^N  \right)=&\mathbb{E}_N[W_N]^{-1}\mathbb{E}_N[\mathbbm{1}_{\mathsf{I}^N} W_N]
\end{align*}
To prove \eqref{equ:4.goal}, it amounts to find a lower bound for $
Y_N':=  \mathbb{E}_N  \left(\mathbb{1}_{\mathsf{I}^N  }W_N\right), $
and an upper  bound for $Y_N:=\mathbb{E}_N  \left(W_N\right).$ Those bounds are obtained through a strong coupling. 

From Assumption A4, we can couple the measures $\mathbb{P}_N $ and $\mathbb{P}_\infty$ in a probability space $(\Omega_N,\mathbb{Q}_N)$. Moreover, these exits  $a>0$ such that for $i=1,2$ and $b_N=aN^{-1/2}\log N$,
\begin{equation}\label{equ:4.couple}
\lim_{N\to\infty} \mathbb{Q}_N\left(\sup_{u\in[\chi ,\overline{x}+2],\ i=1,2}\vert \bar{S}^N_i(u)-B_i(u)\vert \geq b_N \right)=0.
\end{equation}

Let us start with the lower bound for $Y'_N$. Let $\mathsf{D}_N	$ be the event:
	\begin{equation*}
\mathsf{D}_N:=\left\{\sup_{u\in[\chi,\overline{x}+2]}\vert \bar{S}^N_i(u)-B_i(u)\vert \leq b_N ,\ \textup{for}\ i=1,2\right\}.
	\end{equation*}
It is clear from \eqref{equ:4.couple} that
	$$\lim_{N\to\infty} \mathbb{Q}_N(\mathsf{D}_N)=0.$$
Under the Assumption A3, there exists a constant $C$ such that
	\begin{align*}
		W_N(\bar{S}^N )&\geq W_\infty(\bar{S}^N) \cdot  \exp \left(-C  \omega_{(\chi,\overline{x}+2)}(\bar{S}^N ,1/N) \right)
	\end{align*}
Define the event
 \begin{align*}
 \mathsf{J}^N= &\left \{\sup_{u\in[\chi,\overline{x}+2]}(B_1(u)+u^2/2)\geq 2^{-1}\delta\hat{R}+ b_N  \right\}
 \end{align*}	
Assume that the event $\mathsf{D}$ occurs	, we obtain
	\begin{align*}
\omega_{(\chi,\overline{x}+2)}(\bar{S}^N ,1/N)&\leq \omega_{(\chi,\overline{x}+2)} (B^N,1/N)+2b_N, \\
-\log W_\infty(\bar{S}^N ) &\leq - e^{2b_N} \log W_\infty(B  	).
	\end{align*}
Together with $\mathsf{J}^N\cap\mathsf{D}\subset \mathsf{I}^N$, we have
\begin{align*}
	Y_N' \geq  &  \mathbb{E}_\infty \left[\mathbb{1}_{\mathsf{J}^N}\exp\left( e^{2b_N}\log W_\infty(B) 
	-C\omega_{(\chi,\overline{x}+2)} (B ,1/N) -4Cb_N   \right)  \right]-\mathbb{Q}_N( \mathsf{D}_N^\neg).
\end{align*}
Fix an arbitrary $\varepsilon>0$. Let 
$$A'(N,\varepsilon)=  e^{-2C\varepsilon-4Cb_N}  \mathbb{P}_\infty \left(\left\{	\omega_{(\chi,\overline{x}+2)} (B,1/N)<\epsilon\right\}\right),\ C'(N)=\mathbb{Q}_N( \mathsf{D}_N^\neg)$$
Then the above inequality becomes:
	\begin{equation*}
		Y_N'\geq  A'(N,\epsilon)    \mathbb{E}_\infty \left[\mathbb{1}_{\mathsf{J}^N}\exp\left( e^{2b_N}\log W_\infty(B) 
	 \right)  \right]-C'(N). 
	\end{equation*}
Let $c'_N=e^{2b_N}>1$. By the convexity of function $x^{\alpha}$ for $\alpha >1$, Jensen's inequality implies 
\begin{equation}\label{eq:boundY'}
		Y_N'\geq  A'(N,\epsilon)\mathbb{E}_\infty \left[\mathbb{1}_{\mathsf{J}^N }W_\infty \right]^{c'_N}-C(N).
\end{equation}

\vspace*{0.3cm}

Next, we turn to establishing an upper bound for $Y_N$. Since $W_N$ is bounded by $1$, 
	\begin{equation}\label{eq:bound_Y}
	Y_N\leq   \mathbb{E}_N\left[\mathbbm{1}_{\mathsf{D}} W_N  \right]+ \mathbb{Q}_N(\mathsf{D}_N^\neg).
	\end{equation}
When $\mathsf{D}$ occurs, it holds that
\begin{align*}
	\omega_{\chi,\overline{x}+2}(\bar{S}^N ,1/n)&\geq \omega_{\chi,\overline{x}+2} (B ,1/N)-2b_N\\
-\log W_\infty(\bar{S}^N) &\geq - e^{-2b_N} \log W_\infty(B).
	\end{align*}
Arguing similarly as above, we can get the following upper bound
	\begin{align*}
 \mathbb{E}_N\left[\mathbbm{1}_{\mathsf{D}} W_N  \right] &\leq  \mathbb{P}_\infty\left(\{\omega_{(\chi,\overline{x}+2)} (B,1/N)>\epsilon\}\right)  + e^{2C\varepsilon+4Cb_N}  \mathbb{P}_\infty \left[\exp\left( e^{-2b_N} \log W_\infty \right)\right].
\end{align*}
Let $c_N=e^{-2b_N}<1$. By the concavity of $x^\alpha$ for $0<\alpha<1$, we can deduce that
\begin{equation}\label{eq:bound_W}
 \mathbb{E}_N\left[\mathbbm{1}_{\mathsf{D}} W_N  \right] \leq \mathbb{P}_\infty\left(\{\omega_{(\chi,\overline{x}+2)} (B,1/N)>\epsilon\}\right)+e^{2C\varepsilon+4Cb_N}    \mathbb{E}_\infty \left[W_\infty \right]^{c_N}. 
	\end{equation}
Combining  (\ref{eq:bound_Y}) and (\ref{eq:bound_W}), we obtain:
\begin{equation}\label{eq:boundY1}
	Y_N\leq A(N,\epsilon)  \mathbb{E}_\infty \left[W_\infty \right]^{c_N}+C(N,\epsilon ,
	\end{equation}
		where
		\begin{equation*}
		A(N,\epsilon)= e^{2C\varepsilon+4Cb_N},\ C(N,\epsilon)=\mathbb{Q}_N (\mathsf{D}_N^{\neg})+ \mathbb{P}_\infty\left(\{\omega_{(\chi,\overline{x}+2)} (B,1/N)>\epsilon\}\right).
		\end{equation*}\\
		
From \eqref{eq:boundY'} and \eqref{eq:boundY1}, we obtain to get
	\begin{equation*}
\overline{\mathbb{P}}_N[\mathsf{I}^N ]= \frac{Y'_N}{Y_N}\geq \frac{A'(N,\epsilon)\mathbb{E}_\infty \left[\mathbb{1}_{\mathsf{J}^N }W_\infty \right]^{c'_N}-C'(N)  }{A(N,\epsilon)  \mathbb{E}_\infty \left[W_\infty \right]^{c_N} + C(N,\epsilon)}.
	\end{equation*}
When $N$ goes to infinity,
\begin{align*}
&\lim_{N\to\infty} A(N,\varepsilon)=e^{2C\varepsilon},\ \lim_{N\to\infty} A'(N,\varepsilon)=e^{-2C\varepsilon},\\
&\lim_{N\to\infty} C(N,\varepsilon)=\lim_{N\to\infty} C'(N )=0,\ \lim_{N\to\infty} c_N=\lim_{N\to\infty} c'_N=1. 
\end{align*}
As a result, 
\begin{align*}
\liminf_{N\to\infty}\overline{\mathbb{P}}_N(\mathsf{I}^N)\geq e^{-4C\varepsilon} \lim_{N\to\infty}\overline{\mathbb{P}}_\infty(\mathsf{J}^N) \geq 2^{-1}e^{-4C\varepsilon}.
\end{align*}
We have used \eqref{equ:kmu} in the second inequality. Because $\varepsilon>0$ is arbitrary, we conclude that
$$\liminf_{N\to\infty}\overline{\mathbb{P}}_N(\mathsf{I}^N)\geq 2^{-1},$$
which finishes the proof.
\end{proof}

\section{Application for the log-gamma directed polymers}\label{sec:log_gamma}
In this section, we discuss an application of  Theorem~\ref{theorem:main} to the log-gamma line ensembles which are assocaited with the log-gamma directed polymers, as introduced in \cite{Sep}. We aim to show that under the same scaling, the whole scaled line ensemble is tight. Moreover, any subsequential limit enjoys the $\mathbf{H}$-Brownian Gibbs property with $\Ham(x)=e^x$. This is the content of Theorem~\ref{thm:loggamma}. 

This section is organized as follows. In Subsection~\ref{sec:5.1}, we give the definition of the log-gamma line ensembles and show that they enjoy $(\Hamd,\Hrw)$-Gibbs property with suitable $(\Hamd,\Hrw)$. In Subsection~\ref{sec:5.2}, we introduce the weak noise scaling and check that the scaled log-gamma line ensembles satisfy the assumptions in Theorem~\ref{theorem:main}. The convergence of the the lowest indexed curve is essentially proved in \cite{AKQ}. See the discussion after Proposition~\ref{pro:AKQ} for more details. We then proceed to verify the assumptions A1-A4 and prove Theorem~\ref{thm:loggamma}.

\subsection{Log-gamma line ensemble}\label{sec:5.1} 
In this subsection, we introduce the discrete log-gamma line ensembles and discuss some of their properties. We give the definition of the inverse gamma distribution.

\begin{definition}
Fix $\gamma>0$ and let $\Gamma(\gamma)$ be the Gamma function. A random variable $X$ has inverse gamma distribution with shape parameter $\gamma$ if its probability density function is given by
\begin{equation}\label{densityinvgamma}
 \Gamma(\gamma)^{-1} x^{-\gamma-1}e^{-x^{-1}}\mathbbm{1}_{\{x>0\}}.
\end{equation}
We abbreviate with $X \overset{(d)}{=} \textup{Inv-Gamma} (\gamma)$.
\end{definition}
It could be directly computed that for any $k\in\mathbb{N}$ with $\gamma>k$,
\begin{align}\label{def:gamma}
\mathbb{E}[\textup{Inv-Gamma} (\gamma)^k]=\frac{\Gamma(\gamma-k)}{\Gamma(\gamma)}=\frac{1}{(\gamma-1)(\gamma-2)\cdots (\gamma-k)}.
\end{align}

Next, we define the random environment for the log-gamma directed polymer model. Let $K\in\mathbb{N}$ be an integer. The value of $K$ will be fixed throughout this subsection.	 Consider a semi-infinite matrix $d=(d_{ij}:i\in\mathbb{N} ,j\in [1,K]_{\mathbb{Z}} )$ of i.i.d random variables with distribution: 
 \begin{equation}
 d_{ij} \overset{(d)}{=} \textup{Inv-Gamma} (\gamma).
 \end{equation}
For any $n\in\mathbb{N} $, we denote by $d^{n,K}$ the $n\times K$ matrix $(d_{ij}:i\in [1,n]_{\mathbb{Z}} , j\in [1,K ]_{\mathbb{Z}})$. We call $d_{ij}$ the weight at the location $(i,j)$.\\

We present the geometric RSK correspondence which constructs an array (see Definition~	\ref{def:array}) from a $n\times K$ matrices with positive entries. See \cite{COSZ} for more discussions and properties for the geometric RSK correspondence. Then we will focus on the shape of the array (See Definition~\ref{def:shape}).

 For $l,k\in\mathbb{N}$ with $1\leq l\leq k\leq K$, let $\Pi^n_{k,l}$ denote the collection of $l$-tuples $\pi=(\pi_1,\cdots,\pi_l)$ of non-intersecting lattice paths in $\mathbb{Z}^2$ such that for $1\leq r\leq l$, $\pi_r$ is a lattice up-right path from $(1,r)$ to $(n,k+r-l)$. See the left figure in Table~\ref{paths} below for an illustration. Notice that $n<l$ implies $\Pi^n_{k,l}$ is empty or contains a single element. If $n < l < k$, then $\Pi^n_{k,l}$ is empty.  If $n<l = k$, there is a unique $l$-tuple in $\Pi^n_{k,l}$  in which all paths are horizontal. On the other hand, it is easy to see that $\Pi^n_{l,k}$ is non-empty for $l\leq  n \wedge k$.
\begin{definition}
For any $l$-tuples $\pi=(\pi_1,\cdots,\pi_l)$, the weight of $\pi$ is given by
	\begin{equation*}
		wt(\pi) :=\prod_{r=1}^l\prod_{(i,j)\in\pi_r}d_{ij}.
	\end{equation*}
For $l,k\in\mathbb{N}$ with $1\leq l\leq k \leq K$, define
 \begin{equation}\label{def:tau}
 \tau_{k,l}(n) :=\sum_{\pi\in \Pi^n_{k,l}} wt(\pi).
 \end{equation}
If $\Pi^n_{k,l}$ is empty, we set $\tau_{k,l}(n)=0$.
\end{definition}
 From $\tau_{k,l}(n)$, we can define the array $z(n)$ and the shape $y_K(n)$.
\begin{definition}\label{def:array}
The array $z(n)=\{z_{k,l}(n):1\leq k\leq K, 1\leq l\leq k \}$ is defined by
\begin{align}\label{def:z}
z_{k,l}(n):= \left\{ \begin{array}{cc} {\displaystyle \frac{\tau_{k,l}(n)}{\tau_{k,l-1}(n)}}, & l\leq n \wedge k ,\\
&\\
\textup{undefined}, & n < l \leq k.
\end{array}\right.
\end{align}
Here we adopt the convention that $\tau_{k,0}(n)=1$.
\end{definition}
\begin{definition}\label{def:shape}
The shape $y_K(n) :=(y_{K,1}(n),\cdots,y_{K,K}(n))$ of the array $z(n)$ is defined by
	\begin{align}\label{def:y}
		y_{K,i}(n):= \left\{ \begin{array}{cc} z_{K,i}(n), & i \leq n \wedge K,\\
		&\\
		\textup{undefined}, & n < i \leq K.
		\end{array}\right.
	\end{align}
\end{definition}
We are ready to define the discrete log-gamma line ensemble.
\begin{definition}\label{log-gamma line ensemble}
The log-gamma line ensemble is defined by	
\begin{align*}
\mathcal{L}_{K,i}(n) :=\left\{\begin{array}{cc}
\log (y_i(n)) & 1\leq i\leq n\wedge K,\\
&\\
\textup{undefined} &  n<i\leq K.
\end{array}\right. 
\end{align*}
\end{definition}
Let us discuss the space of indices on which log-gamma line ensemble is defined. If $n\geq K$, then $\mathcal{L}_{K,i}(n)$ is defined for all $i\in [1,K]_{\mathbb{Z}}$. On the other hand, if $n\in [1,K-1]_{\mathbb{Z}}$, then $\mathcal{L}_{K,i}(n)$ is only defined for all $i\in [1,n]_{\mathbb{Z}}$. See Figure~\ref{5lines} for an illustration. We summarize the the construction of log-gamma line ensemble in the following Table~\ref{paths}.	\\

The rest of this subsection is devoted to prove the $(\Hamd,\Hrw)$-Gibbs property for log-gamma line ensembles, Proposition~\ref{Gibbs_pro}. 

We begin with the Markov property of $y_K(n)$. For brevity, we will denote $y_K(n)$ simply by $y(n)$. It is shown in \cite[Theorem 3.7 and Theorem 3.9]{COSZ} that the process $\{y(n),n\geq K\}$ is Markovian in $n$. Furthermore, there is a function $w:(0,\infty)^K\to (0,\infty)$ such that the transition probability from $y=(y_1,\dots,y_K)$ to $\tilde{y}=(\tilde{y}_1,\dots,\tilde{y}_K)$ is given by: 
\begin{equation}\label{def:transition_kernel}
\frac{w(\tilde{y})}{w(y)}\prod_{i=1}^{K-1}  e^{-\tilde{y}_{i+1}/y_i} \prod_{j=1}^K\bigg(\Gamma(\gamma)^{-1} (y_j/\tilde{y}_j)^\gamma e^{-y_j/\tilde{y}_j} \bigg)\frac{d\tilde{y}_j}{\tilde{y}_j}.
\end{equation}
The explicit form of $w$ can be found in \cite[(3.9)]{COSZ} but we will not rely on such form. The following Lemma~\ref{Gibbs_pro_pre} shows that an $(\Hamd,\Hrw)$-Gibbs property can be derived from the Markov property \eqref{def:transition_kernel}.

\begin{table}[h!]
     \begin{center}
     \begin{tabular}{c p{8cm}}
     \hline
      Paths in environment & \quad  Definition of log-gamma line ensemble $\mathcal{L}_K$   \\ 
    \cmidrule(r){1-2}
     \raisebox{-\totalheight}{\includegraphics[width=8cm, height=5.5cm]{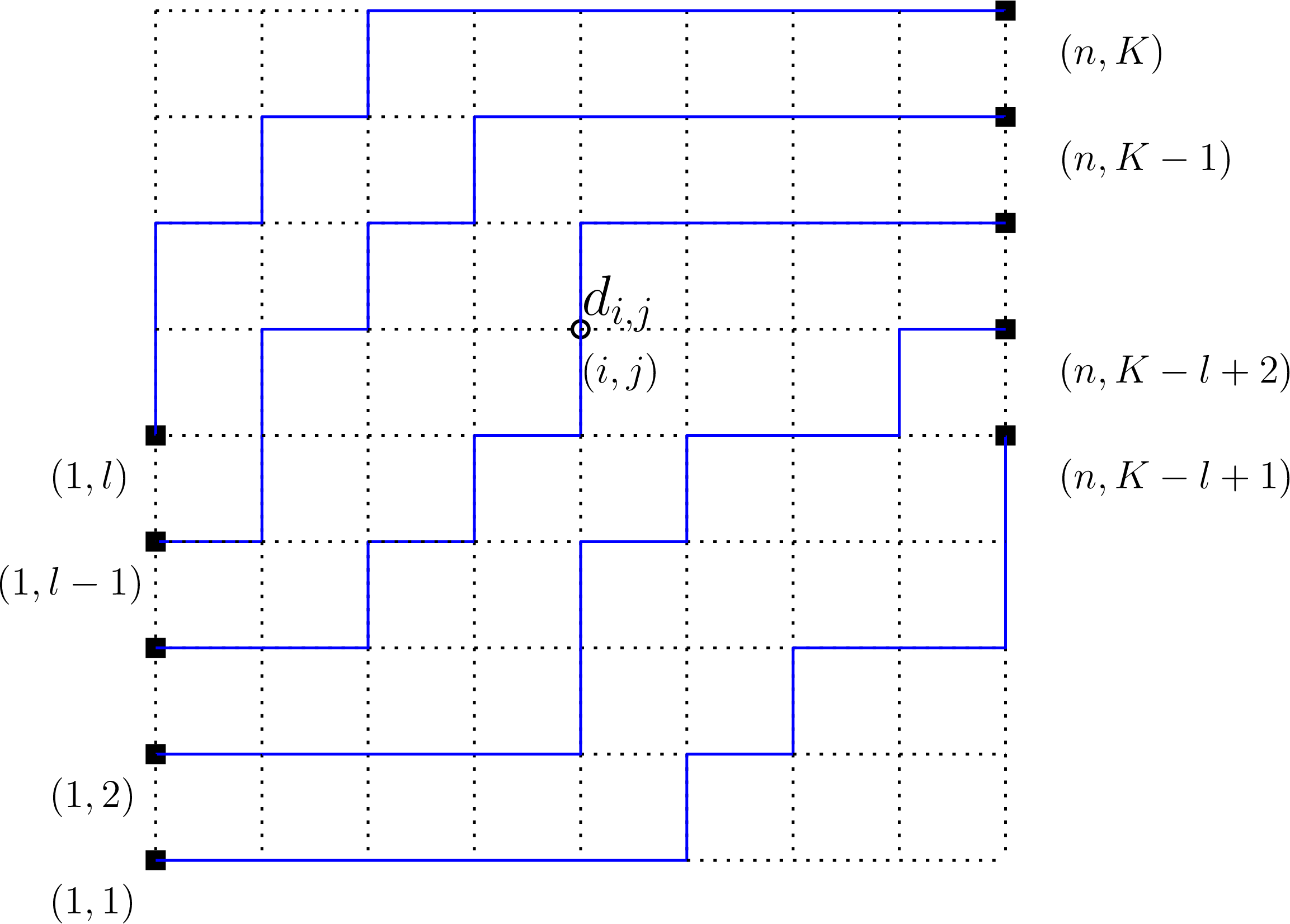}}
      &  Define partition functions as
      \begin{itemize}
      \item  $\displaystyle{z_{K,l}(n) :=\sum_{\pi\in \Pi^n_{K,l}} wt(\pi)},$
      \item  $\displaystyle{wt(\pi) :=\prod_{r=1}^l\prod_{(i,j)\in\pi_r}d_{ij}}$.
      \end{itemize}
      \medskip     
      Define the {\em log-gamma line ensemble} $\mathcal{L}_K$ as
      \begin{itemize}
      \item $\displaystyle{\mathcal{L}_{K,l}(n) := \log{z_{K,l}(n)}- \log{z_{K,l-1}(n)}}, \ l\leq n$
      \item undefined for $l >n$.
\end{itemize}       
      \\
      \hline
      \end{tabular}
      \caption{Summary of the process of constructing $\mathcal{L}(n), n\geq 1$. We regard $n$ as the evolving coordinate. We take the empty sum to equal zero in the definition of $z_{K,l}(n)$ when $n<l$.}
      \label{paths}
      \end{center}
      \end{table}

\begin{figure}
\includegraphics[height = 5cm]{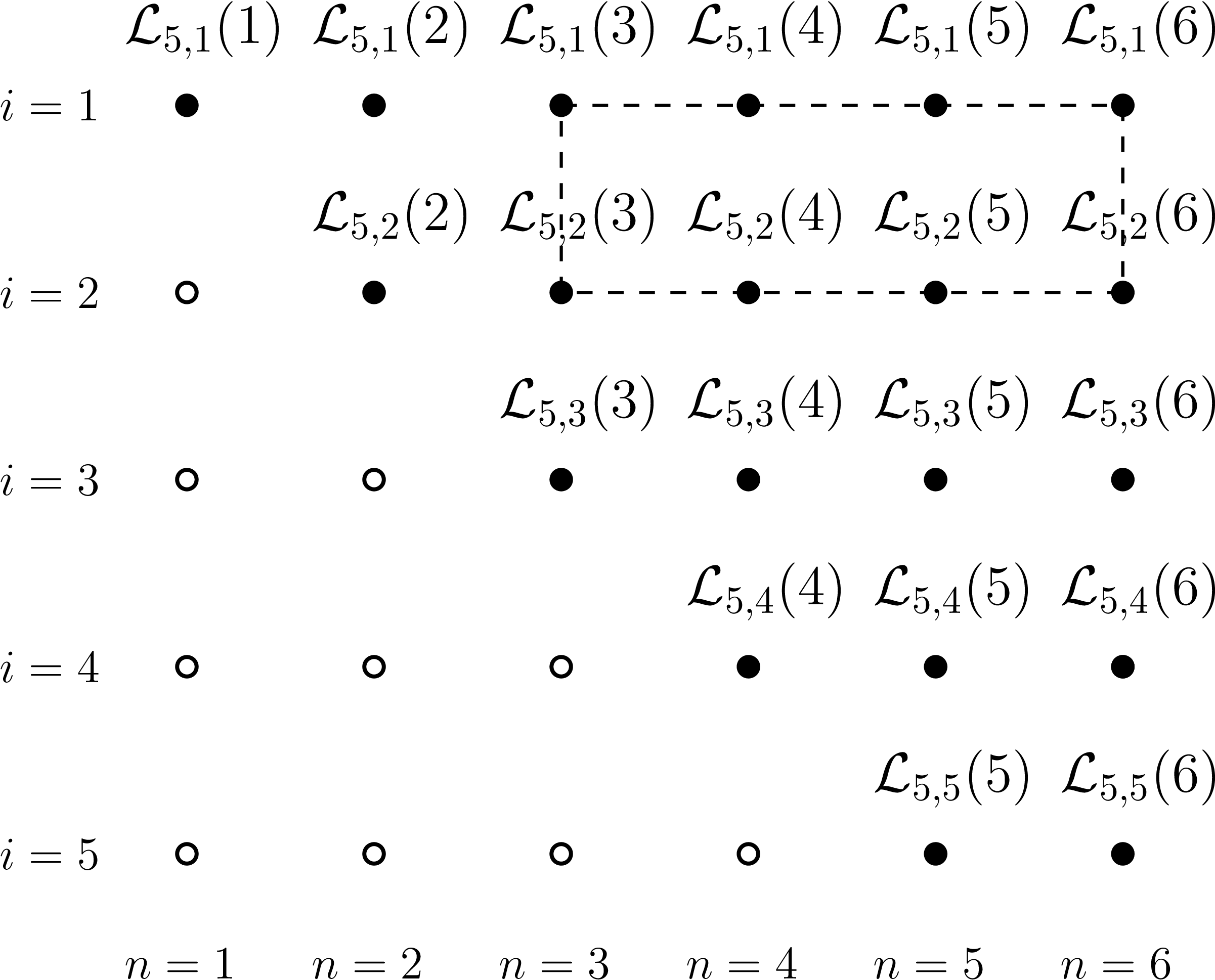}
\caption{Illustration of the log-gamma line ensemble construction with $K=5$. There are five curves $\cL_{5,1}(n)$ through $\cL_{5,5}(n)$. The values of $\cL_{5,i}(n)$ are not defined on the hollow points. The Gibbs property holds for $n \geq 5$ as a consequence of Lemma~\ref{Gibbs_pro_pre} and the Markov property for $\cL_5(n)$ when $n \geq 5$ with transition kernel from \eqref{def:transition_kernel}. We are in need of the same Gibbs property when $n <5$, e.g. the dashed rectangle region.}
\label{5lines}
\end{figure}

\begin{lemma}\label{Gibbs_pro_pre}
Fix any $\hat{y}^0\in (0,\infty)^K $. Let $\hat{y}(n), n\in\mathbb{N}$ be the Markov chain on $(0,\infty)^K$ with the initial state $\hat{y}(0) = \hat{y}^0$ and the transition kernel given by \eqref{def:transition_kernel}. Define the $[1,K]_{\mathbb{N}}\times \mathbb{N}$-indexed line ensemble $\hat{\mathcal{L}}$ by
		\begin{equation*}
		\hat{\mathcal{L}}_{i}(n)= \log ( \hat{y}_i(n)).
		\end{equation*}
Then $\hat{\mathcal{L}}$ enjoys a $(\Hamd,\Hrw)$-Gibbs property with
		\begin{equation}\label{log-gamma_line_rw_pdf}
		\Hrw(x)=\log \Gamma(\gamma)+\gamma x+e^{-x},
		\end{equation}
		and
		\begin{equation}\label{log-gamma_line_hamil}
		\Hamd(\rectangle(\mathcal{L},k,u))=\exp\big(\mathcal{L}_{k+1}(u+1)-\mathcal{L}_k(u)\big).
		\end{equation}
	\end{lemma}
\begin{proof}
We write $\hat{\mathcal{L}}(n)$ for the vector $(\hat{\mathcal{L}}_{ 1}(n),\dots,\hat{\mathcal{L}}_{ K}(n) )$. Because $\hat{y}(n)$ is Markov in $n$, $\hat{\mathcal{L}}_K(n)$ is Markov in $n$ as well. Moreover, from \eqref{def:transition_kernel}, there is a function $\alpha:\mathbb{R}^K\to (0,\infty)$ such that the transition kernel from $\hat{\mathcal{L}}(u)$ to $\hat{\mathcal{L}}(u+1)$ is given by 
\begin{equation}\label{transition_log}
\begin{split}
\frac{\alpha( \hat{\mathcal{L}} (u))}{\alpha( \hat{\mathcal{L}} (u+1))} &\prod_{i=1}^{K-1} \exp\Bigg[-\Hamd(\rectangle(\hat{\mathcal{L}}  ,i,u)) \Bigg] \\
\times&\prod_{j=1}^K \exp\bigg(-\Hrw(\hat{\mathcal{L}}_{ j}(u+1)-\hat{\mathcal{L}}_{ j}(u))\bigg){d\hat{\mathcal{L}}_{j}(u+1)},
\end{split}
\end{equation} 
where $\Hamd$ and $\Hrw$ are given by \eqref{log-gamma_line_hamil} and \eqref{log-gamma_line_rw_pdf} respectively.  The explicit form of $\alpha$ is not consequential for us. Fix $k_1\leq k_2$ with $k_1,k_2\in [1,K]_{\mathbb{Z}}$, and $a<b$ with $a,b\in\mathbb{N}$. We denote $\vec{x} = \big(\hat{\mathcal{L}}_{ k_1}(a),\ldots, \hat{\mathcal{L}}_{k_2}(a)\big)$ as the entrance data,  $\vec{y} = \big(\hat{\mathcal{L}}_{ k_1}(b),\ldots, \hat{\mathcal{L}}_{ k_2}(b)\big)$ as the exit data,  $f=\hat{\mathcal{L}}_{ k_1-1}|_{(a,b)}$ be the upper boundary curve and  $g=\hat{\mathcal{L}}_{ k_2+1}|_{(a,b)}$ be the lower boundary curve and we adopt the convention that $\hat{\mathcal{L}}_0 = +\infty$ and $\hat{\mathcal{L}}_{K+1} = -\infty$. 
		
Given a continuous function $F:C([k_1,k_2]_{\mathbb{Z}}\times [a,b]_{\mathbb{Z}},\mathbb{R})$, we derive from \eqref{transition_log} that
\begin{equation}\label{loggamma_gibbs}
\begin{split}
		&\quad \EE\Big[F\big(\hat{\mathcal{L}} \vert_{[k_1,k_2]_{\Z}\times \Lambda_d(a,b)}\big)\big\vert \mathcal{F}_{\textrm{ext}}\big([k_1,k_2]_{\Z}\times (a,b)_{\mathbb{Z}}\big)\Big]\\
		&\propto \int F\big(\hat{\mathcal{L}}  \vert_{[k_1,k_2]_{\Z}\times \Lambda_d(a,b)}\big)  \prod_{u\in \Lambda_d[a,b]}\prod_{i=k_1}^{k_2} \exp\Bigg[-\Hamd(\rectangle(\hat{\mathcal{L}}  ,i,u)) \Bigg]\times \\
		&\qquad\prod_{j=k_1}^{k_2} \exp\bigg(-\Hrw(\hat{\mathcal{L}}_{ j}(u+1)-\hat{\mathcal{L}}_{ j}(u))\bigg){d\hat{\mathcal{L}}_{ j}(u)},
	\end{split}
\end{equation}
In view of \eqref{loggamma_gibbs}, we have
\begin{align*}
 \EE\Big[F\big(\hat{\mathcal{L}} \vert_{[k_1,k_2]_{\Z}\times \Lambda_d(a,b)}\big)\big\vert \mathcal{F}_{\textrm{ext}}\big([k_1,k_2]_{\Z}\times (a,b)_{\mathbb{Z}}\big)\Big]=\mathbb{E}_{\Hamd,\Hrw}^{k_1,k_2,\Lambda_d(a,b),\vec{x},\vec{y},f,g}[F(\tilde{\mathcal{L}})]. 
\end{align*}
See Definition~\ref{def:RWB_Gibbs_ensemble}. This proves the $(\Hamd,\Hrw)$-Gibbs property.
\end{proof}

Lemma~\ref{Gibbs_pro_pre} shows that $\mathcal{L}_{K}(n)$ defined in Definition~\ref{log-gamma line ensemble} has the $(\Hamd,\Hrw)$-Gibbs property for $n\geq K$. However, for $n<K$, the $(\Hamd,\Hrw)$-Gibbs property is not directly implied by Lemma~\ref{Gibbs_pro_pre}. The main reason is that \eqref{def:transition_kernel} is the Markov kernel for $y(n)$ only for $n\geq K$. We resolve this issue by showing $\mathcal{L}_K$ can be approximated by a sequence $\hat{\mathcal{L}}^M$ which enjoys the $(\Hamd,\Hrw)$-Gibbs property.


Fix $\rho=(\rho_{1}, \rho_2,\dots,\rho_K)=\left(2^{-1}(K-1) ,2^{-1}(K-1) -1,\cdots, -2^{-1}(K-1)  \right).$ For $M>0$, let $y^{0,M} :=(e^{-M\rho_{1}},\dots, e^{-M\rho_{K}}) $. Define $y^M(n) :=(y^M_i(n))_{1\leq i\leq K, n\geq 1}$ as a Markov chain in $\mathbb{Y}_K$ with the initial state $y^M(0)=y^{0,M}$ and transition kernel given by \eqref{def:transition_kernel}. Define $\hat{\mathcal{L}}^M_{i}(n):=\log y^M_i(n)$ as a $[1,K]_{\mathbb{Z}}\times\mathbb{N}$ indexed line ensemble. From Lemma~\ref{Gibbs_pro_pre}, $\hat{\mathcal{L}}^M$ satisfies the $(\Hamd,\Hrw)$-Gibbs property.

The following proposition is proved in \cite[Proposition 5.3]{COSZ}, which says that $\{y^M(n), n\geq 1\}$ converges weakly to $y(n)$ on compact sets. As a consequence, $\hat{\mathcal{L}}^M$ converges weakly to $\mathcal{L}_L$ on compact sets.

\begin{proposition}\label{pro:COSZ}
Let $\{y^M(n), n\geq 1\}$ be the Markov chain on $(0,\infty)^K$ defined as above and let $y_i(n),\ 1\leq i\leq n\wedge K, n\geq 1$ be defined as in \eqref{def:y}. Then for any $n_0\geq 1$, $\{y^M_i(n): 1\leq i\leq n\wedge K,\ 1\leq n\leq n_0\}$ converges weakly to $\{y_i(n): 1\leq i\leq n\wedge K,\ 1\leq n \leq n_0\}$ as $M$ goes to infinity.
\end{proposition}

\begin{remark}
It is proved in \cite[Proposition 5.3]{COSZ} not only the weak convergence of the shape $y(n)$ but also the weak convergence of the array $z(n)$.
\end{remark}

\begin{proposition}\label{Gibbs_pro}
Let $\mathcal{L}_K$ be the line ensemble defined in Definition \ref{log-gamma line ensemble} and $\Hamd,$ $\Hrw$ be defined in \eqref{log-gamma_line_hamil} and \eqref{log-gamma_line_rw_pdf} respectively. Then $\mathcal{L}_K$ satisfies the $(\Hamd,\Hrw)$-Gibbs property in the region it is defined. Precisely, for any $k_1\leq k_2$ with $k_1,k_2\in [1,K]_{\mathbb{Z}}$ and $a<b$ with $a,b \in \mathbb{N}$ and $k_2\wedge K\leq a$, \eqref{loggamma_gibbs} holds with $\hat{\mathcal{L}}$ replaced by $\mathcal{L}_K$.
\end{proposition}

\begin{proof}
From Proposition~\ref{pro:COSZ} and Lemma~\ref{Gibbs_pro_pre}, $\mathcal{L}_K$ is the weak limit of $\hat{\mathcal{L}}^M $ as $M$ goes to infinity, which enjoys the $(\Hamd,\Hrw)$-Gibbs property. Therefore it suffices to show $(\Hamd,\Hrw)$-Gibbs property is preserved under the weak limit. The proof proceeds similarly as the one of Theorem~\ref{theorem:main} (2), thus we omit repeated details and focus on the part that needs modification. 

Fix a compact region $[k_1,k_2]_{\mathbb{Z}}\times [a,b]_{\mathbb{Z}}$ in which $\mathcal{L}_K$ is defined. Through Skorohod representation theorem, in $[k_1,k_2]_{\mathbb{Z}}\times [a,b]_{\mathbb{Z}}$ one can couple $\mathcal{L}_K$ with the converging sequence $\hat{\mathcal{L}}^M $ in the same probability space such that $\hat{\mathcal{L}}^M $ converges to $\mathcal{L}_K$ almost surely. One can reformulate the $(\Hamd,\Hrw)$-Gibbs property using a resampling invariance  in the same way as \eqref{resample}. We can then essentially follow the arguments in the proof of Theorem \ref{theorem:main} (2). The only exception is that the analogous statement \eqref{KMT_application} can be derived by Lemma \ref{lem:RW_coupling} instead of the KMT coupling. 
\end{proof} 

\subsection{Scaled log-gamma line ensemble}\label{sec:5.2}
In this subsection, we introduce the weak noise scaling for log-gamma line ensembles. Applying Theorem~\ref{theorem:main}, we show that a sequence of scaled log-gamma line ensemble is tight when restricted on compact sets. Furthermore, any subsequential limit satisfies the $\Ham$-Brownian Gibbs property with $\Ham(x)=e^x$.\\

We consider a sequence of log-gamma line ensembles in which the parameters $\gamma$ and $K$ changes. Let $\gamma_N = N^{1/2}$. We denote by $\mathcal{L}^N_K $ the line ensemble defined in Definition~\ref{log-gamma line ensemble} with $\gamma$ being $\gamma_N$. 

We start with the weak noise scaling and the convergence of the lowest indexed curve. The following result is a consequence of \cite[Theorem 2.7]{AKQ} with slight modifications of their arguments, which we explain in the proof below.
\begin{proposition}\label{pro:AKQ}
Consider the weak noise scaling for lowest indexed curve as
\begin{equation}\label{def:1227}
\begin{split}
 \tilde{\cL}^N_1(t,u) := &\mathcal{L}^N_{Nt/2,1}\left( 2^{-1}Nt+ {N}^{1/2}u\right)+2^{-1}\log {N}-\log 2\\
 &-(Nt+ {N}^{1/2}u)\left( \log 2-\log( {N}^{1/2}-1) \right) .
\end{split}
\end{equation}
Then $\tilde{\cL}^N_1(t,u)$ converges weakly to $\log \mathcal{Z}_{\sqrt{2}}(t,u)$ in the topology of uniform convergence on compact sets of $\{t\in \mathbb{R}_{+},\ u\in\mathbb{R}\}$.

Here $\mathcal{Z}_{\sqrt{2}}(t,u)$ is defined by the following chaos expansion with convention $t_0 = 0, u_0 = 0$,
\begin{align} 
\mathcal{Z}_{\sqrt{2}}(t,u) := &\rho(t,u)+\sum_{k=1}^{\infty} (\sqrt{2})^k\int_{\Delta_k(0,t]}\int_{\mathbb{R}^k}\prod_{i=1}^k W(t_i,u_i)\times\label{eq:SHE_Chaos}\\
&\rho (t_{i}-t_{i-1},u_{i}-u_{i-1})\rho(t-t_k,u-u_k) du_i dt_i\nonumber.
\end{align}
In the above expression, $W(t,u)$ is a white noise on $\mathbb{R}_+\times\mathbb{R}$ with covariance structure $\mathbb{E}[W(t,x)W(s,y)]=\delta(t-s)\delta(x-y)$. $\rho(t,u)$ is the standard Gaussian heat kernel such that
\begin{align}\label{def:heatkernel}
\rho(t,u)=\frac{e^{-u^2/2t}}{\sqrt{2\pi t}}.
\end{align}
And the integral is over a $k$-dimensional simplex $\Delta_k(0,t] = \{0=t_0<t_1<t_2<\dots<t_k \leq t\}$. 
\end{proposition}

\begin{proof}
We prove the above result by explaining how the convergence follows from the arguments for \cite[Theorem 2.7]{AKQ}. Note that \cite[Theorem 2.7]{AKQ} shows that, under weak noise scaling (called intermediate disorder regime in \cite{AKQ}), a modified point-to-point partition function of directed polymer model in dimension 1+1, denoted by 
\begin{align}\label{equ:12261147}
 2^{-1}N^{1/2} \mathfrak{Z}^{\omega}(Nt+N^{1/2} u,N^{1/2} u;N^{-1/4}), 
\end{align}
converges to the chaos series $\mathcal{Z}_{\sqrt{2}}(t,x)$ in \eqref{eq:SHE_Chaos}, which is the solution to stochastic heat equation with multiplicative noise. The random environment $\omega=\{\omega(i,j), i,j\in\mathbb{N}\}$ considered therein is of mean zero, variance one and has sixth moments. 

To match the above result with the log-gamma polymer model, we need to allow $\omega$ to change in $N$. For $N\in\mathbb{N}$, let $\omega_N$ be an i.i.d. random environment in which each $\omega_N(i,j)$ is distributed as 
\begin{align}\label{def:omega_N}
\omega_N(i,j) \overset{(d)}{=} N^{1/4}\left(\frac{\textup{Inv-Gamma}(\sqrt{N})}{\mathbb{E}\left[\textup{Inv-Gamma}(\sqrt{N})\right]}-1\right).
\end{align}
From the definition $\mathfrak{Z}^{\omega}$ in \cite[(2.4)]{AKQ} and  the definition of$\tau$ in \eqref{def:tau} and $\mathbb{E}\left[\textup{Inv-Gamma}( {N}^{1/2})\right] =  ( {N}^{1/2}-1)^{-1}$, \eqref{equ:12261147} with $\omega$ replaced by $\omega_N$ equals
\begin{align}\label{equ:12261149}
 2^{-1}N^{1/2} \cdot  2^{-(Nt+N^{1/2} u)}(N^{1/2}-1) ^{ Nt+N^{1/2}u }\,\tau_{ Nt/2,1}\left(2^{-1} Nt+N^{1/2}u\right).
\end{align}
From Definition \ref{log-gamma line ensemble}, we have
\begin{align}\label{equ:12271203}
\cL^N_{Nt/2,1}\left( 2^{-1}Nt+ {N}^{1/2}u\right) = \log \tau_{Nt/2,1}\left(2^{-1}Nt+ {N}^{1/2}u\right).
\end{align} 
Comparing \eqref{def:1227}, \eqref{equ:12271203}, \eqref{equ:12261147} and  \eqref{equ:12261149}, we get
$$ \tilde{\cL}^N_1(t,u)=\log \left(2^{-1}N^{1/2} \mathfrak{Z}^{\omega_N}(Nt+N^{1/2} u,N^{1/2} u;N^{-1/4}),  \right).$$

It remains to show the convergence of $2^{-1}N^{1/2} \mathfrak{Z}^{\omega_N}(Nt+N^{1/2} u,N^{1/2} u;N^{-1/4})$. From \eqref{def:omega_N} and \eqref{def:gamma}, $\omega_N$ has mean zero and variance $1+O(N^{-1/2})$. Moreover, the sixth moment of $\omega_N$ is $15+O(N^{-1/2})$. Under such conditions, one can run the same arguments in \cite{AKQ} to obtain the convergence. 
\end{proof}

In order the match the $\mathbf{H}$-Brownian Gibbs property, we perform one more scaling and define the scaled log-gamma line ensemble as follows.
\begin{definition}
Let 
\begin{equation}\label{def:Lbar1227}
\begin{split}
\overbar{\mathcal{L}}^N_i(t,u):=&\mathcal{L}^N_{Nt/8,i}\left( 8^{-1}Nt+2^{-1}{N}^{1/2} u\right)+2^{-1}\log N-\log 2\\
&-\left(4^{-1} Nt+2^{-1}  {N}^{1/2}u\right)\left( \log 2-\log( {N}^{1/2}-1) \right).
\end{split}
\end{equation} 
Notice that $\overbar{\mathcal{L}}^N_i(t,u)$ is defined for $N\in\mathbb{N}$, $t\in  {8}N^{-1} \mathbb{N}$, $u\in {2}{ {N}^{-1/2}}\mathbb{Z}$ and $i\leq (8^{-1}Nt )\wedge (8^{-1}Nt +2^{-1} {N}^{1/2}u )$. 
\end{definition}
The relation between $\overbar{\mathcal{L}}^N_1$ and $\tilde{\mathcal{L}}^N_1$ (defined in \eqref{def:1227}) is $\overbar{\mathcal{L}}^N_1(t,u)=\tilde{\mathcal{L}}^N_1(4^{-1}t ,2^{-1}u )$. \\

Now we are ready to state the main application of Theorem~\ref{theorem:main} as follows. We fix $t =1$ for notation simplicity and the same result holds for any $t>0$ by the same argument modulo the modification that $\overbar{\mathcal{L}}^N_1(t,u) + \frac{u^2}{2t}$ converges to a stationary process.

\begin{theorem}\label{thm:loggamma}
For any $N\in\mathbb{N}$, let $\overline{\mathcal{L}}^N	$ be the discrete line ensemble defined in \eqref{def:Lbar1227}. Fix arbitrary $k\in \mathbb{N}$ and $T>0$. Then restriction of  $\{\overbar{\mathcal{L}}^{N}_j(1,u)$ to $j\in [1,k]_{\mathbb{Z}}$ and $u\in [-T,T]$\ is tight as $N$ varies. Moreover, any subsequential limit line ensemble satisfies the $H$-Brownian Gibbs property with $\Ham(x)= e^x$. 
\end{theorem}

The rest of the section is devoted to prove Theorem~\ref{thm:loggamma}. To simplify notations, we denote $\overbar{\mathcal{L}}^N_i(u)=\overbar{\mathcal{L}}^N_i(1,u)$. We need to verify the assumptions in Theorem~\ref{theorem:main} are satisfies for $\overbar{\mathcal{L}}^N$. We recall those assumptions here for readers convenience.
\begin{itemize}
\item $\overbar{\mathcal{L}}^N$ is a $(\Hamd^N,\HrwN)$-line ensemble for some $\Hamd^N$ and $\HrwN)$.
\item $\overbar{\mathcal{L}}^N_1(u)+u^2/2$ converges weakly to a stationary process.
\item Assumptions A1-A4 in Subsection~\ref{Assumptions} hold. 
\end{itemize}

 We begin by showing the $(\Hamd^N,\HrwN)$-Gibbs property for $\overline{\mathcal{L}}^N$. From Proposition~\ref{Gibbs_pro}, $\cL_K$ satisfies the $(\Hamd,\Hrw)$-Gibbs property for $\Hamd$ and $\Hrw$ given by \eqref{log-gamma_line_hamil} and \eqref{log-gamma_line_rw_pdf} respectively. Let $d_N=2N^{-1/2}$ and $\Lambda^N_d=d_N\mathbb{Z}$. Through a direct calculation, $\overbar{\mathcal{L}}^N_i(u)$ satisfies a $(\Hamd^N,\HrwN)$-random walk Gibbs property over the region 
 \begin{equation*}
 \left\{u\in 2N^{-1/2}\mathbb{Z} ,\ i\leq (8^{-1}N )\wedge (8^{-1}N +2^{-1} {N}^{1/2}u )\right\}
 \end{equation*}
with the local interaction Hamiltonian
\begin{align}\label{loggamma_H^N}
\Hamd^N(\rectangle(\overbar{\mathcal{L}}^N,k,u)):= 2(N^{1/2}-1)^{-1} \exp\left( \overbar{\mathcal{L}}^N_{k+1}(u+2N^{-1/2})-\overbar{\mathcal{L}}^N_k(u) \right),
\end{align}
and the random walk Hamiltonian
\begin{align}\label{loggamma_H^RWN}
\HrwN(x):=&\log\Gamma ( {N}^{1/2})+ {N}^{1/2}\left( x-\log ( {N}^{1/2}-1)+\log 2 \right)\\
&+\exp\left( -x+\log ({N}^{1/2}-1)-\log 2 \right)\nonumber.
\end{align}

Next, we show that $\overbar{\mathcal{L}}^N_1(u)+u^2/2$ converges weakly to a stationary process. Through the scaling invariance of the white noise, $W(at,bu)$ has the same distribution as $ (ab)^{-1/2} W(t,u)$. From \eqref{eq:SHE_Chaos}, this implies that for any $\lambda,\beta>0$, $\mathcal{Z}_{\beta}(\lambda^2 t,\lambda u)$ has the same distribution as $\lambda^{-1}\mathcal{Z}_{\lambda^{1/2}\beta}(t,u)$. Together with Proposition \ref{pro:AKQ}, $\overbar{\cL}^N_1(t,u)$ converges weakly to $\log \mathcal{Z}_1(t,u)+\log 2$. In \cite[Proposition 2.3]{AKQ}, it is shown that for any  $t>0$, $\mathcal{Z}_\beta(t,u)/\rho(t,u)$ is a stationary process in $u$ where $\rho(t,u)$ is the heat kernel in \eqref{def:heatkernel}. Therefore, $\overbar{\cL}^N_1(1,u)+ {u^2}/{2}$ converges weakly to a stationary process. \\

We now verify Assumptions A1 and A2. Assumption A2 can be checked directly from the form in \eqref{loggamma_H^RWN}. The same holds for Assumptions A1 (1) and we focus on Assumptions A1 (2). 

Since $\Hamd^N$ depends only on the second and the sixth entry, to check Assumption A1(2), it's sufficient to consider $\vec{a}= (a_2, a_6),\ \vec{b}=(b_2,b_6)\in(\mathbb{R}\cup\{\pm\infty\})^2$ with $a_2\geq b_2$, $a_6\geq b_6$ and $a_2=b_2$ or $a_6=b_6$. For the case $a_2=b_2$, let $\delta>0$ be a fixed number and
\begin{align*}
\vec{a}'=(a_2+\delta,a_6),\ \vec{b}'=(b_2+\delta,b_6),
\end{align*}
Then 
\begin{align*}
-\Hamd^N(\vec{a}')+\Hamd^N(\vec{a})= &{2(N^{1/2}-1)^{-1}(e^{a_6-a_2}-e^{a_6-a_2-\delta}) },\\
-\Hamd^N(\vec{b}')+\Hamd^N(\vec{b})= & 2(N^{1/2}-1)^{-1}(e^{b_6-b_2}-e^{b_6-b_2-\delta}).
\end{align*}
From the convexity of $e^x$, $a_6\geq b_6$, $a_2=b_2$ and $\delta>0$, we have
\begin{align*}
-\Hamd^N(\vec{a}')+\Hamd^N(\vec{a})\geq -\Hamd^N(\vec{b}')+\Hamd^N(\vec{b}). 
\end{align*}
We then deal with the case $a_6=b_6$. Let $\delta>0$ be a fixed number and
\begin{align*}
\vec{a}'=(a_2,a_6+\delta),\ \vec{b}'=(b_2,b_6+\delta),
\end{align*}
Then 
\begin{align*}
-\Hamd^N(\vec{a}')+\Hamd^N(\vec{a})= &{2(N^{1/2}-1)^{-1}(e^{a_6-a_2}-e^{a_6-a_2+\delta}) },\\
-\Hamd^N(\vec{b}')+\Hamd^N(\vec{b})= &{2(N^{1/2}-1)^{-1}(e^{b_6-b_2}-e^{b_6-b_2+\delta}) }.
\end{align*}
From the convexity of $e^x$, $a_6= b_6$, $a_2\geq b_2$ and $\delta>0$, we have
\begin{align*}
-\Hamd^N(\vec{a}')+\Hamd^N(\vec{a})\geq -\Hamd^N(\vec{b}')+\Hamd^N(\vec{b}). 
\end{align*}
This finishes the verification of Assumption A1. \\

Next, we check Assumption A3. Recall that $d_N=2N^{-1/2}$ and $\Lambda^N_d=2N^{-1/2}\mathbb{Z}$. Fix arbitrary $k\in\mathbb{Z}$, $a<b$ with $a,b\in\Lambda^N_d$ and $\mathcal{L}=(\cL_k,\cL_{k+1})\in C([k,k+1]_{\mathbb{Z}}\times [a,b],\mathbb{R})$, we compute
\begin{align*}
\sum_{u\in \Lambda^{N}_d(a,b)} \Hamd^N(\rectangle(\mathcal{L},k,u))=&\sum_{u\in \Lambda^{N}_d(a,b)}  {2}( {N}^{1/2}-1)^{-1} \exp\left( {\mathcal{L}}_{k+1}(u+2N^{-1/2} )-{\mathcal{L}}_k(u) \right)\\
\leq & (1-N^{-1/2})^{-1} e^{\omega_{(a,b)}(\mathcal{L}_{k+1},d_N )}  \sum_{u\in \Lambda^{N}_d(a,b)} d_N^{-1}  \exp\left( {\mathcal{L}}(u)-{\mathcal{L}}_k(u) \right)\\
\leq &(1-N^{-1/2})^{-1} e^{3\omega_{(a,b)}(\mathcal{L} ,d_N)} \int_a^b \exp(\mathcal{L}_{k+1}(u)-\mathcal{L}_{k}(u))du.
\end{align*}
Similarly,
\begin{align*}
 \sum_{u\in \Lambda^{N}_d(a,b)} \Hamd^N(\rectangle(\mathcal{L},k,u)) \geq e^{-3\omega_{(a,b)}(\mathcal{L} ,d_N)}  \int_a^b \exp(\mathcal{L}_{k+1}(u)-\mathcal{L}_{k}(u))du.
\end{align*}
These two yield Assumption A3.\\

Lastly, we turn to Assumption A4. For any $\gamma>0$, we denote by 
$$Y(\gamma)=-\log\big( \textup{Inv-Gamma} (\gamma)\big)$$ 
the log-gamma random variable with parameter $\gamma$. From \eqref{densityinvgamma} and \eqref{log-gamma_line_rw_pdf}, $\exp(-\HrwN(x))$ is the density function of the random variable $-Y(\sqrt{N})+\log (\sqrt{N}-1)-\log 2$. For any $L>0$ with  $ d_N^{-1} L \in\mathbb{N}$ and $z\in\mathbb{R}$, let $\bar{S}^N_{L,z}$ be the random walk bridges defined in Definition~\ref{def:RW_ensemble} woth $\HrwN$ and $d_N \mathbb{Z}$. We need to verify that $\bar{S}^N_{L,z}$ satisfies the estimate in Assumption A4.\\

To this end, we rely on \cite[Corollary 8.1]{DW}, which provides the desired estimates for normalized random variables. We start by recording the result of \cite[Corollary 8.1]{DW}. Let $m(\gamma)$ and $\sigma(\gamma)^2$ be the mean and variance of $Y(\gamma)$ respectively. For $j\in\mathbb{N}$, let $X_j(\gamma)$ be i.i.d. random variables with $X_1(\gamma) \overset{(d)}{=} \sigma(\gamma)^{-1}(Y(\gamma)-m(\gamma)) $. For any $n\in\mathbb{N}$ and $k\in [0,k]_{\mathbb{Z}}$, let $S_{n,z}(k;\gamma)$ be the random walk bridge defined in \eqref{equ:1225716} with $X_j$ replaced by $X_j(\gamma)$. For general $u\in [0,n]$, we define $S_{n,z}(u;\gamma)$ through linear interpolation. 
\begin{corollary}\label{cor:DW}
For any $b>0$ and $\gamma_0>0$, there exists constants $0<C,a,\alpha'<\infty$ such that the following holds. For every positive integer $n$ and $\gamma\geq \gamma_0$, there is a probability space $(\Omega_n,\mathcal{B}_n,\mathbb{P}_n)$ on which are defined a Brownian bridge $B_1(u)$, $u\in [0,1]$ and a family of processes $\ S_{n,z}(u ;\gamma),\ u\in [0,n]$ and $ z\in\mathbb{R}$ such that for any $r\geq 0$,
\begin{align}\label{KMT_DW}
\mathbb{P}_n\left( \sup_{0\leq u\leq n}\left|n^{1/2} B_1(n^{-1} u )+n^{-1}zu -S_{n,z}(u;\gamma)\right|\geq r \log n \right)\leq C n^{\alpha'-ar}e^{bz^2/n}.
\end{align}
Moreover, $S_{n,z}(u;\gamma)$ is realized as a measurable map $\mathbf{F}_n:\mathbb{R}\times\Omega_n\to C([0,n],\mathbb{R})$. In particular, $S_{n,z}(u;\gamma)$ is measurable in $z$ as a function from $\mathbb{R}$ to $C([0,n],\mathbb{R})$.
\end{corollary}

\begin{proof}
The estimate \eqref{KMT_DW} is a direct consequence of \cite[Corollary 8.1]{DW} and we focus on the measurability of $\mathbf{F}_n$. To simplify the notation, we drop the dependence of $\gamma$ and denote $S_{n,z}(u;\gamma)$ by $S_{n,z}(u)$ \\

The construction of $\mathbf{F}_n$ is based on an induction argument. For $n=1$, one take $\Omega=\{\textup{pt}\}$ be a set with a single element and $\big(\mathbf{F}_1(z,\textup{pt})\big)(u)=uz$.

 Assume $n\geq 2$ and $n=2m$ is even for simplicity. By the induction hypothesis, for $i=1,2$, there exists $(\Omega^i_m,\mathcal{B}^i_m,\mathbb{P}^i_m)$ and measurable maps $\mathbf{F}^i_m:\mathbb{R}\times \Omega_m^i\to C([0,m],\mathbb{R})$ such that $\mathbf{F}^i_m(z,\cdot)$ is distributed as ${S}_{m,z}$. From Lemma~\ref{lem:RW_middle}, there exists a continuous function  $G:\mathbb{R}\times (0,1)\to \mathbb{R}$ such that under the Lebesgue measure on $(0,1)$, $G(z,v)$ is distributed as $S_{n,z}(m)$. From Lemma~\ref{lem:RW_coupling}, there exist a probability space $(\Omega_{\textup{ind}},\mathcal{B}_{\textup{ind}},\mathbb{P}_{\textup{ind}})$ and a measurable map $\mathbf{G}_{\textup{ind}}:\mathbb{R}\times \Omega_{\textup{ind}}\to C([0,n],\mathbb{R})$ such that $\mathbf{G}_{\textup{ind}}(z,\cdot)$ is distributed according to $S_{n,z}$. 
 
We are ready to construct $\Omega_n$ and $\mathbf{F}_n$. Let $\Omega_n=\Omega^1_m\times\Omega_{\textup{mid}}\times\Omega^2_m\times\Omega_{\textup{ind}}$. Set $\mathcal{B}_n$ and $\mathbb{P}_n$ be the product sigma field and measure respectively. Fix some $\varepsilon_0>0$. For $\omega=(\omega^1_m,\omega_{\textup{mid}},\omega^2_m,\omega_{\textup{ind}})\in\Omega$ and $z\in\mathbb{R}$, $\mathbf{F}_n(z,\omega)$ is defined by
\begin{align*}
\big(\mathbf{F}(z,\omega)\big)(u):=\left\{
\begin{array}{cc}
\big(\mathbf{G}_{\textup{ind}}(z,\omega_{\textup{ind}})\big)(u), & |z|\geq \varepsilon_0 n,\\
\big(\mathbf{F}^1_{m}(G(z,\omega_{\textup{ind}}),\omega^1_{m})\big)(u), & |z|< \varepsilon_0 n\ \textup{and}\ u\in [0,m],\\
G(z,\omega_{\textup{ind}})+ \big(\mathbf{F}^2_{m}(z-G(z,\omega_{\textup{ind}}),\omega^2_{m})\big)(u), & |z|< \varepsilon_0 n\ \textup{and}\ u\in (m,n].
\end{array}
\right.
\end{align*}
From the measurability of $\mathbf{F}^1_{m}$, $\mathbf{F}^2_{m}$, $G$ and $\mathbf{G}_{\textup{ind}}$, $\mathbf{F}_n$ is measurable.
\end{proof}

Denote\ $q(N) :=\log( {N}^{1/2}-1)-\log 2-m( {N}^{1/2})$ and we will use $m,\sigma,q$ as shorthand for $m( {N}^{1/2}),\sigma( {N}^{1/2}),q(N)$ to simply notations. We have 
$$-Y(\sqrt{N})+\log ( {N}^{1/2}-1)-\log 2 \overset{(d)}{=}-\sigma X_1( {N}^{1/2})+q.$$
Together with $d_N=2N^{-1/2}$, this implies
\begin{align}\label{tildeS RWB and S RWB} 
\bar{S}^N_{L,z}\left(u\right)\overset{(d)}{=} -\sigma S_{d_N^{-1} L,-\sigma^{-1}(z-d_N^{-1}Lq)}\left(d_N^{-1} u;N^{1/2}\right)+d_N^{-1} uq.
\end{align}

To apply Corollary~\ref{cor:DW} to $\bar{S}^N_{L,z}$, we use a tilting trick and identify the random walk bridge $\bar{S}^N_{L,z}$ with another bridge. For two random variables $X$ and $X'$ with density functions $f_X$ and $f_{X'}$ separately, we say $X$ and $X'$ are related through tilting if there exist $t\in\mathbb{R}$ and a positive constant $C$ such that
\begin{align*}
f_X(x) = Ce^{tx}f_{X'}(x).
\end{align*}
We need the following result.
\begin{lemma}\label{lem:tilting}
Suppose $X$ and $X'$ are related through tilting. Then the random walk bridges, constructed by the distribution of $X$ and $X'$ separately, have the same distribution.
\end{lemma}
\begin{proof}
This lemma could be proved by directly comparing the densities of the two bridges.
\end{proof}

Choose $\xi$ and $\mu$ (depending on $N$) such that
\begin{align*}
m(\xi)=&m(\sqrt{N})+q({N}),\\
\mu=&\sigma(\xi)/\sigma(\sqrt{N}).
\end{align*}
Through a direct verification, $X_1( {N}^{1/2})-\sigma^{-1} q$ and $\mu X_1(\xi)$ are related through tilting. From Lemma~\ref{lem:tilting} and \eqref{tildeS RWB and S RWB}, it holds that 
\begin{align*}
\bar{S}^N_{L,z}(u)\;
\overset{(d)}{=}\; -\sigma\mu S_{d_N^{-1} L ,-\sigma^{-1}\mu^{-1}z }\left(d_N^{-1}u ;\xi\right).
\end{align*}

Let $n=d_N^{-1} L$ and $u'=d_N^{-1} u$. By the scale invariance for Brownian bridges $ {L}^{1/2} B_1(L^{-1}u) \overset{(d)}{=} B_L(u)$, we deduce that
\begin{align*}
\mu\sigma\left( n^{1/2} B_1(n^{-1} u')+ n^{-1}\mu^{-1}\sigma^{-1}z u' +S_{n,-\mu^{-1}\sigma^{-1}z}(u';\xi) \right)
\end{align*}
has the same distribution as
\begin{align*}
2^{-1/2}\mu\sigma N^{1/4}B_L(u)+L^{-1}z {u} -\bar{S}^N_{L,z}(u).
\end{align*}

In order to apply Lemma~\ref{cor:DW}, we compute the asymptotics for $\xi$ and $\mu$. First by \cite[(8.10)]{DW}, as $N$ goes to infinity, we have
\begin{align*}
m( {N}^{1/2})=& 2^{-1}\log {N}+O(N^{-1/2}),\ \sigma( {N}^{1/2})= N^{-1/4}+O(N^{-3/4}),\ q(N)= -\log 2+O(N^{-1/2}).
\end{align*}
Therefore,
\begin{align*}
\xi=&2^{-1}N^{1/2}+O(1),\ \mu= 2^{1/2}+O(N^{-1/2}).
\end{align*}
Since $\xi$ goes to infinity, \eqref{KMT_DW} applies and we obtain the following estimate. For any $b>0$, there exist $0<C,a,\alpha'<\infty$ such that for every $N\geq 2$ and $L>0$
\begin{equation}\label{KMT_badcoef}
\begin{split}
&\mathbb{P}\left( \sup_{0\leq u\leq L}\left| 2^{-1/2}\mu\sigma N^{1/4}B_L(u)+L^{-1}z {u} -\bar{S}^N_{L,z}(u). \right|\geq r \mu\sigma\log(d_N^{-1} L ) \right)\\
\leq &C ( d_N^{-1}L)^{\alpha'-ar} e^{2\mu^{-2}\sigma^{-2}N^{-1/2}b z^2/L}.
\end{split}
\end{equation}

We are ready to verify Assumption A4. Fix $b_1,b_2>0$. From the asymptotic of $\sigma$ and $\mu$, we choose $b>0$ such that
\begin{align*}
\sup_{N\geq 2} 2\mu^{-2}\sigma^{-2}N^{-1/2}b= b_2. 
\end{align*}
Let $C,a$ and $\alpha'$ be determined through Corollary \ref{cor:DW} with $\gamma_0=\inf_{N}\xi.$ Take $r>0$ such that $\alpha'-ar=-b_1$ and then take $a_1$ such that 
\begin{align*}
a_1 = 2^{3/2}r\, \sup_{N}  N^{-1/4} \mu(N)\sigma(N^{1/2}).
\end{align*} 
Set $a_2=\{8, 2C\}$. With constants above, we can rewrite \eqref{KMT_badcoef} as
\begin{align}\label{equ:last1}
&\mathbb{P}\left( \sup_{0\leq u\leq L}\left| 2^{-1/2}\mu\sigma N^{1/4}B_L(u)+L^{-1}z {u} -\bar{S}^N_{L,z}(u). \right|\geq r \mu\sigma\log(d_N^{-1} L ) \right) \leq   2^{-1}  a_2 (d_N^{-1}L)^{-b_1}e^{b_2z^2/L}.
\end{align}

The last step is to change to coefficient of $B_L(u)$ to $1$. From the asymptotic of $\mu$ and $\sigma$, there exists a constant $C_0$ such that for all $N$,
\begin{align*}
|2^{-1/2}\mu\sigma N^{1/4}-1|\leq C_0 N^{-1/2}.
\end{align*}
From Lemma~\ref{Bbridge_sup} and $d_N=2N^{-1/2}$, 
\begin{align}\label{equ:last2}
&\mathbb{P}\left( \sup_{0\leq u\leq L}\left|C_0N^{-1/2}B_L(u) \right|\geq 2^{-1} {a_1}   d_N^{1/2} \log(d_N^{-1}L ) \right) \leq 4(d_N^{-1}L)^{-4^{-1} a_1^2 C_0^{-2} N^{1/2}  L^{-1} }.
\end{align}
By taking $N$ large enough depending on $L$ and $b_1$, we can arrange
\begin{align}\label{equ:last3}
&4(d_N^{-1}L)^{-4^{-1} a_1^2 C_0^{-2} N^{1/2}  L^{-1} } \leq 2^{-1} a_2 ( d_N^{-1} L)^{-b_1}.
\end{align}
Combining \eqref{equ:last1}, \eqref{equ:last2} and \eqref{equ:last3}, we conclude that
\begin{equation*}
\begin{split}
&\mathbb{P}\left( \sup_{0\leq u\leq L}\left| B_L(u)+ L^{-1}z {u} -\bar{S}^N_{L,z}(u) \right|\geq a_1 d_N^{1/2} \log(d_N^{-1} L ) \right)\\
\leq &\mathbb{P}\left( \sup_{0\leq u\leq L}\left| 2^{-1/2}\mu\sigma N^{1/4} B_L(u)+ L^{-1}z {u} -\bar{S}^N_{L,z}(u) \right|\geq 2^{-1}a_1 d_N^{1/2} \log(d_N^{-1} L )  \right)\\
&+ \mathbb{P}\left( \sup_{0\leq u\leq L}\left|C_0N^{-1/2}B_L(u) \right|\geq 2^{-1} a_1 d_N^{1/2} \log(d_N^{-1} L )\right)\\
\leq & a_2 ( d_N^{-1} L)^{-b_1} e^{b_2z^2/L}.
\end{split}
\end{equation*}
This finishes the verification of Assumption A4. The proof of Theorem \ref{thm:loggamma} is completed.

\begin{appendix}
\section{Random walk bridges}\label{sec:rwb}
\begin{lemma}\label{lem:fcontinue}
For any $k\in\mathbb{N}$, the $f_k(x)$ defined in \eqref{def:fk} is a positive continuous function.
\end{lemma}
\begin{proof}
We prove inductively on $k$. For $k=1$, $f_1(x)=\exp\left( -\Hrw (x) \right)$ is clearly positive and continuous. Assume $k\geq 2$ and $f_{k-1}(x)$ is positive and continuous. The positivity of $f_k(x)$ is easy to deduced from \eqref{def:fk} and we focus on the continuity.\\

Because $f_1(x)$ is uniformly continuous on compact sets, for any $M>0$, the function
\begin{align*}
f_{k,M}(x):= \int_{-M}^M f_1(x-y) f_{k-1}(y)\, dy 
\end{align*}
is continuous. Because $f_{k-1}(x)$ is integrable on $\mathbb{R}$, for any $\varepsilon>0$, there exists $M>0$ such that
\begin{align*}
\int_{|x|>M} f_{k-1}(x)\, dx\leq  \varepsilon.
\end{align*}
This implies $0\leq f_k(x)-f_{k,M}(x)\leq \varepsilon \max_{y\in\mathbb{R}} f_1(y)$. The continuity of $f_k(x)$ then follows.
\end{proof}

\begin{lemma}\label{lem:RW_middle}
Fix $n\in\mathbb{N}$ and $k\in [0,n]_{\mathbb{N}}$. Recall that $S_{n,z}(k)$ is the random walk bridge defined in \eqref{equ:1225716}. Let  $(\Omega_{\textup{mid}},\mathcal{B}_{\textup{mid}},\mathbb{P}_{\textup{mid}})$ be $(0,1)$ equipped with the standard Borel sigma algebra and the Lebesgue measure. There exists a continuous function $G:\mathbb{R}\times (0,1)\to\mathbb{R}$ such that $G(z,\cdot)$ has the same distribution as $S_{n,z}(k)$ under $\mathbb{P}_{\textup{mid}}$.
\end{lemma}
\begin{proof}
The cases $k=0$ and $k=n$ are simple because $S_n(k)$ is deterministic. In the rest of the proof, we assume $0<k<n$. In term of the functions $f_k$ defined in \eqref{def:fk}, the cumulative distribution function of $S_{n,z}(k)$ is given by $F(z,x)=f_n^{-1}(z)\int_{-\infty}^x f_k(y)f_{n-k}(z-y)\, dy.$ In view of Lemma~\ref{lem:fcontinue}, $F(z,x)$ is a continuous function on $\mathbb{R}^2$ and strictly increasing in $x$. Let $G(z,v)$ be the function defined on $\mathbb{R}\times (0,1)$ such that $F(z,G(z,v))=v$. It is straightforward to check that for all $z\in\mathbb{R}$, through measure $\mathbb{P}_{\textup{mid}}$, $G(z,\cdot)$ has the same distribution as $S_{n,z}(k)$. It remains to show the continuity of $G$. 

Fix $(z_0,v_0)\in \mathbb{R}\times (0,1)$. Let $x_0=G(z_0,v_0)$. For any $\varepsilon>0$, let $x_{\pm}=x_0\pm \varepsilon$ and $v_\pm=F(z_0,x_\pm)$. By the strict monotonicity of $F(z_0,\cdot)$, $v_-<v_0<v_+$. Let $\delta=2^{-1}\min\{ v_+-v_0, v_0-v_- \}$. By the continuity of $F$, there exists $\delta'>0$ such that for all $z\in [z_0-\delta,z_0+\delta]$, we have $F(z,x_+)\geq v_0+\delta$ and $F(z,x_-)\leq v_0-\delta$. It can be verified that for all $(z,v)$ with $|z-z_0|\leq \delta'$ and $|v-v_0|\leq\delta$, we have $|G(z,v)-G(z_0,v_0)|\leq \varepsilon.$ This concludes the continuity of $G$. 

\end{proof}
\begin{proof}[Proof of Lemma~\ref{lem:RW_coupling}]
We use an induction argument in $n$. For $n=1$, we can simply take $\Omega_1=\{\textup{pt}\}$ be a set with a single element and $\mathbf{G}_1(z,\textup{pt})(u)=uz$.  

For $n\geq 2$, assume $n=2k$ is even for simplicity. By the induction hypothesis, for $i=1,2$, we may construct probability spaces $(\Omega^{(i)}_k,\mathcal{B}^{(i)}_k ,\mathbb{P}^{(i)}_k)$, and measurable maps $\mathbf{G}^{(i)}_k:\mathbb{R}\times\Omega^{(i)}_k\to C([0,k],\mathbb{Z})$ that satisfies the assertion in the lemma. 
 
We are ready to construct the probability space $(\Omega_n,\mathcal{B}_n,\mathbb{P}_n)$. Let $\Omega_n=\Omega^{(1)}_k\times \Omega_{\textup{mid}}\times \Omega^{(2)}_k$. Let $(\Omega_{\textup{mid}},\mathcal{B}_{\textup{mid}},\mathbb{P}_{\textup{mid}})$ and $G$ be given in Lemma~\ref{lem:RW_middle}. We define $\mathcal{B}_n$ to be the product sigma algebra and $\mathbb{P}_n$ to be the product measure. Given $z\in\mathbb{R}$ and $\omega= (\omega^{1},v,\omega^{(2)})\in\Omega_n$, define
\begin{align*}
\mathbf{G}_{n}(z,\omega)(u):=\left\{ \begin{array}{cc}
\mathbf{G}_k^{(1)}( G(z,v),\omega^{(1)} )(u),  & 0\leq u\leq k,\\
G(z,v)+ \mathbf{G}_k^{(2)}(z- G(z,v),\omega^{(2)} )(u), & k< u\leq n.
\end{array}
\right.
\end{align*}
It can be checked directly that for any $z\in\mathbb{R}$, $\mathbf{G}_n(z,\omega)$, as a random function, has the same distribution as $S_{n,z}$. Fix $\omega= (\omega^{1},v,\omega^{(2)})\in\Omega_n$. From the continuity of $\mathbf{G}^{(1)}(\cdot,\omega^{(1)}), \mathbf{G}^{(2)}(\cdot,\omega^{(2)})$ and $G(\cdot,v)$, we obtain the continuity of $\mathbf{G}_{n}(\cdot,\omega)$. Finally, from \cite[Lemma 4.51]{AB}, $\mathbf{G}_n$ is a measurable function on $\mathbb{R}\times\Omega_n$.
\end{proof}
\section{Stochastic monotonicity}\label{appendix}
In this section we prove the monotone coupling lemma, Lemma~\ref{lem:monotone} by adapting the approach of \cite{CH16}. The main idea is to realize $(\Hamd,\Hrw)$-random walk bridge line ensembles as stationary measures of certain Markov process and show that such Markov process preserves order. To ensure the Markov processes converge to the stationary measures, we prefer the state space to be finite. To that end, we need to discretize the $\Hamd$-random walks.\\


Throughout this section, we assume $\Lambda_d=\mathbb{Z}$ for simplicity. We begin by proving monotone coupling for finite-state discrete random walks. Let $X$ be a discrete random variable with compact support. We define a $(\Hamd ,X)$-random walk bridge line ensemble similar to the one in Definition~\ref{def:RWB_Gibbs_ensemble} except that the free random walk bridge measures are constructed using the law of $X$. We write $\mathbb{P}^{k_1,k_2,(a,b)_{\mathbb{Z}},\vec{x},\vec{y},f,g}_{\Hamd,X}$ the law of a $(\Hamd ,X)$-random walk bridge line ensemble. 

\begin{lemma}\label{lem:monotone_discrete}
Fix $\ell\in\mathbb{N}$ and let $X$ be a random variable which takes value in $\ell  \mathbb{Z}\cap [-\ell ,\ell]$. Assume that $\mathbb{P}(X=x)>0$ for all $x\in \ell\mathbb{Z}$ with $|x|\leq \ell$ and that $-\log\mathbb{P}(X=k\ell )$ is a convex function in $k\in[-\ell^2,\ell^2  ]_{\mathbb{Z}}$. Further assume that the Hamiltonian $\Hamd$ satisfies Assumption A1. Then the following holds.

Fix arbitrary $k_1\leq k_2$ with $k_1,k_2\in\mathbb{Z}$ and $a<b$ with $a,b \in \mathbb{Z}$. For $i=1,2$, fix vectors $(\vec{x}^i,\vec{y}^i)\in \mathbb{R}^{k_2-k_1+1}$ such that each entry lies in $\ell  \mathbb{Z}\cap [-2^{-1}\ell ,2^{-1}\ell]$. Fix functions $f^i:(a,b)_{\mathbb{Z}}\to \mathbb{R}\cup \{\infty\}$, $g^i:(a,b)_{\mathbb{Z}} \to \mathbb{R}\cup \{-\infty\}$. 

Let $\mathcal{Q}^i=\{\mathcal{Q}^i_{k_1},\dots ,\mathcal{Q}^i_{k_2} \} $ be a $[k_1,k_2]_{\mathbb{Z}}\times  [a,b]_{\mathbb{Z}}$-indexed line ensemble which has the law $\mathbb{P}^{k_1,k_2, (a,b)_{\mathbb{Z}},\vec{x}^{i},\vec{y}^{i},f^i,g^i}_{\Hamd,X}$. Assume that the $i=1$ vectors and functions are pointwise greater than or equal to their $i=2$ counterparts (e.g. $f^1(u)\geq f^2(u)$ for all $u\in (a,b)_{\mathbb{Z}}$). Then there exists a coupling such that almost surely $\mathcal{Q}^1_j(u)\geq \mathcal{Q}^2_j(u)$ for all $j\in [k_1,k_2]_{\mathbb{Z}}$ and $u\in [a,b]_{\mathbb{Z}}$.
\end{lemma}

\begin{proof}
 . For all $x\in\mathbb{R}$, we denote $\Hrw(x)=-\log\mathbb{P}(X=x) $ with the convention $-\log 0=\infty$. We will construct the monotone coupling associated to the random walk with same distribution as $X$.\\ 
 
The coupling is constructed through a Markov chain argument. We first introduce the state space and the initial state. For $i=1,2$ Let $\mathfrak{X}^i$ be the collection of $\mathcal{Q}^i\in C([k_1,k_2]_{\mathbb{Z}}\times [a,b]_{\mathbb{Z}},\mathbb{R})^2$ such that 
\begin{align*}
&\mathcal{Q}^i_j(a)=x^i_j,\ \mathcal{Q}^i_j(b)=y^i_j\ \textup{for all}\ j\in [k_1,k_2]_{\mathbb{Z}},\\
&\mathcal{Q}^i_j(u+1)-\mathcal{Q}^i_j(u)\in \ell^{-1}\mathbb{Z}\cap [-\ell,\ell]\ \textup{for all} j\in [k_1,k_2]_{\mathbb{Z}},\ u\in [a,b-1]_{\mathbb{Z}}. 
\end{align*} 
The initial configuration is set to be $(\mathcal{Q}^{1}_j(u))_0=(\mathcal{Q}^{2}_j(u))_0=0$ for all $u\in (a,b)_{\mathbb{Z}}$. Because $x^i_j,y^i_j\in \ell\mathbb{Z}\cap [-2^{-1}\ell,2^{-1}\ell]$, these states are in $\mathfrak{X}^1$ and $\mathfrak{X}^2$ respectively.\\

Next, we define $\big((Q^1)_t, (Q^2)_t\big)$. For each $u \in (a,b)_{\mathbb{Z}}$,  $j\in [k_1,k_2]_{\mathbb{Z}}$ and $s\in\{+1,-1\}$, there are an independent exponential clock with rate one and an independent uniform random variables on $(0,1)$, $U^{u,j,s}$. Given $0\leq t_0<t_1$, suppose that the clock labeled $(u,j,s)$ rings at the time $t=t_0$ and there is no other clock rings in the time period $[t_0,t_1]$. We describe below how to update $((Q^1)_{t_0}, (Q^2)_{t_0})$ to $((Q^1)_t, (Q^2)_t)$ for $t\in (t_0,t_1]$.

For brevity, we write $((\mathcal{Q}^1)_{t_0},(\mathcal{Q}^2)_{t_0})$ as $(\mathcal{Q}^1 , \mathcal{Q}^2)$. We set $\tilde{\mathcal{Q}}^i$ by $\tilde{\mathcal{Q}}^i_j(u)=\mathcal{Q}^i_j(u)+\ell^{-1}s$ and other entries remain unchanged. Define 
\begin{equation*}
R_i:=\frac{W^{k_1,k_2, (a,b)_{\mathbb{Z}},\vec{x}^i,\vec{y}^i,f^i,g^i}_{\Hamd}(\tilde{\mathcal{Q}}^{i})\mathbb{P}_{\free, X}^{k_1,k_2, (a,b)_{\mathbb{Z}},\vec{x}^i,\vec{y}^i}(\tilde{\mathcal{Q}}^{i})}{W^{k_1,k_2, (a,b)_{\mathbb{Z}},\vec{x}^i,\vec{y}^i,f^i,g^i}_{\Hamd}(\mathcal{Q}^{i})\mathbb{P}_{\free, X}^{k_1,k_2, (a,b)_{\mathbb{Z}},\vec{x}^i,\vec{y}^i}(\mathcal{Q}^{i})}.
\end{equation*}
Then for $t\in (t_0,t_1]$ $(\mathcal{Q}^i)_t$ is defined to be $\tilde{\mathcal{Q}}^i$ if $ R_i\geq U^{u,j,s}$ and it remains to be $\mathcal{Q}^i=(\mathcal{Q}^i)_{t_0}$ if $R_i< U^{u,j,s}$. Notice that from Assumption A1 (1), $W^{k_1,k_2, (a,b)_{\mathbb{Z}},\vec{x}^i,\vec{y}^i,f^i,g^i}_{\Hamd}( {\mathcal{Q}} )>0$ for any $ {\mathcal{Q}}\in C([k_1,k_2]_{\mathbb{Z}}\times [a,b]_{\mathbb{Z}},\mathbb{R}).$ With the assumption on $X$, we get $R_i>0$ if an only if $\tilde{Q}^i\in\mathfrak{X}^i$. Therefore, the above procedure defines a Markov chain $(\mathcal{Q}^i)_t$ on $\mathfrak{X}^i$. Moreover, it is straightforward to check that  $\mathbb{P}^{k_1,k_2, (a,b)_{\mathbb{Z}},\vec{x}^{i},\vec{y}^{i},f^i,g^i}_{\Hamd,X}$ are stationary measures.\\

Two Markov chains $(\mathcal{Q}^{i})_t$, $i=1,2$, are coupled through the same collection of clocks and uniform random variables $U^{u,j,s}$. We now show that this dynamics preserves the ordering. That is, for any $t\geq 0$,
$(\mathcal{Q}^{1}_j(u))_t\geq (\mathcal{Q}^{2}_j(u))_t$. Notice that at each update step, we could only change $\mathcal{Q}^{i}_j(u )$ to $\tilde{\mathcal{Q}}^{i}_j(u )=\mathcal{Q}^{i}_j(u )+\delta$ or  $\tilde{\mathcal{Q}}^{i}_j(u )=\mathcal{Q}^{i}_j(u )-\delta$. Hence there are only two cases that the ordering could possibly be violated. The first case is when the clock $(u,j,+1)$ rings and $\mathcal{Q}^{1}_j(u )=\mathcal{Q}^{2}_j(u )=z$, then the ordering will not be hold if $R_1< U^{u, j,+1}\leq R_2$. We will prove that we always have $R_1\geq R_2$. Therefore, the ordering is preserved. From the definition of $R_i$, one obtains 
\begin{align*}
R_i=R_{i,\textup{RW}}\cdot R_{i,\Hamd},
\end{align*}
where
\begin{align*}
R_{i,\textup{RW}}=&\frac{\exp\Big(-\Hrw\big(z+\ell^{-1} -\mathcal{Q}^{i}_j(u-1)\big)-\Hrw\big(\mathcal{Q}^{i}_j(u+1)-z-\ell^{-1}\big)\Big)}{\exp\Big(-\Hrw\big(z-\mathcal{Q}^{i}_j(u-1)\big)-\Hrw\big(\mathcal{Q}^{i}_j(u+1)-z\big)\Big)},
\end{align*}
and
\begin{align*}
R_{1,\Hamd}=&\prod_{k=j-1}^j\prod_{v=u-1}^{u+1} \exp\Big(-\Hamd\big(\rectangle(\tilde{\mathcal{Q}}^i,k,v)\big)+\Hamd\big(\rectangle(\mathcal{Q}^i,k,v)\big)\Big).
\end{align*}
Since $\Hrw$ is convex and $\mathcal{Q}^1\geq \mathcal{Q}^2$,  we have
\begin{align*}
&\quad-\Hrw(z+\ell^{-1}-\mathcal{Q}^1_j(u-1))+\Hrw(z-\mathcal{Q}^1_j(u-1))\\
&\geq -\Hrw(z+\ell^{-1}-\mathcal{Q}^2_j(u-1))+\Hrw(z-\mathcal{Q}^2_j(u-1))
\end{align*}
and
\begin{align*}
&\quad-\Hrw(\mathcal{Q}^1_j(u+1)-z-\ell^{-1})+\Hrw(\mathcal{Q}^1_j(u+1)-z)\\
&\geq -\Hrw(\mathcal{Q}^2_j(u+1)-z-\ell^{-1})+\Hrw(\mathcal{Q}^2_j(u+1)-z).
\end{align*}
This yields $R_{1,\textup{RW} }\geq R_{2,\textup{RW}}$. By Assumption A2 and $\mathcal{Q}^1\geq \mathcal{Q}^2$, we have each individual term in $R_{1,\Hamd}$ is greater than or equal to the one in $R_{2,\Hamd}$. For instance, take
\begin{align*}
\vec{a}=\rectangle({\mathcal{Q}}^1,j,u ), \quad \vec{a'}=\rectangle(\tilde{\mathcal{Q}}^1,j,u ),
\end{align*} 
and
\begin{align*}
\vec{b}=\rectangle({\mathcal{Q}}^2,j,u ), \quad \vec{b'}=\rectangle(\tilde{\mathcal{Q}}^2,j,u ).
\end{align*} 
Then Assumption A1 (2) implies
\begin{align*}
&{  -\Hamd\big(\rectangle(\tilde{\mathcal{Q}}^1,j,u_m)\big)+\Hamd\big(\rectangle(\mathcal{Q}^1,j,u_m)\big) } \geq {  -\Hamd\big(\rectangle(\tilde{\mathcal{Q}}^2,j,u_m)\big)+\Hamd\big(\rectangle(\mathcal{Q}^2,j,u_m)\big) }.
\end{align*}
The other five terms are dealt with similarly. This implies $R_1\geq R_2$. In short the ordering preserved after this update. A similar argument applies when $(u_m,j,-)$ rings. As a result, $(\mathcal{Q}^{1}_j )_t\geq (\mathcal{Q}^{2}_j)_t$ for all $t\geq 0$. \\

Now we show that the above Markov chains are irreducible. We drop the index $i$ for simplicity. Assume $k_1=k_2$. Let $x$ and $y$ be the left and right end points. We apply indunction on $b-a$. If $b-a=1$, the assertion is obvious. Assume $b-a\geq 2 $ and let $u=a+1$. Fix a commutation class and let $\mathcal{Q}_{\max}$ and $\mathcal{Q}_{\min}$ be states in such class for which $\mathcal{Q}_{\max}(u)$ and $\mathcal{Q}_{\min}(u)$ attains maximum and minimum. From $x,y\in \ell^{-1} \mathbb{Z}\cap [-2^{-1}\ell ,2^{-1}\ell ]$, we have $\mathcal{Q}_{\max}(u)= x+\ell $. The reason is that $\mathcal{Q}_{\max}(u)< x+\ell$ implies $\mathcal{Q}_{\max}(u+k)=\mathcal{Q}_{\max}(u)-k\ell $ for all $k\in [1,b-a-1]_{\mathbb{Z}}$, which is impossible. Similarly, $\mathcal{Q}_{\min}(u_1)= x-M$. This implies that in any communication class, there exists a state $\mathcal{Q}_{\textup{mid}}$ such that $\mathcal{Q}_{\textup{mid}}(u)=0$. From induction hypothesis, the Markov chain is irreducible. Then we can apply induction on $k_2-k_1$ to prove the general case.\\


Finally, $(\mathcal{Q}^i)_t$ converges to the stationary measures $\mathbb{P}^{k_1,k_2, (a,b)_{\mathbb{Z}},\vec{x}^{i},\vec{y}^{i},f^i,g^i}_{\Hamd,X}$. Notice that the ordered relation defines a closed set on $C([k_1,k_2]_{\mathbb{Z}}\times [a,b]_{\mathbb{Z}},\mathbb{R})^2$. Therefore it is preserved under weak convergence. This gives the desired coupling.
\end{proof}

To prove Lemma~\ref{lem:monotone}, we approximate the $\Hrw$-random walks by random walks which satisfy the assumptions in Lemma \ref{lem:monotone_discrete}. For any $\ell\in\mathbb{N}$, we define a discrete random variable $X^\ell$ by
\begin{align*}
\mathbb{P}(X^{\ell}=k\ell^{-1})=\left\{
\begin{array}{cc}
C_{\ell}^{-1}\exp(-\Hrw(k\ell^{-1})) & k\in [-\ell^2,\ell^2]_{\mathbb{Z}},\\
0 & \textup{others},
\end{array}
\right. 
\end{align*}
where $C_{\ell}=\sum_{k=-\ell^2}^{\ell^2} \exp(-\Hrw(k\ell^{-1}))$ is the normalizing constant.

Let us recall the notations in Lemma~\ref{lem:monotone}. Fix  $k_1\leq k_2$ with $k_1,k_2\in\mathbb{Z}$ and $a < b$ with $a,b\in\mathbb{Z}$. For $i=1,2$, fix vectors $(\vec{x}^i,\vec{y}^i)\in \mathbb{R}^{k_2-k_1+1}$, and functions $f^i: (a,b)_{\mathbb{Z}}\to \mathbb{R}\cup \{\infty\}$, $g^i:(a,b)_{\mathbb{Z}}\to \mathbb{R}\cup \{-\infty\}$. 
Let $\mathcal{Q}^i=(Q^i_{k_1},\dots, Q^i_{k_2}) $ be the $[k_1,k_2]_{\mathbb{Z}}\times  [a,b]_{\mathbb{Z}}$-indexed line ensemble with laws $\mathbb{P}^{k_1,k_2, (a,b)_{\mathbb{Z}},\vec{x}^i,\vec{y}^i,f^i,g^i}_{\Hamd,\Hrw}$. Moreover, the $i=1$ vectors and functions are pointwise greater than or equal to their $i=2$ counterparts. 

We take boundary vectors $\vec{x}^{i,\ell},\vec{y}^{i,\ell}\in (\ell^{-1}\mathbb{Z})^{k_2-k_1+1}$ such that $ {x}_j^{1,\ell}\geq {x}_j^{2,\ell}$, $ {y}_j^{1,\ell}\geq {y}_j^{2,\ell}$, $\vec{x}^{i,\ell}$ converges to $ \vec{x}^{i}$ and $\vec{y}^{i,\ell}$ converges to $\vec{y}^{i}$. Let $\mathcal{Q}^{i,\ell}_j(u)$ be the line ensembles with laws $\mathbb{P}^{k_1,k_2,(a,b)_{\mathbb{Z}} ,\vec{x}^{i,\ell},\vec{y}^{i,\ell},f^i,g^i}_{\Hamd,X^{\ell}}$.
\begin{lemma}\label{lem:Qconvergence}
Under the assumptions in Lemma~\ref{lem:monotone}, $\mathcal{Q}^{i,\ell}$ converges weakly to $\mathcal{Q}^{i}$ as $\ell$ goes to infinity.
\end{lemma}

With Lemma~\ref{lem:Qconvergence}, we are ready to prove Lemma~\ref{lem:monotone}.
\begin{proof}[Proof of Lemma~\ref{lem:monotone}]
From Assumption A2 and Lemma~\ref{lem:monotone_discrete}, $\mathcal{Q}^{1,\ell}$ and $\mathcal{Q}^{2,\ell}$ can be coupled together such that $\mathcal{Q}_j^{1,\ell}(u)\geq \mathcal{Q}_j^{2,\ell}(u)$ for all $j\in [k_1,k_2]_{\mathbb{Z}}$ and $u\in [a,b]_{\mathbb{Z}}$. From Lemma~\ref{lem:Qconvergence}, the coupling of $\mathcal{Q}^{1}$ and $\mathcal{Q}^{2}$ can be obtained by letting $\ell$ go to infinity. 
\end{proof}
Lastly, we prove Lemma~\ref{lem:Qconvergence}
\begin{proof}[Proof of Lemma~\ref{lem:Qconvergence}]
Let $m=b-a$ and $f(x)=\exp(-\Hrw(x))$. We start by showing that $\ell^{-1}C_{\ell}$ converges to $1$ as $\ell$ goes to infinity. Because $\Hrw$ is convex, there exists $M_0>0$ such that $f(x)$ is non-increasing for $x\geq M_0$ and is non-decreasing for $x\leq -M_0$. We compute
\begin{align*}
\ell^{-1}C_{\ell}=&\frac{1}{\ell}\sum_{x\in [-\ell,\ell]\cap\ell^{-1}\mathbb{Z}} f(x)\\
                 =& \underbrace{\frac{1}{\ell}\sum_{x\in [-M,M]\cap\ell^{-1}\mathbb{Z}} f(x)}_{\textsc{I}}+\underbrace{\frac{1}{\ell} \sum_{x\in( [-\ell,-M)\cup(M,\ell])\cap\ell^{-1}\mathbb{Z}} f(x)}_{\textsc{II}}.
\end{align*}
Since $f(x)$ is monotone for $|x|\geq M_0$, for any $M\geq M_0+1$ and $\ell\in\mathbb{N}$ we have
\begin{align*}
\frac{1}{\ell} \sum_{x\in    (M,\ell] \cap\ell^{-1}\mathbb{Z}} f(x)\leq \int_{M-1}^\infty f(x)\, dx.
\end{align*}
Therefore, the second term $\textsc{II}$ converges uniform in $\ell$ to zero as $M$ goes to infinity. For any fixed $M>0$, $\textsc{I}$ converges to the integral of $f(x)$ in $[-M,M]$. As a result, $\ell^{-1}C_{\ell}$ converges to the total integral of $f(x)$ which is $1$.\\

Next, we show that the random walk bridges of $X^\ell$ converges weakly to the random walk bridge of $\Hrw$. L	et $S^{\ell}_m:=X^{\ell}_1+X^{\ell}_2+\dots +X^{\ell}_m$. Here $X^{\ell}_1,X^{\ell}_2,\dots,X^{\ell}_m$ are i.i.d. random variables with the same law as $X^{\ell}$. We claim that for any $z\in\mathbb{R}$ and any sequence $z^\ell$ in$\ell^{-1}\mathbb{Z}$ with $z^{\ell} $ converging to $z$, it holds that
\begin{align}\label{aaaa}
\lim_{\ell\to\infty} \ell\mathbb{P}(S^\ell_m=z^{\ell})= \prod_{j=1}^{m-1} \int_{x_j
\in\mathbb{R} }dx_j \prod_{j=1}^{m} f(x_j-x_{j-1}).
\end{align}
In \eqref{aaaa} we used the convention $x_0=0$ and $x_m=z$. From the definition,
\begin{align*}
\left(\ell^{-1}C_\ell \right)^m\cdot\ell \mathbb{P}(S^\ell_m=z^{\ell})=\frac{1}{\ell^{m-1} }\sum_{\substack{ y_k\in [-\ell,\ell]\cap\ell^{-1}\mathbb{Z}\\ k=1,2,\dots ,m-1}}  f(y_1)f(y_2)\dots f(y_{m-1})f(z^{\ell}-\sum_{k=1}^{m-1}y_k). 
\end{align*}
For any $m' \in [1,m]_{\mathbb{Z}}$ and $M>0$,
\begin{align*}
&\frac{1}{\ell^{m-1} }\sum_{\substack{ y_{m'}\in [-\ell,\ell]\cap\ell^{-1}\mathbb{Z}\\ |y_{m'}|>M } }\sum_{\substack{ y_k\in [-\ell,\ell]\cap\ell^{-1}\mathbb{Z}\\ k=1,2,\dots ,m-1, k\neq m'}}  f(y_1)f(y_2)\dots f(y_{m-1})f(z^{\ell}-\sum_{k=1}^{m-1}y_k)\\
\leq &\max_{x\in \mathbb{R}}f(x) \cdot  \frac{1}{\ell^{m-1} }\sum_{\substack{ y_{m'}\in [-\ell,\ell]\cap\ell^{-1}\mathbb{Z}\\ |y_{m'}|>M } }\sum_{\substack{ y_k\in [-\ell,\ell]\cap\ell^{-1}\mathbb{Z}\\ k=1,2,\dots ,m-1, k\neq m'}}  f(y_1)f(y_2)\dots f(y_{m-1})\\
=&\max_{x\in \mathbb{R}}f(x) \cdot (\ell^{-1}C_{\ell})^{m-2}\left( \frac{1}{\ell} \sum_{y\in( [-\ell,-M)\cup(M,\ell])\cap\ell^{-1}\mathbb{Z}} f(y)\right),
\end{align*}
which converges uniformly in $\ell$ to zero as $M$ goes to infinity. Therefore,
\begin{align*}
\ell \mathbb{P}(S^\ell_m=z^{\ell})=\left(\ell^{-1}C_\ell \right)^{-m}\cdot \frac{1}{\ell^{m-1}}\sum_{\substack{ y_k\in [-M,M]\cap\ell^{-1}\mathbb{Z}\\ k=1,2,\dots ,m-1}}  f(y_1)f(y_2)\dots f(y_{m-1})f(z^{\ell}-\sum_{k=1}^{m-1}y_k)+o(1),
\end{align*}
and \eqref{aaaa} follows. By a similar argument, we have for any $a_1,a_2,\dots, a_{m-1}\in\mathbb{R}$,

\begin{align*}
\ell \mathbb{P}(S^\ell_1\leq a_1,S^\ell_2\leq a_2,\dots,S^\ell_{m-1}\leq a_{m-1},S^\ell_m=z^\ell)
\end{align*}
converges to
\begin{align*}
\prod_{j=1}^{m-1} \int_{x_j\leq a_j}dx_j \prod_{j=1}^{m} f(x_j-x_{j-1}),
\end{align*}
where we use the convention $x_0=0$ and $x_m=z$. As a result, the free bridges of $X^\ell$ converge to the free bridges of $\Hrw$.\\

Lastly, let $\mathbb{E}^i$ be the expectation of $\mathbb{P}^{k_1,k_2, (a,b)_{\mathbb{Z}},\vec{x}^i,\vec{y}^i,f^i,g^i}_{\free,\Hrw}$ and $\mathbb{E}^{i,\ell}$ be the expectation of $\mathbb{P}^{k_1,k_2, (a,b)_{\mathbb{Z}},\vec{x}^{i,\ell},\vec{y}^{i,\ell},f^i,g^i}_{\free,\Hrw}$. These are free random walk bridges defined in Definition~\ref{def:RW_ensemble}. Denote the normalizing constants
\begin{align*}
Z^{i,\ell}:=&\mathbb{E}^{i,\ell} [ W_{\Hamd}^{k_1,k_2, (a,b)_{\mathbb{Z}},\vec{x}^{i,\ell},\vec{y}^{i,\ell},f^i,g^i}],\\
Z^{i }:=&\mathbb{E}^{i } [ W_{\Hamd}^{k_1,k_2,(a,b)_{\mathbb{Z}},\vec{x}^{i },\vec{y}^{i} ,f^i,g^i}].
\end{align*} 
We remark that as functions from $C([k_1,k_2]_{\mathbb{Z}}\times [a,b]_{\mathbb{Z}},\to\mathbb{R})$ to $\mathbb{R}$, $W_{\Hamd}^{k_1,k_2,\Lambda_d(a,b),\vec{x}^{i,\ell},\vec{y}^{i,\ell},f^i,g^i}$ and $W_{\Hamd}^{k_1,k_2,\Lambda_d(a,b),\vec{x}^{i },\vec{y}^{i},f^i,g^i}$ are actually identical.
From the continuity of $\Hamd$, they are continuous. Together with the convergence of free bridges proved above, we get 
\begin{align*}
\lim_{\ell\to\infty}Z^{i,\ell}=Z^{i}.
\end{align*}
Similarly, for any other bounded continuous functional $F$ from $C([k_1,k_2]_{\mathbb{Z}}\times [a,b]_{\mathbb{Z}},\to\mathbb{R})$ to $\mathbb{R}$, it holds that
\begin{align*}
\lim_{\ell\to\infty}\mathbb{E}^{i,\ell} [F\times W_{\Hamd}^{k_1,k_2,\Lambda_d(a,b),\vec{x}^{i,\ell},\vec{y}^{i,\ell},f^i,g^i}] =\mathbb{E}^{i} [F\times W_{\Hamd}^{k_1,k_2,\Lambda_d(a,b),\vec{x}^i,\vec{y}^i,f^i,g^i}]. 
\end{align*}
This implies  
\begin{align*}
\mathbb{E}^{i,\ell}[F(\mathcal{Q}^{i,\ell})]=\frac{1}{Z^{i,\ell}}\mathbb{E}^{i,\ell} [F\times W_{\Hamd}^{k_1,k_2,\Lambda_d(a,b),\vec{x}^{i,\ell},\vec{y}^{i,\ell},f^i,g^i}]
\end{align*}
converges to
\begin{align*}
\mathbb{E}^{i }[F(\mathcal{Q}^{i })]=\frac{1}{Z^{i }}\mathbb{E}^{i } [F\times W_{\Hamd}^{k_1,k_2,\Lambda_d(a,b),\vec{x}^{i },\vec{y}^{i },f^i,g^i}].
\end{align*}
Thus $\mathcal{Q}^{i,\ell}$ converges weakly to $\mathcal{Q}^{i}$.
\end{proof}

\end{appendix}

\end{document}